\documentclass{amsart}

\usepackage{graphicx}
\usepackage{stmaryrd, amssymb}
\usepackage{graphics}
\newcommand{\abs}[1]{\vert#1\vert}

\newcommand{\dsp}{\displaystyle}
\newcommand{\bU}{{\mathbf U}}

\newcommand{\ovV}{\overline{V}}
\newcommand{\ovv}{\overline{v}}
\newcommand{\uU}{\underline{U}}

\newcommand{\uV}{{\underline V}}

\newcommand{\R}{{\mathbb R}}

\newcommand{\dt}{\partial_t}
\newcommand{\dz}{\partial_z}
\newcommand{\dx}{\partial_x}
\newcommand{\dy}{\partial_y}

\newcommand{\curl}{\mbox{\textnormal{curl} }}

\newcommand{\eps}{\varepsilon}

\newcommand{\Rey}{{\mathbf R}}

\newcommand{\bom}{\boldsymbol{\omega}}

\newtheorem{prop}{Proposition}

\newtheorem{lemma}{Lemma}

\theoremstyle{remark}
\newtheorem{remark}{Remark}

\begin{document}

\title{Modeling shallow water waves}
\author{D. Lannes}
\address{David Lannes, Universit\'e de Bordeaux et CNRS, 351 Cours de la Lib\'eration, 33405 Talence Cedex, France}
\thanks{D. L. is supported by the Fondation Del Duca de l'Acad\'emie des Sciences, the ANR grants ANR-17-
CE40-0025 NABUCO and ANR-18-CE40-0027-01 Singflows, and the Conseil R\'egional d'Aquitaine}
\maketitle

\begin{abstract}
We review here the derivation of many of the most important models that appear in the literature (mainly in coastal oceanography) for the description of waves in shallow water. We show that these models can be obtained using various asymptotic expansions of the "turbulent" and non-hydrostatic terms that appear in the equations that result from the vertical integration of the free surface Euler equations. Among these models are the well-known nonlinear shallow water (NSW), Boussinesq and Serre-Green-Naghdi (SGN) equations for which we review several pending open problems. More recent models such as the multi-layer NSW or SGN systems, as well as the Isobe-Kakinuma equations are also reviewed under a unified formalism that should simplify comparisons. We also comment on the scalar versions of the various shallow water systems which  can be used to describe unidirectional waves in horizontal dimension $d=1$; among them are the KdV, BBM,  Camassa-Holm and  Whitham equations. Finally, we show how to take vorticity effects into account in shallow water modeling, with specific focus on the behavior of the turbulent terms. As examples of challenges that go beyond the present scope of mathematical justification, we review recent works using shallow water models with vorticity to describe wave breaking, and also derive models for the propagation of shallow water waves over strong currents.
\end{abstract}
\section{General introduction}


\begin{figure}
\begin{center}
\includegraphics[width=0.7\linewidth]{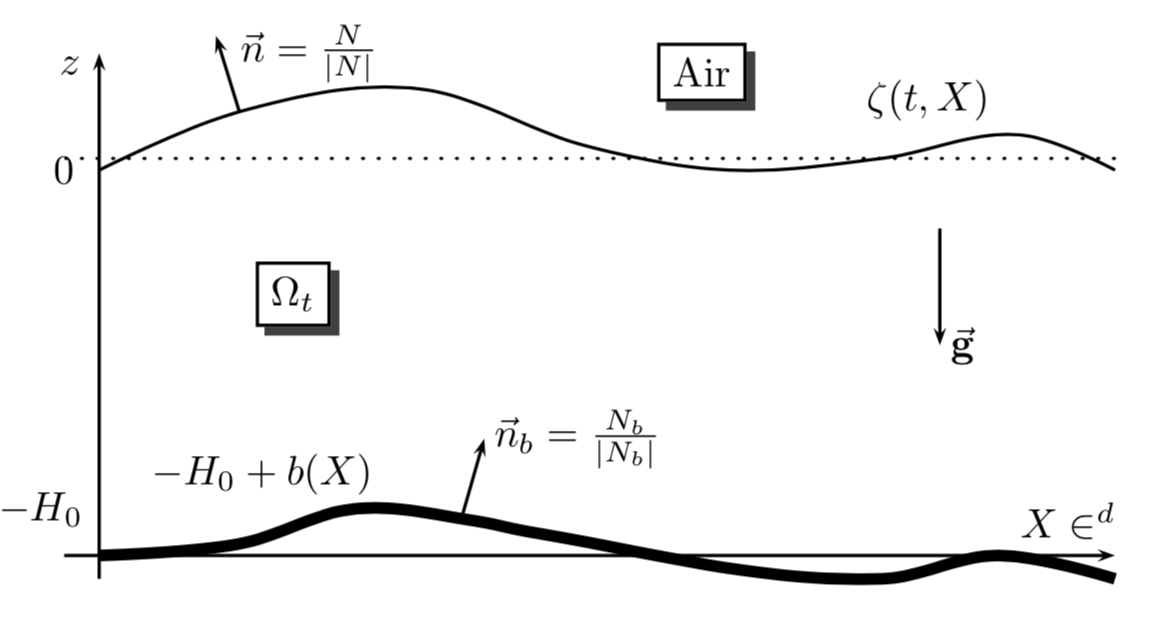}
\end{center}
\caption{Main notations.}
\end{figure}

\subsection{Brief overview}

The goal of this article is to review several models that have been derived for the modeling of shallow water flows, with a specific focus on those of interest for applications to coastal oceanography. Most of these models, such as the Nonlinear Shallow Water (NSW), the Boussinesq and the Serre-Green-Naghdi (SGN) equations, have been derived many years ago, and their full justification as approximations of the water waves equations is also well established (see for instance \cite{Lannes2013} and references therein). However, many mathematical questions related to these models remain open and some of them are of great physical relevance. To mention just a few, the formation of singularities (related to the well known phenomenon of wave breaking for instance) or the mathematical understanding of several boundary conditions (e.g. wall, generating, transparent) for nonlinear dispersive models are real issues encountered by oceanographers and computers scientists when they develop operational numerical wave models. We tried in this paper to present these open mathematical questions in their physical context, with the hope to spark new mathematical studies and advances on these physically motivated issues. 

Also included in this review are less standard and more recent models, such as multi-layer shallow water, Boussinesq and Serre-Green-Naghdi equations, the Isobe-Kakinuma model and rotational shallow water models. For these systems of equations, many mathematical questions remain open. A careful and unified derivation of these less known models is proposed here, which should allow one to compare them one with another. In the process, we also propose some new sets of equations that do not seem to have been studied before. We also present an application of rotational shallow water models to the modeling of wave breaking; a full justification of such models is out of reach since wave breaking is an extremely complex phenomenon, but the ability of equations based on the shallow water system have proved surprisingly efficient to account correctly for the wave breaking phenomenon.

As already said, we restricted our attention here to shallow water modeling for coastal flows, but such models also occur in many other contexts: geophysical flows at larger  scales with Coriolis effects (rotating fluids, see for instance \cite{zeitlin2007,gallagher}), granular flows and debris flows (e.g. \cite{bouchut2015}), internal waves (e.g. \cite{Craig:2005zl,Bona:2008ff,duchene2011} and the extensive review \cite{saut2013asymptotic}) and, more recently, the interaction between waves and floating or partially immersed objects (\cite{Lannes2017,bosi2019spectral,bocchi2019return,godlewski2018congested,bresch2019waves}). These are not treated in the present paper.

\subsection{Organization of the paper}

Starting from the Euler equations with a free surface, we derive in Section \ref{sectbasic} several formulations of the water waves equations, and write them in dimensionless form. When written in elevation-discharge formulation, the water waves equations take the form of the nonlinear shallow water equations with two additional terms:  a "turbulent" term and a term accounting for the non-hydrostatic effects of the pressure. The derivation of approximate models to the water waves equations is done by an asymptotic analysis of these two terms in the shallow water regime, that is, when the depth is much smaller than the typical horizontal scale.

This asymptotic analysis is performed in Section \ref{sect2}. We first describe the inner structure of the velocity and pressure fields in \S \ref{sectinner} and \S \ref{sectNH} respectively. We then show in \S \ref{sectNSW} that at leading order, the turbulent and non-hydrostatic terms can be neglected so that the behavior of the waves is described at leading order by the nonlinear shallow water (NSW) equations. In \S \ref{sectBoussinesq}, we work with a higher precision, but make a smallness assumption on the size of the waves (the so called weak nonlinearity assumption); we show that some non-hydrostatic terms, responsible for dispersive effects, must be kept, leading to the Boussinesq systems. Removing the smallness assumption, we obtain in \S \ref{sectSGN} the more complicated Serre-Green-Naghdi (SGN) model. For all these different models, we review known mathematical results and mention several open problems. Finally, we describe the multi-layer approach (\S \ref{sectmultilayer}) and the Isobe-Kakinuma model (\S \ref{sectIsobe}) which have been derived to have a better resolution of the vertical structure of the velocity and/or to improve the precision of Boussinesq or SGN models without introducing high order derivatives that are numerically difficult to implement. In dimension $d=1$ and in shallow water, perturbations of the surface elevation essentially split into two counter propagating waves. Under certain assumptions, it is possible to describe the behavior of one of these components independently of the other. The interest is that it is governed by a single scalar equation, much easier to analyze and from which one can therefore gain some useful insight on the wave. Such scalar models are considered in \S \ref{sectscalar}. We finally explain briefly in \S \ref{sectjustif} the procedure to rigorously justify all these models.

Finally, Section \ref{sectrot} is devoted to the derivation of shallow water models in the presence of vorticity. We first generalize in \S \ref{sectWWrot} various formulations of the water waves equations when the vorticity is non zero and introduce the notion of vorticity strength. In particular, we show that a rigorous derivation of rotational shallow water models is possible up to a certain vorticity strength. The influence of the vorticity (and more specifically of the shear velocity it induces) on the inner structure of the velocity and pressure fields is described in \S \ref{sectinnervort}. The consequences on the NSW and SGN equations is then studied in \S \ref{sectNSWSGNvort}; the main consequence is that the "turbulent" term in the elevation-discharge formulation of the water waves equations cannot be neglected anymore, and that the SGN equations must be extended with a third equation on the "turbulent tensor". This latter model can be rigorously justified. There exist also other models that have not been justified so far but that are physically relevant; we describe for instance in \S \ref{WB} a model aiming at modeling wave breaking through "enstrophy creation" and, in \S \ref{sectlargevort}, we formally derive NSW and Boussinesq models in the presence of a "strong" vorticity.

\medbreak

\noindent
{\bf Acknowledgement.} The author wants to express his gratitude to V. Duch\^ene, E. Fernandez-Nieto and J.-C. Saut for their precious comments on this work. Many thanks also to the organizers of the CEMRACS 2019 where I gave a course based on these notes; the present article owes a lot to the discussions held there with the participants.

\subsection{Notations}
Let us give here several notations that will be used throughout this paper.
\begin{itemize}
\item $d=1,2$ denotes the horizontal dimension and $X\in \R^d$ the horizontal variables; the vertical variable is denoted $z$.
\item We denote by $\nabla_{X,z}$ the $(d+1)$-dimensional gradient operator, and by $\nabla$ the $\R^d$-dimensional gradient taken with respect to the variable $X$ only. Similar conventions are used for $\Delta_{X,z}$  and $\Delta$.
\item The velocity field in the fluid domain is denoted $\bU\in \R^{d+1}$. We denote by $V\in \R^d$ and $w$ its horizontal and vertical components respectively. When $d=1$ we write $v$ instead of $V$.
\item We denote by $Q=\int_{-h_0+b}^{\zeta} V$ the horizontal discharge and by $\ovV=Q/h$ ($h=h_0+\zeta-b$) the vertically averaged horizontal velocity; in dimension $d=1$, these quantities are denoted $q$ and $\ovv$ respectively. 
\item We use the notation $f(D)$ for Fourier multipliers defined, when possible, by $\widehat{f(D)u}=\widehat{f}\widehat{u}$, the notation $\widehat{\cdot}$ standing for the Fourier transform on $\R^d$.
\end{itemize}

\section{Basic equations}\label{sectbasic}

Starting from the free surface Euler equations (\S \ref{sectFSE}), we derive two formulations of the water waves problem: the Zakharov-Craig-Sulem formulation in \S \ref{sectZCS}, which is very convenient for the mathematical analysis of the equations, and the elevation-discharge formulation in \S \ref{sectelevdisc}, whose structure is much closer to the various shallow water models derived in this paper. The dimensionless version of these formulations is then derived in \S \ref{sectnondim} and will be used throughout this paper to derive asymptotic approximations of the water waves equations in shallow water.

\subsection{The free surface Euler equations}\label{sectFSE}

Denoting by $X\in \R$  ($d=1,2$) the horizontal coordinates and by $z$ the vertical coordinate, we assume that the elevation of the surface of the water above the rest state $z=0$ is given at time $t$ by the graph of a function $\zeta(t,\cdot)$, and that the bottom is parametrized by a time independent function $-h_0+b$ ($h_0>0$ is a constant); the domain occupied by the fluid at time $t$ is therefore
$$
\Omega_t=\{ (X,z)\in \R^d\times \R, -h_0+b(X)<z <\zeta(t,X)\}.
$$
We also denote by ${\bf U}(t,X,z) \in \R^{d+1}$ the velocity of a fluid particle located at $(X,z)$ at time $t$, and by $V(t,X,z)\in \R^d$ and $w(t,X,z)$ its horizontal and vertical component respectively. For a non viscous fluid of constant density $\rho$, the balance of forces in the fluid domain is given by the Euler equations
\begin{equation}\label{Euler1}
\dt \bU +\bU\cdot \nabla_{X,z} \bU=-\frac{1}{\rho}\nabla_{X,z} P - g {\bf e}_z \quad \mbox{ in }\quad \Omega_t,
\end{equation}
where $g$ is the acceleration of gravity and ${\bf e}_z$ is the unit upwards vertical vector. Incompressibility then takes the form
\begin{equation}\label{Euler2}
\nabla_{X,z}\cdot \bU=0 \quad \mbox{ in }\quad \Omega_t,
\end{equation}
and we also assume that the flow is irrotational 
\begin{equation}\label{Euler3}
\nabla_{X,z}\times \bU=0 \quad \mbox{ in }\quad \Omega_t;
\end{equation}
we discuss in Section \ref{sectrot} how to remove this latter assumption.

In addition to the equations \eqref{Euler1}-\eqref{Euler3} which are given in the fluid domain $\Omega_t$, we need boundary conditions. Two of them are given at the surface: the first one is the so-called kinematic boundary condition and expresses the fact that fluid particles do not cross the surface
\begin{equation}\label{Euler4}
\dt \zeta - \uU\cdot N =0
\end{equation}
with the notations
$$
 \uU(t,X)={\bf U}(t,X,\zeta(t,X)) \quad \mbox{ and }\quad N=\left(\begin{array}{c} -\nabla \zeta \\ 1 \end{array}\right);
$$
the second boundary condition at the surface is the so-called dynamic boundary condition
\begin{equation}\label{Euler5}
P=P_{\rm atm}=\mbox{constant}\quad\mbox{ on }\quad \{z=\zeta(t,X)\}.
\end{equation}
\begin{remark}
The condition \eqref{Euler5} means that surface tension is neglected, which is relevant for applications to coastal oceanography where the scales involved are significantly larger than the capillary scale; see for instance \cite{Lannes2013} and references therein for generalizations including surface tension.\\
Inversely, the scales considered in coastal oceanography are in general small enough to neglect the variations of the atmospheric pressure. In some specific cases such as storms or meteotsunamis for instance, it is however relevant to consider 	a variable surface pressure \cite{Melinand}.
\end{remark}
Finally, a last boundary condition is needed at the bottom, assumed to be impermeable
\begin{equation}\label{Euler6}
U_{\rm b}\cdot N_{\rm b}=0,
\end{equation}
with the notations
$$
 U_{\rm b}(t,X)={\bf U}(t,X,-h_0+b(t,X)) \quad \mbox{ and }\quad N_{\rm b}=\left(\begin{array}{c} -\nabla b \\ 1 \end{array}\right).
$$

The question of solving equations \eqref{Euler1}-\eqref{Euler6} is a free surface problem in the sense that the equations are cast on a domain which is itself one of the unknowns (as $\Omega_t$ is determined by $\zeta(t\cdot)$). In order to solve it, it is necessary to find an equivalent formulation in which the equations are cast in a fixed domain. To mention only the local Cauchy problem, several equivalent formulations have been used: a Lagrangian formulation of the free surface in the pioneering work \cite{Nalimov} that solved the problem when $d=1$ and for small data, as well as in \cite{Wu1,Wu2} where the assumption of small data was removed and the result extended to the two dimensional case $d=2$; a variational and geometrical approach based on Arnold's remark that the motion of an inviscid incompressible fluid can be viewed as the geodesic flow on the infinite-dimensional manifold of volume-preserving diffeomorphisms \cite{Shatah}; a full Lagrangian formulation of Euler's equations \cite{Lindblad,Coutand2007}, etc. We describe below two other formulations: one is Zakharov's Hamiltonian formulation \cite{Zakharov1968} whose well-posedness was proved in \cite{Lannes2005} (see also \cite{Alazard2014} for the low regularity Cauchy problem and \cite{Alvarez,Iguchi2009} for uniform bounds in several asymptotic regimes), as well as a formulation in $(\zeta,Q)$, where $Q$ is the horizontal discharge, that proves very useful to derive and understand the mechanism at stake in shallow water asymptotic models. For other recent mathematical advances on the water waves equations, such as long time/global existence, we refer to the surveys \cite{ionescu,Delort}.

%
%
%
%

\subsection{The Zakharov-Craig-Sulem formulation}\label{sectZCS}

From the irrotationality assumption, there exists a velocity potential $\Phi$ such that $\bU=\nabla_{X,z}\Phi$. The Euler equation \eqref{Euler1} reduces therefore to the Bernoulli equation
\begin{equation}\label{Bernoulli}
\dt \Phi + \frac{1}{2} \abs{\nabla_{X,z}\Phi}^2 + gz =- \frac{P-P_{\rm atm}}{\rho}.
\end{equation}
From the incompressibility condition \eqref{Euler2} and the bottom boundary condition \eqref{Euler6}, we also know that $\Delta_{X,z}\Phi=0$ in $\Omega_t$ and that $N_b\cdot \nabla_{X,z}\Phi=0$ at the bottom. It follows that $\Phi$ (and therefore the velocity field $\bU$) is fully determined by the knowledge of its trace $\psi$ at the surface, $\psi(t,X)=\Phi(t,X,\zeta(t,X))$. The full water waves equations \eqref{Euler1}-\eqref{Euler6} can therefore be reduced to a set of two evolution equations on $\zeta$ and $\psi$. The equation for $\zeta$ is furnished by the kinematic equation \eqref{Euler4} while the equation on $\psi$ is obtained by taking the trace of the Bernoulli equation \eqref{Bernoulli} at the surface. Zakharov remarked in \cite{Zakharov1968} that these equations can be put in canonical Hamiltonian form,
\begin{equation}\label{Hamil}
\dt \zeta=\frac{\delta H}{\delta \psi} , \qquad \dt \psi= -\frac{\delta H}{\delta \zeta}, 
\end{equation}
where the Hamiltonian is $H=\frac{1}{\rho}E$, with $E$ the mechanical (potential+kinetic) energy,
$$
H(\zeta,\psi)=\frac{1}{2}\int_{\R^d}g \zeta^2+ \frac{1}{2}\int_{\R^d}\int_{-h_0+b}^\zeta \abs{\nabla_{X,z}\Phi}^2.
$$
\begin{remark}\label{remLuke}
It was also remarked by Luke \cite{Luke}  that the water waves problem has a variational structure. Indeed, defining a Lagrangian density and an action by
\begin{equation}\label{eqLuke}
{\mathcal L}_{\rm Luke}(\Phi,\zeta)=-\int_{-h_0+b}^\zeta \Big( \dt \Phi+\frac{1}{2}\vert \nabla_{X,z}\Phi\vert^2 +g z \Big) {\rm d}z
\end{equation}
and
$$
 {\mathcal I}(\Phi,\zeta)=\int_{t_0}^{t_1}\int_{\R^d} {\mathcal L}_{\rm Luke}(\Phi,\zeta) {\rm d}{X}{\rm d}z {\rm d}t,
$$
 he showed that the corresponding Euler-Lagrange equation coincides with the water waves equations. We refer to \cite{miles1977hamilton} for considerations on the relation between Luke's Lagrangian and Zakharov's Hamiltonian appraoches.
\end{remark}
Introducing the Dirichlet-Neumann operator $G[\zeta,b]$ defined by
$$
G[\zeta,b]\psi= N\cdot \nabla_{X,z}\Phi_{\vert_{z=\zeta}}
\quad\mbox{ where }\quad
\begin{cases}
\Delta_{X,z}\Phi=0 &\mbox{in } \Omega_t \\
\Phi_{\vert_{z=\zeta}}=\psi ,&
N_b\cdot \nabla_{X,z}\Phi_{\vert_{z=-h_0+b}}=0,
\end{cases}
$$
Craig and Sulem \cite{Craig1,Craig2} wrote the equation on $\zeta$ and $\psi$ in explicit form
\begin{equation}\label{ZCS}
\begin{cases}
\dsp \dt \zeta - G[\zeta,b]\psi=0,\\
\dsp \dt \psi + g \zeta +\frac{1}{2}\abs{\nabla \psi}^2 -\frac{1}{2}\frac{(G[\zeta,b]\psi+\nabla\zeta\cdot \nabla\psi)^2}{1+\abs{\nabla\zeta}^2}=0.
\end{cases}
\end{equation}
The local well posedness of this formulation was proved in \cite{Lannes2005}. Not to mention other related issues such as global well posedness for small data, this local existence result has been extended in two different directions: low regularity in \cite{Alazard2014} and uniform bounds in shallow water \cite{Alvarez,Iguchi2009}. These two extensions go somehow in two opposite directions as low regularity focuses on the behavior at high frequencies, while the shallow water limit, considered throughout these notes,  is essentially a low frequency asymptotic. 

\subsection{The elevation/discharge formulation}\label{sectelevdisc}

The Zakharov-Craig-Sulem equations are a set of evolution equations on two functions, $\zeta$ and $\psi$, that do not depend on the vertical variable $z$. Another way of getting rid of the vertical variable is to integrate vertically the free surface Euler equations. Denoting by $V$ and $w$ the horizontal and vertical components of the velocity field $\bU$, this leads to the introduction of the {\it horizontal discharge} $Q$,
\begin{equation}\label{defQ}
Q(t,X):=\int_{-h_0+b(X)}^{\zeta(t,X)} V(t,X,z){\rm d}z;
\end{equation}
integrating the horizontal component of the Euler equation \eqref{Euler1} and using the boundary conditions \eqref{Euler4} and \eqref{Euler6}, this gives
\begin{equation}\label{formzetaQ}
\begin{cases}
\dt \zeta+\nabla\cdot Q=0,\\
\dt Q+\nabla\cdot \Big(\int_{-h_0+b}^\zeta V\otimes V\Big)+\frac{1}{\rho}\int_{-h_0+b}^\zeta \nabla P=0
\end{cases}
\end{equation}
(see for instance the proof of Proposition 3 in \cite{Lannes2017} for details of the computations). 
\begin{remark}\label{remnoncons}
Note that the second equation can also be written as
$$
\dt Q+\nabla\cdot \Big(\int_{-h_0+b}^\zeta \big(V\otimes V+  \frac{P}{\rho}\mbox{Id}\big)\Big)=- (\frac{1}{\rho}  P)_{\vert_{z=-h_0+b}}\nabla b  ;
$$
 when the bottom is flat, the right-hand-side vanishes and the equation takes the form of a conservation law for the horizontal momentum as observed in \cite{whitham1962} in the one dimensional case. When the bottom is not flat, the right-hand-side is non zero and non conservative. Even in the shallow water approximation where it is reduced to $- g (h_0+\zeta-b) \nabla b$, this additional term may induce considerable difficulties (see \S \ref{sectweakNSW} below).
  \end{remark}
The next step is to decompose the pressure term. A special solution to the free surface Euler equations \eqref{Euler1}-\eqref{Euler6} corresponds to the rest state $\zeta=0$, $U=0$; the vertical component of the Euler equation \eqref{Euler1} and the boundary condition \eqref{Euler6} then give the following ODE for $P$,
$$
-\frac{1}{\rho}\dz P-g=0, \qquad P_{\vert_{z=0}}=P_{\rm atm},
$$
and the solution, $P=P_{\rm atm}-\rho g z$ is called \emph{hydrostatic} pressure. When the fluid is not at rest, the solution to the ODE
$$
-\frac{1}{\rho}\dz P-g=0, \qquad P_{\vert_{z=\zeta}}=P_{\rm atm},
$$
namely, $P_{\rm H}=P_{\rm atm}-\rho g (z-\zeta)$ is still called hydrostatic and it is often convenient to decompose the pressure field $P$ into its hydrostatic and non-hydrostatic components,
$$
P=P_{\rm atm}+\rho g (\zeta-z)+P_{\rm NH};
$$
integrating the vertical component of \eqref{Euler1} from $z$ to $\zeta$ and taking into account the boundary condition \eqref{Euler6}, one readily derives the following expression for the non-hydrostatic pressure,
\begin{equation}\label{PNH}
P_{\rm NH}(t,X,z)=\rho \int_z^{\zeta(t,X)} (\dt w +{\bf U}\cdot \nabla_{X,z}w).
\end{equation}
The evolution equations on $\zeta$ and $Q$ can then be written under the form
$$
\begin{cases}
\dt \zeta+\nabla\cdot Q=0,\\
\dt Q+\nabla\cdot \Big(\int_{-h_0+b}^\zeta V\otimes V\Big)+gh\nabla \zeta+\frac{1}{\rho}\int_{-h_0+b}^\zeta \nabla P_{\rm NH}=0,
\end{cases}
$$
where $h$ is the water height, $h=h_0+\zeta-b$.
The quadratic term in the second equation shows the importance of measuring the vertical dependance of the horizontal velocity $V$; this dependance is considered as a variation with respect to the vertical average of $V$. More precisely, we decompose the horizontal velocity field as
$$
V(t,X,z)=\overline{V}(t,X)+V^*(t,X,z)
$$
where for any function $f(t,\cdot)$ defined in the fluid domain $\Omega_t$, we use the notation
$$
\overline{f}(t,X)=\frac{1}{h}\int_{-h_0+b}^\zeta f (t,X,z){\rm d}z\quad \mbox{ and }\quad f^*(t,X,z)=f(t,X,s)-\overline{f}(t,X).
$$
We can therefore write 
\begin{equation}\label{defR}
\int_{-h_0+b}^\zeta V\otimes V= \frac{1}{h}Q\otimes Q+\Rey
\quad\mbox{ with }\quad
\Rey=\int_{-h_0+b }^\zeta V^*\otimes V^*
\end{equation}
so that the equations take the form
\begin{equation}\label{formzetaQ3}
\begin{cases}
\dsp \dt \zeta+\nabla\cdot Q=0,\\
\dsp \dt Q+\nabla\cdot(\frac{1}{h}Q\otimes Q)+gh\nabla \zeta+\nabla\cdot\Rey+\frac{1}{\rho}\int_{-h_0+b}^\zeta \nabla P_{\rm NH}=0.
\end{cases}
\end{equation}
\begin{remark}
The average horizontal velocity $\overline{V}$ and the horizontal discharge $Q$ are related through $Q=h \overline{V}$. Instead of \eqref{formzetaQ3}, one can therefore equivalently write a system of equations on the variables $\zeta$ and $\ovV$, namely,
\begin{equation}\label{formzetaQ3V}
\begin{cases}
\dsp \dt \zeta+\nabla\cdot (h\ovV)=0,\\
\dsp \dt \ovV+\ovV\cdot \nabla \ovV+g\nabla \zeta+\frac{1}{h}\nabla\cdot\Rey+\frac{1}{\rho h}\int_{-h_0+b}^\zeta \nabla P_{\rm NH}=0.
\end{cases}
\end{equation}
\end{remark}
Obviously, the last two terms of the second equation in \eqref{formzetaQ3} are the most complicated ones. To begin with, they are defined through \eqref{PNH} and \eqref{defR} in terms of the velocity field $\bU(t,X,z)$ and not in terms of $\zeta$ and $Q$. We can however state the following result in which $\Omega$ stands for the fluid domain corresponding to the surface parametrization $\zeta$.
\begin{prop}\label{propclosed}
The equations \eqref{formzetaQ3} form a closed set of equations in $\zeta$ and $Q$. More precisely, if we denote
$$
L^2_b(\Omega,{\rm div},\curl):=\{{\bf U}\in L^2(\Omega)^{d+1}, {\rm div }\,{\bf U}=0, \curl \bU=0 \quad\mbox{and}\quad U_b\cdot N_b=0 \},
$$
then the discharge and reconstruction mappings respectively defined by
$$
{\mathfrak D}[\zeta]: \begin{array}{lcl}
L^2_b(\Omega,{\rm div},\curl) & \to & H^{1/2}(\R^d)^d \\
\bU=\left(\begin{array}{c} V \\ w \end{array}\right) &\mapsto & Q:=\int_{-h_0+b}^\zeta V
\end{array}
$$
and
$$
{\mathfrak R}[\zeta]: \begin{array}{lcl}
H^{1/2}(\R^d)^d  & \to & L^2_b(\Omega,{\rm div},\curl)  \\
Q &\mapsto & \nabla_{X,z}\Phi
\end{array}
\quad\mbox{ with }\quad
\begin{cases}
\Delta_{X,z}\Phi=0  \quad \mbox{ in }\Omega,\\
N\cdot \nabla_{X,z}\Phi _{\vert_{z=\zeta}}=-\nabla\cdot Q\\
N_b\cdot \nabla_{X,z}\Phi _{\vert_{z=-h_0+b}}=0
\end{cases}
$$
are well defined and ${\mathfrak R}[\zeta]$ is a left-inverse to ${\mathfrak D}[\zeta]$.
\end{prop}
We refer to \cite{Lannes2017} for the proof, which relies on the key observation that
$$
N\cdot \nabla_{X,z}\Phi _{\vert_{z=\zeta}}=- \nabla\cdot \big( \int_{-h_0+b}^\zeta V \big).
$$
As a consequence of Proposition \ref{propclosed}, the last two terms in \eqref{formzetaQ3} are (non explicit, non local, non linear) functions of $\zeta$ and $Q$:
\begin{itemize}
\item Since $V^*=V-\ovV$ denotes the fluctuation of the horizontal velocity $V$ with respect to its vertical average $\ovV$,   of the horizontal velocity field. The tensor $\Rey=\int_{-h_0+b}^\zeta V^*\otimes V^*$ measures the contribution to the momentum equation of these fluctuations. It is therefore reminiscent of the Reynolds stress tensor in turbulence.
\item The non-hydrostatic pressure contains nonlinear but also linear terms; as we shall see, it contains in particular the linear dispersive effects that are important for a good description of wave propagation.
\end{itemize}
These terms are very complex, but it is possible to derive relatively simple asymptotic expansions in terms of $\zeta$ and $Q$ in some particular regimes. In deep water, asymptotic models can be derived for waves of small steepness (see for instance \cite{Matsuno1992,Matsuno1993,Choi1995,Craig2006,BonnetonLannes,Lannes2013}), but we shall focus throughout these notes on {\it shallow water} models.

\subsection{Nondimensionalization of the equations}\label{sectnondim}

In order to study the asymptotic behavior of the solutions to the water waves equations, it is convenient to introduce non-dimensionalized quantities based on the typical scales of the problem, namely: the typical depth $h_0$, the order of the surface variation $a_{\rm surf}$, the order of the bottom variations $a_{\rm bott}$ and the typical horizontal scale $L$. We can therefore form three dimensionless parameters
$$
\mu=\frac{h_0^2}{L^2}, \qquad \eps=\frac{a_{\rm surf}}{h_0},\qquad \beta=\frac{a_{\rm bott}}{h_0}.
$$
The first one is the {\it shallowness parameter}, the second the {\it amplitude parameter}, and the third the {\it topography parameter.} We are interested throughout this article in {\it shallow water configurations}, in the sense that $\mu$ is assumed to be small.
\begin{remark}
Another parameter, the {\it steepness} $\epsilon=\frac{a}{L}=\eps \sqrt{\mu}$ is also found in the literature, but its main relevance is in intermediate to deep water, and it will therefore not been used in these notes.
\end{remark}
Dimensionless quantities are defined as follows,
\begin{align*}
\widetilde{X}=\frac{X}{L}, \qquad \widetilde{z}=\frac{z}{h_0},\qquad \widetilde{t}=\frac{t}{L/\sqrt{gh_0}}, \\
 \widetilde{\zeta}=\frac{\zeta}{a_{\rm surf}}, \qquad \widetilde{b}=\frac{b}{a_{\rm bott}},
\qquad 
\widetilde{Q}=\frac{Q}{a_{\rm surf}\sqrt{gh_0}},\qquad \widetilde{w}=.\frac{w}{aL/h_0\sqrt{g/h_0}}.
\end{align*}
Plugging into \eqref{formzetaQ3} then yields the dimensionless form of the equations. Omitting the tildes for the sake of clarity, they read
\begin{equation}\label{formzetaQ3_ND}
\begin{cases}
\dt \zeta+\nabla\cdot Q=0,\\
\dt Q+\eps \nabla\cdot(\frac{1}{h}Q\otimes Q)+h\nabla \zeta+\eps \nabla\cdot\Rey+\frac{1}{\eps}\int_{-1}^{\eps \zeta} \nabla P_{\rm NH}=0,
\end{cases}
\end{equation}
where the dimensionless water height is $h=1+\eps\zeta-\beta b$ and the dimensionless "turbulent" tensor ${\bf R}$ and non-hydrostatic pressure are
\begin{equation}\label{defRPNH}
{\bf R}=\int_{-1+\beta b}^{\eps \zeta} V^*\otimes V^*
\quad\mbox{ and }\quad
\frac{1}{\eps}P_{\rm NH}=\int_z^{\eps \zeta} \big( \dt w +\eps V\cdot \nabla w +\frac{\eps}{\mu}w \dz w \big),
\end{equation}
with, in their dimensionless version, 
$$
\ovV=\frac{1}{h}\int_{-1+\beta b}^{\eps\zeta} V(t,X,z){\rm d} z
\quad \mbox{ and }\quad V^*(t,X,z)=V(t,X,z)-\ovV(t,X).
$$
The equations \eqref{formzetaQ3_ND} can equivalently be written in $(\zeta,\ovV)$ variables (recall that $Q=h\ovV$),
\begin{equation}\label{formzetaQ3_ND_V}
\begin{cases}
\dt \zeta+\nabla\cdot (h\ovV)=0,\\
\dt \ovV+\eps \ovV\cdot \nabla\ovV+\nabla \zeta+\eps \frac{1}{h}\nabla\cdot\Rey+\frac{1}{\eps h}\int_{-1}^{\eps \zeta} \nabla P_{\rm NH}=0.
\end{cases}
\end{equation}
\begin{remark}\label{remZCS_ND}
Similarly, one can derive a dimensionless version of the Zakharov-Craig-Sulem formulation \eqref{ZCS},
\begin{equation}\label{ZCS_ND}
\begin{cases}
\dsp \dt \zeta - \frac{1}{\mu}G_\mu[\eps\zeta,\beta b]\psi=0,\\
\dsp \dt \psi + \zeta +\eps \frac{1}{2}\abs{\nabla \psi}^2 -\frac{1}{2}\eps\mu\frac{(\frac{1}{\mu}G_\mu[\eps \zeta,\beta b]\psi+\eps \nabla\zeta\cdot \nabla\psi)^2}{1+\eps^2\abs{\nabla\zeta}^2}=0,
\end{cases}
\end{equation}
where 
$G_\mu[\eps\zeta,\beta b]\psi= \big( \dz\Phi -\eps \mu \nabla \zeta\cdot \nabla \Phi\big)_{\vert_{z=\eps\zeta}}$ and 
$$
\begin{cases}
(\dz^2+\mu \Delta)\Phi=0 &\mbox{for  } -1+\beta b < z < \eps \zeta \\
\Phi_{\vert_{z=\eps \zeta}}=\psi ,&
\big( \dz\Phi -\beta \mu \nabla b \cdot \nabla \Phi\big)_{\vert_{z=-1+\beta b}}=0.
\end{cases}
$$
Setting $\eps=\beta=0$, one gets the linearized water waves equations for a flat bottom. In this case, the equation for $\Phi$  can be explicitly solved and the Dirichlet-Neumann operator becomes a simple Fourier multiplier $G_\mu[0,0]=\sqrt{\mu}\abs{D}\tanh(\sqrt{\mu}\abs{D})$. In particular, the linear dispersion relation for the water waves equations is 
$$
\omega_{\rm WW}^2 = k^2 \frac{\tanh(\sqrt{\mu}k)}{\sqrt{\mu}k},
$$
where ${\bf k}$ is a wave number of a plane wave solution of the linearized equations, $k=\abs{{\bf k}}$ and $\omega_{\rm WW}$ the associated frequency.
\end{remark}

\section{The nonlinear shallow water equations and higher order approximation for irrotational flows}\label{sect2}

We derive and comment in this section several shallow water asymptotic models. 
In the dimensionless version of  the water waves equations \eqref{formzetaQ3_ND} there are the nonlocal "turbulent"  and non-hydrostatic components. These two terms involve the velocity and pressure fields inside the fluid domain and if one wants to study their asymptotic behavior in shallow water it is therefore necessary to describe the inner structure of the velocity and pressure fields; this is performed in \S \ref{sectinner} and \S \ref{sectNH} respectively. The first model obtained in the shallow water asymptotics is the nonlinear shallow water (NSW) system; it is derived in \S \ref{sectNSW} where its mathematical properties and several open problems are also reviewed. We then address in \S \ref{sectBoussinesq} the Boussinesq equations which furnish a second order approximation with respect to the shallowness parameter $\mu$, but under a smallness assumption on the amplitude of the waves (weak nonlinearity). Removing this smallness assumption, one obtains the more general but more complex Serre-Green-Naghdi equations (SGN), which are derived and commented in \S \ref{sectSGN}. In order to get a better resolution of the vertical structure of the flow, multi-layer extensions of these models have been recently proposed; we present them in \S \ref{sectmultilayer}. Another type of higher order model, described in \S \ref{sectIsobe}, is the Isobe-Kakinuma, derived from variational arguments. We then turn in \S \ref{sectscalar} to investigate one directional waves that are interesting because they can be described by a single scalar equation easier to analyze.

\subsection{The inner structure of the velocity field}\label{sectinner}

It is possible to describe the inner structure of the velocity field in shallow water by using the incompressibility and irrotationality conditions \eqref{Euler2} and \eqref{Euler3}, as well as the bottom boundary condition \eqref{Euler6}. In their dimensionless version, these conditions become
\begin{equation}\label{liste}
\begin{cases}
\mu\nabla\cdot V+\dz {w}=0,\\
\dz V-\nabla w=0,\\
\nabla^\perp \cdot V=0,\\
w_b-\beta \mu \nabla b\cdot V_b=0.
\end{cases}
\end{equation}
\begin{remark}
The irrotationality assumption made above can be relaxed to derive asymptotic models; as shown in Section \ref{sectrot} the contribution of the vorticity is only felt in the asymptotic models if the \emph{vortex strength} defined in \eqref{strength} is large enough.
\end{remark}
The first and last equations can be used to obtain
$$
w=-\mu\nabla\cdot \big[ (1+z-\beta b)\overline{V}\big]-\mu\nabla\cdot \int_{-1+\beta b}^z V^*.
$$
and with the second equation this yields
\begin{align*}
V^*&=- \Big(\int_z^{\eps\zeta} \nabla w\Big)^*\\
&=\mu \Big( \int_z^{\eps\zeta} \nabla\nabla\cdot \big[(1+z'-\beta b)\overline{V}\big]{\rm d}z'\Big)^*+\mu \Big(\int_{z}^{\eps\zeta}\nabla\nabla\cdot \int_{-1+\beta b}^z V^* \Big)^*.
\end{align*}
It is therefore natural to introduce the operators ${\bf T}[\eps\zeta,\beta b]$ and ${\bf T}^*[\eps\zeta,\beta b]$ acting on $\R^d$-valued functions and defined on the fluid domain $\Omega$ and defined as
\begin{equation}\label{defTT}
{\bf T}[\eps\zeta,\beta b]W=\int_z^{\eps\zeta}\nabla\nabla\cdot \int_{-1+\beta b}^{z'}W
\quad\mbox{ and }\quad
{\bf T}^*[\eps\zeta,\beta b]W=\big({\bf T}[\eps\zeta,\beta b]W\big)^*.
\end{equation}
The above expression for $V^*$ can then be written under the form
$$
(1-\mu {\bf T}^*)V^*=\mu {\bf T}^* \overline{V}
$$
so that 
\begin{align*}
V^*&=\mu {\bf T}^* \overline{V}+O(\mu^2).
\end{align*}
Since  $\ovV$ does not depend on $z$, the quantity ${\bf T}^* \overline{V}$ can be computed explicitly, leading to  a shallow water expansion of the inner velocity field in terms of $\zeta$ and $\overline{V}$. When the bottom is flat ($b=0$), this expansion reads
\begin{equation}\label{expinnervel}
\begin{cases}
V=\overline{V}- \frac{1}{2}\mu \big(  (1+z)^2-\frac{1}{3} h^2 \big) \nabla\nabla\cdot \overline{V}+O(\mu^2),\\
w=-\mu (1+z)\nabla\cdot \overline{V}+O(\mu^2);
\end{cases}
\end{equation}
for the sake of clarity, the generalization in the presence of topography is given in \eqref{expinnerveltopo} in Appendix \ref{Appbottom}.
\begin{remark}
For the derivation of the asymptotic models below, the first order approximation $V=\ovV+O(\mu)$ is enough, but the formula \eqref{expinnervel} shows that it is possible to reconstruct the vertical dependance of the velocity from the knowledge of $\ovV$. One could actually generalize the procedure used above to reconstruct the inner velocity field at order $O(\mu^N)$ for any $N$. This formula however would involve high order derivatives of $\ovV$ and could more importantly be completely irrelevant. For instance, if $\ovV$ and $\zeta$ are given through the resolution of the nonlinear shallow water derived below, then, as we shall see, $\ovV$ will only be known up to an error of size $O(\mu t)$; therefore, even the $O(\mu)$ corrector in the formula for $V$ in \eqref{expinnervel} is irrelevant since it is of the same order as the error made on $\ovV$.
\end{remark}

\subsection{The inner structure of the pressure field}\label{sectNH}

As already seen, the pressure field can be written as the sum of the hydrostatic pressure and a non-hydrostatic correction. In dimensionless variables, this reads
$$
P=(\eps \zeta-z)+P_{\rm NH}
\quad \mbox{ with }\quad
\frac{1}{\eps}P_{\rm NH}= \int_z^{\eps \zeta} \big( \dt w +\eps V\cdot \nabla w +\frac{\eps}{\mu}w \dz w \big).
$$
From the asymptotic expansion \eqref{expinnervel}, we deduce that, when the bottom is flat,
\begin{align}
\label{expP}
\frac{1}{\eps} P_{\rm NH}&=-\mu\big[ \frac{h^2}{2}-\frac{(1+z)^2}{2}\big]  \big(\dt +\eps \ovV\cdot \nabla -\eps\nabla\cdot \ovV\big)\nabla\cdot\ovV
+O(\mu^2);
\end{align}
we refer to \eqref{expPtopo} for the generalization of this formula when the bottom is not flat.

It follows that if one knows $\zeta$ and $\ovV$ (from experimental measurement or, approximately, by solving one of the asymptotic models derived below) then it is possible to reconstruct the pressure field in the fluid domain. An interesting problem for applications to coastal oceanography is the inverse problem: is it possible to reconstruct the surface elevation $\zeta$ by pressure measurements at the bottom (through pressure sensors lying on the sea bed). In the case of progressive waves (solitary or cnoidal waves), it is possible to do so (see for instance \cite{Oliveras2012,ClamondConstantin}) but the situation is more complex for general non progressive waves. Indeed, as many inverse problems, this reconstruction is an ill-posed problem (one roughly has to solve a Laplace equation in the fluid domain with no boundary condition at the surface and double Dirichlet and Neumann condition at the bottom). An heuristic formula was proposed in \cite{Vasan2017} and a weakly nonlinear reconstruction was derived in \cite{BonnetonLannes2017} (and experimentally validated with in situ measurements \cite{Bonneton2018,Mouragues2019}) using an additional argument of nonsecular growth to circumvent this ill-posedness.

\subsection{First order approximation: the nonlinear shallow water equations}\label{sectNSW}

The nonlinear shallow water equations are an approximation of order $O(\mu)$ of the water waves equations \eqref{formzetaQ3_ND} in the sense that terms of order $O(\mu)$ are dropped. The main point consists therefore in studying the dependence of the "turbulent" and non-hydrostatic terms on $\mu$.

From the results of \S \ref{sectinner} and \S \ref{sectNH}, and recalling the definition \eqref{defRPNH} of ${\bf R}$ and $P_{\rm NH}$, we easily get that
$$
\nabla\cdot \Rey=O(\mu^2) \quad \mbox{ and }\quad \frac{1}{\eps}\int_{-1+\beta b}^{\eps\zeta}\nabla P_{\rm NH}=O(\mu).
$$
Neglecting the $O(\mu)$ terms in the $(\zeta,Q)$ formulation of the water waves equations \eqref{formzetaQ3_ND}, one obtains the nonlinear shallow water equations (NSW),
\begin{equation}\label{eqNSW}
\begin{cases}
\dt \zeta+\nabla\cdot Q=0,\\
\dt Q+\eps \nabla\cdot (\frac{1}{h}Q\otimes Q)+h\nabla\zeta=0,
\end{cases}
\mbox{ for }t\geq 0,\quad x\in \R^d,
\end{equation}
with $h=1+\eps\zeta-\beta b$ (see \eqref{eqNSWop} below for an equivalent formulation in $(\zeta,\overline{V})$ variables). This is a hyperbolic system of equations that furnishes a quite rough but very robust approximation for shallow water waves. We review below several known results and open problems related to the NSW model.
\subsubsection{The initial value (or Cauchy) problem for strong solutions to the NSW equations}
The NSW equations \eqref{eqNSW} can be equivalently written in $(\zeta,\ovV)$ variables (recall that $Q=h\ovV$),
\begin{equation}\label{eqNSWop}
\begin{cases}
\dt \zeta+\nabla\cdot (h \overline{V})=0,\\
\dt \overline{V}+\eps\overline{V}\cdot \nabla \overline{V}+\nabla\zeta=0,
\end{cases}
\mbox{ for }t\geq 0,\quad x\in \R^d
\end{equation}
with $h=1+\eps\zeta-\beta b$ and with initial condition
$$
(\zeta,\overline{V})_{\vert_{t=0}}=(\zeta^{\rm in},\overline{V}^{\rm in}).
$$
There  is local conservation of energy for the NSW equations,
\begin{equation}\label{locNRJNSW}
\dt {\mathfrak e}_{\rm NSW}+ \nabla\cdot {\mathfrak F}_{\rm NSW}=0,
\end{equation}
with energy density and energy flux given by
$$
{\mathfrak e}_{\rm NSW}= \frac{1}{2} \big[ \zeta^2 + h \vert \overline{V}\vert^2 \big]
\quad \mbox{ and }\quad
{\mathfrak F}_{\rm NSW}= \big( \zeta+\eps \frac{1}{2} \abs{\overline{V}}^2 \big)h \ovV;
$$
in particular, this yields conservation of the mechanical energy,
$$
\frac{d}{dt}E_{\rm NSW}=0
\quad\mbox{ with }\quad
E_{\rm NSW}=\int_{\R^d}{\mathfrak e}_{\rm NSW}.
$$
Under the non vanishing depth condition,
\begin{equation}\label{nonvanish}
\exists h_{\rm min}>0,\qquad \sup_{X\in \R^d}h(t,X)\geq h_{\rm min},
\end{equation}
the conservation of $E_{\rm NSW}$ therefore furnishes a control of the $L^2$-norm of $(\zeta,\overline{V})$. The non-vanishing depth condition actually ensures that the NSW equations form a Friedrichs symmetrizable hyperbolic system. It follows therefore from the general theory of Friedrich symmetrizable hyperbolic systems (see for instance \cite{AG,Taylor3,BG}) that the initial value problem  is locally well posed for times of order $O(1/\eps)$ if the initial data $(\zeta^{\rm in},\ovV^{\rm in})$ belongs to $H^s(\R^d)^{1+d}$ with $s>1+d/2$ and satisfies the non vanishing depth condition \eqref{nonvanish}. Note that the $O(1/\eps)$ time scale for the life span of the solutions is optimal in dimension $d=1$ since shocks are known to develop at this time scale. Despite recent breakthroughs \cite{christodoulou2014,luk2018shock,buckmaster2019} (these references deal with the isentropic Euler equations which are related to  the NSW equations as explained below), the scenario for shock formation in dimension $d=2$ remains a difficult open problem. Finally, let us mention that if the non-vanishing depth condition  is relaxed, then the problem becomes a much more complex free boundary system of equations (see below). 

\subsubsection{Weak solutions}\label{sectweakNSW}
 In the case of a flat topography ($b=0$) the NSW equations coincide with the isentropic Euler equations for compressible gases, with $h$ playing the role of the density and with pressure law ${\mathcal P}(\rho)= \frac{1}{2} g \rho^2$, and it is therefore possible to use the construction of weak-entropy solutions following the dense literature on compressible gases, such as \cite{Diperna,Lions,Chen2012}; these solutions are obtained as the inviscid limit of viscous generalization of the NSW equations. We refer to \cite{Bresch2009} for a review on these topics. Uniqueness remains an open problem. \\ 
The situation for the two-dimensional case is even more complicated, and almost nothing is known. As stated by Lax \cite{Lax},
\begin{quote}
There is no theory for the initial value problem for
compressible flows in two space dimensions once shocks
show up, much less in three space dimensions. This is a
scientific scandal and a challenge.
\end{quote}
Fortunately,
\begin{quote}
Just because we cannot prove that compressible 
flows with
prescribed initial values exist doesn't mean that we cannot compute them.
\end{quote}
And indeed, shocks are computed for the NSW in many applications; in coastal oceanography for instance, shocks are relevant because they are used to describe wave breaking. The mathematical entropy coincides for the NSW equations with the energy; the dissipation of entropy associated to weak entropy solutions is therefore a dissipation of energy that corresponds with a pretty good accuracy to the energy actually dissipated by wave breaking \cite{Bonneton2010}. See also \S \ref{WB} below for more considerations on the modeling of wave breaking.\\
Let us finally mention briefly the case of a non flat bottom  ($b\neq 0$); the momentum equation is then given by
 $$
 \dt Q+ \nabla\cdot \big(\eps \frac{1}{h} Q\otimes Q+\frac{1}{2\eps }h^2 \mbox{Id} \big)=-\frac{\beta}{\eps } h \nabla b
 $$
which is no longer in conservative form due to the presence of the source term in the right-hand-side (inherited from a similar non conservative term in the full averaged Euler equations, see Remark \ref{remnoncons}). Even in dimension $d=1$, there is no fully satisfactory theory at this day to define weak solutions and products of shocks in this framework \cite{abgrall2010}; this is another theoretical and numerical challenge.

\subsubsection{Initial Boundary value problems.}\label{sectIBVPNSW}
 The equations \eqref{eqNSWop} are cast on $\R^d$ but the equations must sometimes be considered in a domain with a boundary. This boundary can be physical (e.g. a wall) or artificial: for instance, for numerical simulations, one has to consider a bounded domain whose boundary has no physical relevance. For the sake of clarity, let us discuss first the one-dimensional case $d=1$, on a finite interval $[0,L]$,
\begin{equation}\label{NSW1dop}
\begin{cases}
\dt \zeta +\dx (h \overline{v})=0,\\
\dt \overline{v}+\eps\overline{v} \dx \overline{v}+\dx\zeta=0,
\end{cases}
\mbox{ for }t\geq 0,\quad x\in (0,L)
\end{equation}
with $h=1+\eps\zeta-\beta b$ and with initial condition
$$
(\zeta,\overline{v})_{\vert_{t=0}}=(\zeta^{\rm in},\overline{v}^{\rm in}) \quad \mbox{ on }\quad [0,L].
$$
In addition, boundary conditions must be imposed at $x=0$ and $x=L$. Some examples of boundary conditions are
\begin{itemize}
\item Generating boundary conditions. The water elevation is known (from buoy measurements for instance) at the entrance of the domain and prescribed as a boundary data,
$$\zeta(t,0)=f(t);$$
in this case, the boundary $x=0$ is non physical.
\item Wall. There is a fixed wall located at $x=L$, on which the waves bounces back. In this case the boundary $x=L$ is physical and the corresponding boundary condition is
$$
\overline{v}(t,L)=0.
$$
\item  Transparent conditions. Such boundary conditions are very important for numerical simulations in the cases where there is no physical boundary condition at $x=L$ and one wants to impose a boundary condition that does not create any artificial reflexion.  In the particular case of the NSW in dimension $d=1$, a simple analysis of the Riemann invariants shows that such a condition is given by
$$
R_-(\zeta,\overline{v}):=2(\sqrt{h}-1)-\eps \overline{v}=0 \qquad \mbox{at } x=L,
$$
where $R_-$ is the left going Riemann invariant (see \S \ref{sectFNscal} below for more details).
\end{itemize}
Initial boundary value problems for hyperbolic systems have been considered quite intensively \cite{Majda1,Majda2,Majda3,Metivier2001,Metivier2012,Freistuhler1998,BG,Coulombel2003}; we refer to \cite{IguchiLannes} for a sharp general theory in dimension $d=1$ showing that such problems are locally well-posed in $H^m$ ($m\geq 2$) under suitable compatibility conditions. In the particular case of $2\times 2$ systems, an analysis based on Riemann invariants can also be performed \cite{Li1985}, and proves very useful for numerical implementation (see for instance \cite{Marche,LannesWeynans}). In dimension $d=2$, the "wall" boundary condition $\overline{V}\cdot {\bf n}=0$ can be deduced from classical works on the compressible Euler equations \cite{Schochet1986} but other types of boundary conditions are much more delicate and remain an open problem.

\subsubsection{A free boundary problem: the shoreline problem.} 
The non-vanishing depth condition \eqref{nonvanish} is of course a serious restriction for applications to coastal oceanography, where one typically has to deal with beaches. Let us consider the case for instance where the shoreline is at time $t$ the graph of some function $y\in \R \mapsto \underline{X}(t,y)$ if $d=2$ (and a single point $\underline{x}(t)$ if $d=1$) and that the sea is, say, on the right part of the shoreline (see Figure \ref{figshoreline}).
\begin{figure}
\begin{center}
\includegraphics[width=0.7\linewidth]{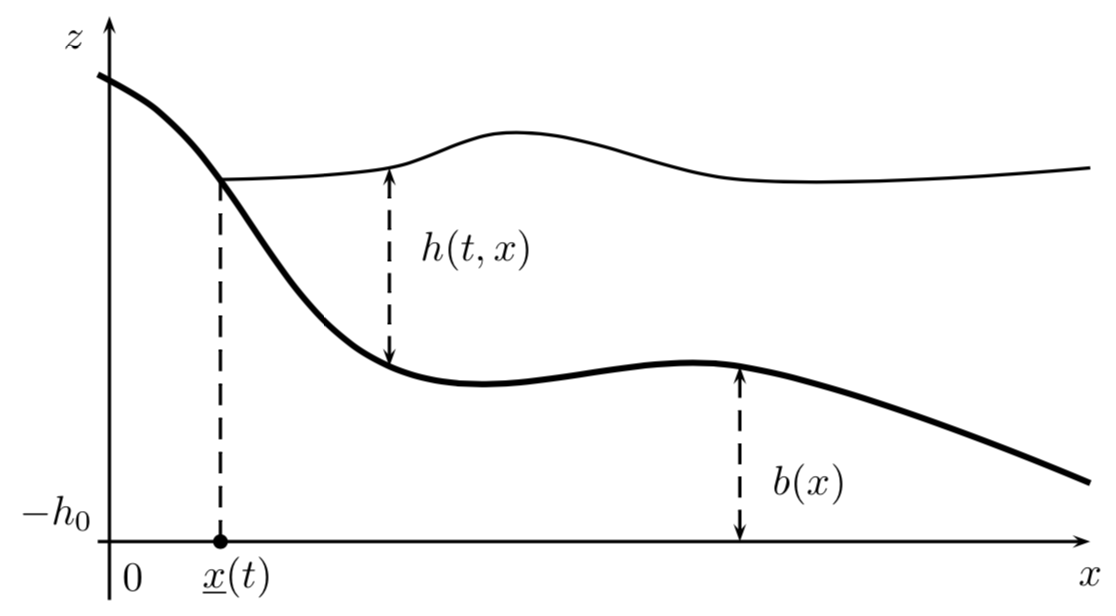}
\end{center}
\caption{The shoreline problem in dimension $d=1$}
\label{figshoreline}
\end{figure}

The initial value problem is then much more difficult since it is now a free boundary problem: one must solve the NSW equations on ${\mathcal O}_t=\{X=(x,y)\in \R^2, x>\underline{X}(t,y)\}$ (or ${\mathcal O}_t=\{x\in \R, x>\underline{x}(t)\}$ if $d=1$) whose boundary, the shoreline (or more accurately, its projection on the horizontal plane) evolves according to the kinematic equation
\begin{equation}\label{kinematic}
\dt \underline{X}=\overline{V}_{\vert_{x=\underline{X}(t,y)}}\cdot \left(\begin{array}{c} 1 \\ \partial_y\underline{X}\end{array}\right)
\qquad (\mbox{or }\underline{x}'(t)=u(t,\underline{x}(t)) \mbox{ if } d=1),
\end{equation}
which involves the trace at the boundary of the velocity. A reasonable assumption to solve this free boundary problem is to assume that the surface of the water is transverse to the bottom topography at the shoreline in the following sense
\begin{equation}\label{transverse}
\partial_\nu h <0 \mbox{ on } \{X=\underline{X}(t,y)\},
\end{equation}
where $\nu$ is the outwards unit normal to ${\mathcal O}_t$ (if $d=1$ this condition reduces to $\dx h (t,\underline{x}(t))>0$, i.e. the surface of the water is not tangent to the bottom at the contact point).
Proving that the shoreline problem is well-posed consists in proving that there exists a smooth enough family of mapping $t\mapsto \underline{X}(t,\cdot)$ (or simply $t\mapsto \underline{x}(t)$) on some time interval $[0,T]$ and a family of smooth enough functions $\zeta$ and $\overline{V}$ solving the nonlinear shallow water on ${\mathcal O}_t$ and the kinematic equation \eqref{kinematic}. In dimension $d=1$, such a result can be found in \cite{LannesMetivier} as a particular case of a more general result for the Green-Naghdi equations, but the dispersive terms of this latter make the analysis more complicated than necessary, and the proof could certainly be simplified considerably if one is only interested in the nonlinear shallow water equations. Let us also mention that the isentropic Euler equations for compressible gases with vacuum has been solved in \cite{JangMasmoudi1} and \cite{CoutandShkoller1} for $d=1$ and \cite{JangMasmoudi2} and \cite{CoutandShkoller2} for $d\geq 2$ under the assumption  of a physical boundary condition at the interface with vacuum (using the terminology of \cite{Liu}), namely,
$$
-\infty < \partial_\nu c^2 <0 \mbox{ at the inferface with vacuum},
$$
where $c=({\mathcal P}'(\rho))^{1/2}$ is the sound speed. Using the analogy mentioned in \S \ref{sectweakNSW}, the vacuum problem with physical boundary condition exactly coincides with the shoreline problem with transversality condition \eqref{transverse} in the case of a flat topography ($b=0$); an extension of the techniques of the above references to the case of a non-flat topography looks feasible and could be done to cover the two-dimensional case $d=2$. Let us also mention \cite{de2019priori} (and \cite{ming2017water,ming2018water} for a non zero surface tension) where the water waves equations are solved in the presence
of an emerging bottom. The derivation of the NSW equations from the water waves equations in this context is an open problem.

\subsection{Weakly nonlinear second order approximations: the Boussinesq equations}\label{sectBoussinesq}

Compared to the NSW equations, the Boussinesq equations have a better precision, namely, $O(\mu^2)$ instead of $O(\mu)$, but require an additional assumption of {\it weak nonlinearity} that can be formulated as a smallness condition on $\eps$,
\begin{equation}\label{weakNL}
 \mbox{Weak nonlinearity:}\qquad \eps=O(\mu).
\end{equation}
Traditionally (but not always as we shall see below for the Boussinesq-Peregrine model), an assumption on the smallness of the topography variations is also made,
\begin{equation}\label{weaktopo}
 \mbox{Small topography variations:}\qquad \beta=O(\mu).
\end{equation}
Under these two assumptions, terms of size $O(\eps\mu)$ and $O(\beta \mu)$ can be treated as $O(\mu^2)$ terms, and the results of \S \ref{sectinner} and \S \ref{sectNH} yield the following approximations on the turbulent and non-hydrostatic terms  ${\bf R}$ and $P_{\rm NH}$ defined in \eqref{defRPNH},
\begin{align*}
\nabla\cdot \Rey&=O(\mu^2)\\
\frac{1}{\eps}\int_{-1}^{\eps\zeta}\nabla P_{\rm NH}&=-\mu \frac{1}{3}\nabla\nabla\cdot \dt \overline{V} + O(\mu^2)\\
&=-\mu \frac{1}{3}\Delta \dt \overline{V} + O(\mu^2),
\end{align*}
the last identity stemming from the third equation in \eqref{liste} and \eqref{expinnervel}. Plugging these approximations into \eqref{formzetaQ3_ND_V} and dropping the $O(\mu^2)$ terms, one obtains the following Boussinesq equations
\begin{equation}\label{Boussinesq}
\begin{cases}
\dt \zeta +\nabla\cdot(h\overline{V})=0,\\
(1-\mu \frac{1}{3}\Delta)\dt \overline{V}+\eps(\overline{V}\cdot \nabla \overline{V})+\nabla\zeta=0.
\end{cases}
\end{equation}
\begin{remark}
The irrotationality assumption has been used to replace $(1-\mu \frac{1}{3}\nabla\nabla^{\rm T})\dt \ovV$ by the simpler term $(1-\mu \frac{1}{3}\Delta)\dt \ovV$. In the presence of vorticity, it is in general not possible to do so (see \S \ref{sectBoussinesqlarge} below).
\end{remark}

There is actually not a single  Boussinesq model, but a whole family. There are various reasons why many formally equivalent Boussinesq models have been derived, such as their mathematical structure (well-posedness, conservation of energy, integrability, solitary waves, etc.) or their physical properties. Among the latter, the linear dispersive properties of these models is a central question. The linear dispersion associated to \eqref{Boussinesq} is 
\begin{equation}\label{dispBouss}
\omega_{\rm B}^2=\frac{k^2}{1+\frac{1}{3}\mu k^2}
\end{equation}
where ${\bf k}$ is a wave number, $k=\abs{{\bf k}}$ and $\omega$ the associated frequency. This dispersion relation is as expected a $O(\mu^2)$ approximation of the linear dispersion relation of the full water waves equations (see Remark \ref{remZCS_ND}), 
$$
\omega_{\rm WW}^2=k^2 \frac{\tanh(\sqrt{\mu}k)}{\sqrt{\mu}k},
$$
but the two formulas differ significantly when $\sqrt{\mu}k$ is not very small (i.e. for shorter waves and/or larger depth). It is possible to derive Boussinesq models with better dispersive properties and that differ from \eqref{Boussinesq} by $O(\mu^2)$ terms, and therefore keep the same overall $O(\mu^2)$ precision. These new Boussinesq systems depend on several parameters. The first one can be introduced using the so-called BBM trick \cite{BBM} that is based on the observation that
\begin{align*}
\dt \ovV&=-\nabla\zeta +O(\mu),\\
&=\alpha \dt \ovV-(1-\alpha)\nabla\zeta +O(\mu),
\end{align*}
for any real number $\alpha$. This substitution can be made in the dispersive term in the second equation of \eqref{Boussinesq}, 
$$
-\mu\frac{1}{3}\Delta \dt \ovV=-\mu\alpha \frac{1}{3}\Delta\dt \zeta  +\mu \frac{1}{3}(1-\alpha)\Delta \nabla\zeta +O(\mu^2)
$$
and induces only a $O(\mu^2)$ modification of  \eqref{Boussinesq};  the resulting model therefore keeps the overall $O(\mu^2)$ precision of \eqref{Boussinesq}. Other parameters can be introduced, following an idea of Nwogu \cite{Nwogu1993}, by making a change of of unknown for the velocity. More precisely, we introduce the velocity $V_{\theta,\delta}$ by
\begin{equation}\label{Nwogu}
V_{\theta,\delta}=(1-\mu \theta \frac{1}{3}\Delta)^{-1}(1-\mu \delta \frac{1}{3}\Delta) \ovV
\end{equation}
(this new quantity $V_{\theta,\delta}$ is an approximation of the velocity field at some level line in the fluid domain, see for instance \cite{Lannes2013}). Finally, a fourth parameter $\lambda$ can be introduced by remarking that since we have $\dt \zeta=-\nabla\cdot \ovV_{\theta,\delta}+O(\mu)$ from the first equation, it is possible to add $-\mu \frac{\lambda}{3}(\Delta \dt \zeta -\Delta \nabla\cdot V_{\theta,\delta})$ to the first equation (this is a variant of the BBM trick used above). One finally obtains the so called $abcd$ Boussinesq systems \cite{BonaChenSaut1,BonaChenSaut2,BonaColinLannes},
\begin{equation}\label{abcd}
\begin{cases}
(1-\mu {\bf b}\Delta)\dt \zeta +\nabla\cdot (h V)+ \mu {\bf a}\Delta \nabla\cdot V=0,\\
(1-\mu {\bf d}\Delta)\dt V +\nabla \zeta+\eps (V\cdot \nabla) V +\mu {\bf c} \Delta \nabla \zeta=0,
\end{cases}
\end{equation}
where $h=1+\eps\zeta-\beta b$, $V$ stands for $V_{\theta,\delta}$ and
$$
{\bf a}=-\frac{\theta+\lambda}{3}, \quad {\bf b}=\frac{\delta+\lambda}{3}, \quad {\bf c}=- \frac{\alpha+\delta -1}{3}, \quad {\bf d}=\frac{\alpha+\theta}{3}
$$
(so that ${\bf a }+{\bf b}+{\bf c}+{\bf d}=\frac{1}{3}$). This family of approximations can be extended by changing the structure of the nonlinearity \cite{BonaColinLannes,Chazel}.
\begin{remark}
For the NSW equations, the $(\zeta,Q)$ formulation \eqref{eqNSW} and the $(\zeta,\ovV)$ formulation \eqref{eqNSW} are totally equivalent for smooth solutions, and this will also prove true for the Serre-Green-Naghdi equations. However, such an equivalence does not hold for the Boussinesq systems. We derived the $abcd$ family of Boussinesq systems \eqref{abcd} from the $(\zeta,\ovV)$ formulation \eqref{formzetaQ3_ND_V} of the water waves equation; the same procedure applied to the $(\zeta,Q)$ formulation \eqref{formzetaQ3_ND} leads to slightly different models; we refer to \cite{Filippini2015} for an analysis of the slight differences between these models.
\end{remark}

Let us conclude this small survey on Boussinesq systems by considering what happens if the assumption \eqref{weaktopo} of small topography variations is not made. Since $\beta$ must now be considered as a $O(1)$ rather than $O(\mu)$ quantity, the expansion given above  for the non-hydrostatic term must be revisited. We now get from \S \ref{sectinner} and \S \ref{sectNH} that
\begin{align*}
\frac{1}{h}\frac{1}{\eps}\int_{-1}^{\eps\zeta}\nabla P_{\rm NH}&=\mu {\mathcal T}_b \dt \ovV + O(\mu^2),
\end{align*}
where 
$$
{\mathcal T}_b V=-\frac{1}{3h_b}\nabla\cdot (h_b^3 \nabla\cdot V)+\frac{\beta}{2h_b}\big[\nabla(h_b^2 \nabla b \cdot V)-h_b^2 \nabla b \cdot \nabla\cdot V \big]+\beta^2 \nabla b \nabla b \cdot V
$$
(notice that $h_b{\mathcal T}_b$ is a positive symmetric second order elliptic operator). Plugging this approximation into \eqref{formzetaQ3_ND_V} and dropping the $O(\mu^2)$ terms, one obtains the  Boussinesq-Peregrine \cite{Peregrine67} system
\begin{equation}\label{Peregrine}
\begin{cases}
\dt \zeta +\nabla\cdot (h\ovV)=0,\\
(1+\mu {\mathcal T}_b)\dt \ovV+\nabla\zeta +\eps (\ovV\cdot \nabla )\ovV=0;
\end{cases}
\end{equation}
a generalization of the $abcd$ systems for large topography variations can be derived from \eqref{Peregrine} by adapting the above procedure (see \cite{Lannes2013}).

Let us now describe some of mathematical results and open problems dealing with the Boussinesq models derived in this section.

\subsubsection{The initial value problem for strong solutions}

The (hyperbolic) NSW equations \eqref{eqNSWop} are locally well posed in Sobolev spaces over a $O(1/\eps)$ time scale and this is sharp because shocks occur for such times. The Boussinesq systems being a dispersive perturbation of the NSW equations, one expects that solutions to locally well posed Boussinesq models should exist on a time scale which is at least $O(1/\eps)$. One may indeed expect dispersion to help, but methods based on dispersive estimates only yield an existence time of order $O(1/\sqrt{\eps})$ \cite{linares2012}. A convenient and easy option to reach the $O(1/\eps)$ time scale is to work with $abcd$ systems with a symmetrized nonlinearity \cite{BonaColinLannes,Chazel,Lannes2013}; this $O(1/\eps)$ time scale has finally beed proved for the original $abcd$ systems in a series of papers \cite{Ming2012,Saut2012,Burtea2016,Burtea2016b,Saut2017} for all the linearly well posed $abcd$ systems, except for the case ${\bf b}={\bf d}=0$ and ${\bf a}={\bf c}>0$ which remains open.

The above references (except \cite{Chazel}) deal with a flat topography but, as remarked in \cite{Saut2012}, it is not difficult to extend them to the case of a non flat topography satisfying the assumption \eqref{weaktopo} of small topography variations.  Proving existence over $O(1/\eps)$ times is much more difficult for Boussinesq models with large topography variations (i.e. without assumption \eqref{weaktopo}) such as the Boussinesq-Peregrine model \eqref{Peregrine}. Local well posedness for this system has been proved in \cite{duchene2016} for times $O(1/\max\{\eps,\beta\})$ but the time scale $O(1/\eps)$ has only been proved in \cite{benoit2017} for a variant of the Boussinesq-Peregrine model \eqref{Peregrine} tailored to allow the implementation of low Mach techniques developed in \cite{breschmetivier} for the lake equations.

There are surprisingly few results regarding global existence. This has been proved for the "standard" Boussinesq system \eqref{Boussinesq} in \cite{Schonbek,Amick}, where a weak solution is constructed using a parabolic regularization of the mass conservation equation,  mimicking the hyperbolic theory; the solution is then proved to be regular and unique. For the general $abcd$ systems \eqref{abcd} in dimension $d=1$, global well posedness has been proved in some specific cases using the particular structure of the equations, such as the Bona-Smith system (${\bf a}=-1/3$, ${\bf b}=0$, ${\bf c}=-1/3$, ${\bf d}=1/3$) \cite{bonasmith} and the Hamiltonian cases (${\bf b}={\bf d}>0$, ${\bf a} \leq 0$, ${\bf c}<0$) \cite{BonaChenSaut2}; for this latter system, the two-dimensional case has been treated in \cite{hu2009global}. When ${\bf b}={\bf d}<0$ refined scattering results in the energy space have also been proved \cite{kwak2019,kwak2019b}.

\subsubsection{Initial boundary value problems}

The problem of initial boundary value problems is extremely important for applications to coastal oceanography and several numerical solutions have been proposed, such as the source function method \cite{WKS} for instance; these methods however are not fully satisfactory and require a significant increase of computational time.

 In contrast with hyperbolic systems of equations for which the initial boundary value problem has been intensively studied, there is almost no theoretical result if a dispersive perturbation is added to the equations, as this is the case for the Boussinesq equations. There are only some results concerning the one dimensional case, particular examples of the $abcd$ family \eqref{abcd} and/or  specific boundary conditions:  homogeneous boundary conditions as in \cite{Xue,Adamy,dougalis2009,dougalis2009bis,dougalis2010}, or  \cite{BonaChen,Antonopoulos} for the Bona-Smith system,  where the regularizing dispersive terms of the first equation (due to the fact that ${\bf b}>0$) plays a central role. In \cite{LannesWeynans}, {\it generating} boundary conditions (see  \S \ref{sectIBVPNSW}) have been considered for the Boussinesq-Abott system, a dispersive perturbation of the NSW equations written in $(\zeta,q)$ variables  \eqref{eqNSW}. This latter reference is based on the concept of dispersive boundary layer introduced in \cite{bresch2019waves} for the analysis of a wave-structure interaction problem; it provides a local well-posedness of the initial boundary value problem. However, as the other local well-posedness results given in the above references, the existence time thus obtained is far from the $O(1/\eps)$ time scale which, as seen above, is the relevant one. Reaching such a time-scale is considerably more difficult and requires a precise analysis of the dispersive boundary layer; to this day such an analysis has only been performed in \cite{bresch2019waves}.
 
Another relevant issue is the convergence towards the initial boundary value problem for the NSW equations as the dispersive (or shallowness) parameter $\mu$ tends to zero; here again, the analysis of the dispersive boundary layer should be a key point (such a convergence has been proved in \cite{bresch2019waves}).

For  {\it transparent} boundary conditions  (which allow waves to cross the boundary of the computational domain without reflexion, see  \S \ref{sectIBVPNSW}), the situation looks even more complicated. There are some results for the linear problem: for scalar equations (linear KdV or BBM for instance) \cite{BMN,BNS} and for the linearization of \eqref{Boussinesq} around the rest state \cite{KazakovaNoble}. The nonlinear case remains open.

\subsection{Second order approximation: the Serre-Green-Naghdi equations and variants}\label{sectSGN}

The Serre-Green-Naghdi (SGN) equations are an approximation of order $O(\mu^2)$ of the water waves equations \eqref{formzetaQ3_ND} in the sense that the terms of order $O(\mu)$ that were neglected in the nonlinear shallow water equations are kept, and only terms of order $O(\mu^2)$ are dropped.  The precision of this model is therefore the same as the precision of the Boussinesq models investigated in \S \ref{sectBoussinesq}, but they have a wider range of application since they do not require the weak nonlinearity assumption \eqref{weakNL} nor the weak topography assumption \eqref{weaktopo}. The price to pay is that the $O(\eps\mu)$ and $O(\beta\mu)$ terms must be kept in the model, making it more complicated than  the Boussinesq systems \eqref{abcd}. {\it For the sake of clarity, we consider here the case of a flat bottom only ($b=0$) and refer to Appendix \ref{Appbottom} for the equations with a non flat topography.}

The  "turbulent" and non-hydrostatic terms in \eqref{formzetaQ3_ND} can be expended as follows, following the results of \S \ref{sectinner} and \S \ref{sectNH}, 
\begin{align*}
\nabla\cdot \Rey&=O(\mu^2) \\
 \frac{1}{\eps}\int_{-1}^{\eps\zeta}\nabla P_{\rm NH}&=
 \mu h  {\mathcal T}\big[\dt \ovV+\eps \nabla\cdot \big(h  \ovV\otimes \ovV\big)\big] +\mu \eps h{\mathcal Q}_1\big(\zeta, \ovV\big)
 +O(\mu^2)
\end{align*}
where
\begin{align*}
{\mathcal T}V&=-\frac{1}{3h}\nabla\big( h^3 \nabla\cdot V\big),\\
{\mathcal Q_1}(\zeta,V)&=\frac{2}{3h}\nabla \big[ h^3 \big( \dx V\cdot \dy V^\perp +(\nabla\cdot V)^2 \big)\big].
\end{align*}
Therefore, even in a fully nonlinear regime and with the higher $O(\mu^2)$ precision of the SGN equations, the contribution of the "turbulent" term $\eps \nabla\cdot \Rey$ remains too small to be relevant and can be neglected. All the additional terms of the SGN equations with respect to the NSW equation are therefore due to the non hydrostatic pressure. Plugging the above expansions into \eqref{formzetaQ3_ND} and dropping the $O(\mu^2)$ terms, one obtains the SGN equations,
\begin{equation}\label{SGN}
\begin{cases}
\dt \zeta +\nabla\cdot Q=0,\\
(1+\mu {\bf T})\big[ \dt Q +\eps \nabla\cdot \big( \frac{1}{h} Q\otimes Q \big) \big] +h\nabla\zeta +\eps \mu h {\mathcal Q}_1 (h ,\frac{Q}{h})=0,
\end{cases}
\end{equation}
where ${\bf T}=h {\mathcal T}\frac{1}{h}$. We refer to \eqref{SGNtopo} for the generalization of these equations when the topography is not flat. These equations are actually known under different names, such as  Serre \cite{Serre,SuGardner}, Green-Naghdi \cite{GreenNaghdi,Kim}, 
or fully nonlinear Boussinesq \cite{Wei}.
\begin{remark}
Replacing $Q=h\ovV$ in \eqref{SGN}, one obtains the following equivalent formulation (as far as smooth solutions are concerned) in $(\zeta,\ovV)$,
\begin{equation}\label{SGN_V}
\begin{cases}
\dt \zeta +\nabla\cdot (h\ovV)=0,\\
(1+\mu {\mathcal T})\big[ \dt \ovV +\eps \nabla\cdot \big( {h} \ovV\otimes \ovV \big) \big] +\nabla\zeta + \eps \mu {\mathcal Q}_1 (h ,\ovV)=0.
\end{cases}
\end{equation}
Neglecting the $O(\eps\mu)$ terms in this system, one recovers of course the Boussinesq equations \eqref{Boussinesq}.
\end{remark}

As in the weakly nonlinear regime with the Boussinesq equations, it is possible to derive formally equivalent systems using similar procedures (the "BBM-trick" and a change of unknown for the velocity); a family of SGN equations generalizing the $abcd$ Boussinesq systems \eqref{abcd} can be derived \cite{Chazel2011,Lannes2013}. In a similar vein, it is possible to derive equivalent systems (i.e. systems that differ formally from \eqref{SGN} by $O(\mu^2)$ terms) that have a better mathematical structure \cite{Israwi2010b,Israwi2011} or that are more adapted to numerical computations \cite{LannesMarche}.

\subsubsection{Known results and open problems}

We review here several known results and open problems about the SGN equations.

\noindent
- {\it Initial value problems and singularity formation.} There  is local conservation of energy for the SGN equations,
\begin{equation}\label{locNRJSGN}
\dt {\mathfrak e}_{\rm SGN}+ \nabla\cdot {\mathfrak F}_{\rm SGN}=0,
\end{equation}
with energy density and energy flux given (when the bottom is flat, see \cite{CastroLannes2} for the generalization to non flat bottoms) by
\begin{align}
\label{SGNdens}
{\mathfrak e}_{\rm SGN}&= \frac{1}{2} \big[ \zeta^2 + h \vert \overline{V}\vert^2  + \mu \frac{1}{6} h^3 \abs{\nabla\cdot \ovV}\big],\\
\label{SGNflux}
{\mathfrak F}_{\rm SGN}&= \big[ \zeta  +\frac{1}{2} \vert \overline{V}\vert^2  + \eps \mu \frac{1}{6} h^2 \abs{\nabla\cdot \ovV}^2-\mu\frac{1}{3} h (\dt +\eps \ovV\cdot \nabla)(h\nabla\cdot \ovV)\big] h\ovV
\end{align}
Integrating over $\R^d$, this yields conservation of the mechanical energy,
$$
\frac{d}{dt}E_{\rm SGN}=0
\quad\mbox{ with }\quad
E_{\rm SGN}=\int_{\R^d}{\mathfrak e}_{\rm SGN}
$$
In addition to the control of the $L^2$-norm of $(\zeta,\overline{V})$ that we had for the NSW equations, we now have a control of $\sqrt{\mu}\nabla\cdot \overline{V}$ provided that the non-vanishing conditions \eqref{nonvanish} is satisfied. This allows one to control the extra nonlinear terms $\eps\mu{\mathcal Q}_1(h,\ovV)$ in \eqref{SGN_V} which has therefore a semi-linear structure. Local existence was proved in \cite{Li} for small times, and in \cite{Alvarez2,Israwi2011,fujiwara2015,duchene2016} for times of order $O(1/\eps)$, uniformly with respect to $\mu \in (0,1)$. Another interesting fact shown in \cite{Duchene} is that smooth solutions to the SGN equations can be obtained as relaxation limits of an augmented quasilinear system of conservation laws proposed in \cite{favrie2017}.

Contrary to the NSW equations, the SGN equations contain third order dispersive term that play a regularizing role. The question of global well posedness therefore becomes relevant, and one could conjecture in dimension $d=1$ a scenario similar to the one observed for the Camassa-Holm equation which is somehow the "unidirectional version" of the SGN equations (see below), namely: one has global existence for some data and wave breaking for others (i.e., the $L^\infty$-norm is bounded but the derivative of the velocity and/or the surface elevation blows up in finite time). This scenario is supported by numerical computations showing that there exist shocks relating a constant state to a periodic wave train, and that, at least numerically, such shocks can be dynamically obtained  \cite{GavrilyukNkonga}.

\medbreak

\noindent
- {\it Initial boundary value problems.} With respect to the NSW equations, the new dispersive and nonlinear terms of the SGN equation render the analysis much more complicated in the presence of a boundary. The case of a wall boundary condition $\overline{V}\cdot {\bf n}=0$ is the simplest one since the boundary terms in the energy estimates vanish. In the particular one dimensional case $d=1$, the result can be adapted from \cite{LannesMetivier} but considerable simplifications could be made using the non-vanishing depth condition \eqref{nonvanish}. Even in dimension $d=1$, other types of boundary conditions (e.g. generating and transparent) are much more complex and remain open. The case of transparent boundary conditions for the linearized SGN equations around the rest state (which are actually the same as the linearized Boussinesq equations around the rest state) has been addressed in \cite{KazakovaNoble}.

In view of the difficulty of the nonlinear case, an alternative has been proposed, consisting in  implementing a perfectly matched layer (PML) approach for a hyperbolic relaxation of the Green-Naghdi equations \cite{Kazakova}. This approach can also be used to deal with generating boundary conditions  but the size of the layer in which the PML approach is implemented is typically of two wavelength, which for applications to coastal oceanography can typically represent an increase of $100\%$ of the computational domain. Other methods such as the source function method \cite{WKS} also require a significant increase of computational time.
\medbreak
\noindent
- {\it Free boundary problems: the shoreline problem.}  As for the NSW equation, it is natural to remove the non-vanishing depth condition \eqref{nonvanish} and to consider the shoreline problem (see above). This problem has been solved in dimension $d=1$ in \cite{LannesMetivier}, but the two dimensional case remains open.

\subsection{Multi-layer hydrostatic and non-hydrostatic models}\label{sectmultilayer}

We have already mentioned the interest in coastal application for Boussinesq or SGN models with an improved linear dispersion. With this goal in mind, higher order Boussinesq and SGN models that are precise up to $O(\mu^{k})$ ($k\geq 3$) terms have been proposed (see for instance \cite{BonaChenSaut1,matsuno}); such models however contain high order derivatives that make them difficult to implement numerically. A more recent alternative to these models, initiated in \cite{casulli} (see also \cite{Stelling:2003fk,Ma,audusse2011} and references below),  are the so called non hydrostatic models that resolve the vertical flow structure in the governing equations. The key step in this approach is a vertical discretization of the non-hydrostatic pressure. As remarked for instance  in \cite{bai2013dispersion}, it turns out that this numerical approach can be interpreted through the decomposition of the fluid domain in $N$ artificial layers of fluid. We propose in this section a systematic derivation of multi-layer NSW, Boussinesq and SGN type models and make the link with various multi-layer models that can be found in the literature.

We first derive in \S \ref{sectEulerav} a multi-layer averaged Euler system deduced from the original Euler equation through vertical averaging on $N$ different horizontal layers of fluid. As in the single layer case, it is necessary to analyze the structure of the velocity and pressure fields in each layer in order to derive simpler asymptotic models; this structure is investigated in \S \ref{procinnerlayer}. At first order, a multi-layer hydrostatic (or NSW) model is derived in \S \ref{sectmultiNSW}; similar multi-layer generalizations of the Boussinesq and SGN models are then derived in \S \ref{sectmultiBouss} and \S \ref{sectmultiSGN}. 
\medbreak

\noindent
{\bf N.B.} {\it Throughout this section we directly work with the dimensionless variables introduced in \S \ref{sectnondim}.}

\subsubsection{Multi-layer averaging of the Euler equations}\label{sectEulerav}

The fluid domain is decomposed into $N$ horizontal layers ${\mathcal L}_j$, $1\leq j \leq N$,
with
$$
{\mathcal L}_j=\{ (X,z)\in \R^{d+1},  z_{j-1/2}(t,X)<z< z_{j+1/2}(t,X)\},
$$
where the functions $z_{j+1/2}$, $0\leq j\leq N$ are the boundaries of these layers. One has $z_{1/2}=-1+\beta b$ and $z_{N+1/2}=\eps \zeta$, but otherwise these functions have no physical meaning, and can be chosen in several ways (the interior interfaces can be time independent or related to the evolution of the free surface for instance). As in \cite{fernandez2014}, we take them of the form
$$
z_{j+1/2}=z_{1/2}+(\sum_{k=1}^{j} l_k) h, \quad\mbox{ with } l_1,\dots,l_N\in [0,1] \quad \mbox{ and }\quad \sum_{k=1}^N l_k=1,
$$
where we recall that $h=1+\eps\zeta-\beta b$. We denote by $V_j$ and $w_j$ the horizontal and vertical velocities in  ${\mathcal L}_j$ and
$$
h_j=z_{j+1/2}-z_{j-1/2}(=l_j h), \qquad  \ovV_j =\frac{1}{h_j}\int_{z_{j-1+2}}^{z_{j+1/2}} V, \quad \mbox{ and }\quad
V_j^*=V_j-\ovV_j;
$$
we also denote by $V_j^+$ and $w_j^+$ (resp. $V_j^-$ and $w_j^-$) the traces of $V_j$ and $w_j$ on the upper boundary $\{z=z_{j+1/2}\}$ of ${\mathcal L}_j$ (resp. its lower boundary $\{z=z_{j-1/2}\}$).

If the velocity field $\bU$ is irrotational and incompressible in the whole fluid domain $\Omega_t$, then the $\ovV_j$ are not independent variables and their evolution is slaved to the evolution of $\ovV$. For instance, in the case of a flat surface and a flat bottom one readily computes $\ovV_j=m_l(D)\ovV$, where $m_l(D)$ is the Fourier multiplier of symbol 
$$
m_l(\xi)=\frac{\sinh(\sqrt{\mu}(z_{j+1/2}+1)\vert\xi\vert)-\sinh(\sqrt{\mu}(z_{j-1/2}+1)\vert\xi\vert)}{\sinh(\sqrt{\mu}\vert\xi\vert)}.
$$
It is therefore possible to write the equations in terms of $\zeta$ and $\ovV_{j_0}$ (for some $1\leq j_0\leq N$) instead of $\zeta$ and $\ovV$. The corresponding Boussinesq models may be of interest because of their dispersive properties \cite{lynett2004linear,floarticle}.

The idea behind multi-layer models is different. Such models must be seen rather as a discretization of the $d+1$-Euler equations. This discretization leads to a piecewise approximation of the velocity field built on the layer averaged velocities $\ovV_j$. These quantities are then set to evolve independently according to an approximation of the layer averaged Euler equations derived below.

This piecewise approximation of the velocity field is assumed to be incompressible in the whole fluid domain, and irrotational in each layer. The incompressibility condition  imposes  continuity of the normal velocity at the interfaces, namely, in dimensionless variables,
\begin{equation}\label{contnorm}
\forall 1\leq j\leq N-1, \qquad w_{j+1}^--\mu V_{j+1}^-\cdot  \nabla z_{j+1/2}=w_{j}^+-\mu V_{j}^+\cdot  \nabla z_{j+1/2}.
\end{equation}

Let us now proceed to derive evolution equations on $\zeta$ and $(\ovV_j)_{1\leq j\leq N}$. We note first that the equation for the conservation of mass can be equivalently written 
$$
\dt \zeta+\nabla\cdot \big( \sum_{j=1}^N (l_j h \ovV_j)=0.
$$
Mimicking the computations performed in the single layer case, an equation on $\ovV_j$ can also be obtained by averaging the horizontal component of the Euler equation \eqref{Euler1} in ${\mathcal L}_j$, one gets, in dimensionless form,
\begin{align*}
\dt (h_j \ovV_j)+\eps \nabla\cdot (h_j \ovV_j\otimes \ovV_j)+ h_j \nabla \zeta+\eps \nabla\cdot {\bf R}_j&+\frac{1}{\eps}\int_{z_{j-1/2}}^{z_{j+1/2}} \nabla P_{\rm NH}\\
&=
\eps \big( V_j^+ G_{j+1/2} -V_j^- G_{j-1/2} \big) ,
\end{align*}
where $G_{j+1/2}$ denotes the mass transfer from $z_{j-1/2}$ to $z_{j+1/2}$,
\begin{align}
\nonumber
G_{j+1/2}&=\frac{1}{\eps }\dt z_{j+1/2}-\frac{1}{\mu} \big( w_{j+1}^--\mu V_{j+1}^-\cdot  \nabla z_{j+1/2}\big)\\
\label{defGj}
&=\frac{1}{\eps }\dt z_{j+1/2}-\frac{1}{\mu} \big( w_{j}^+-\mu V_{j}^+\cdot  \nabla z_{j+1/2}\big)
\end{align}
(in particular, as there is no mass transfer across the free surface and the bottom, one has $G_{1/2}=G_{N+1/2}=0$). In this equation, the non hydrostatic pressure $P_{\rm NH}$ is defined as in \eqref{defRPNH}, while the "turbulent" tensor in the layer ${\mathcal L}_j$ is defined as 
$$
{\bf R}_j=\int_{z_{j-1/2}}^{z_{j+1/2}} V_j^*\otimes V_j^*.
$$
Remarking further that
$$
\dt (h_j \ovV_j)+\eps \nabla\cdot (h_j \ovV_j\otimes \ovV_j)=h_j\dt \ovV_j +\eps h_j \ovV_j \cdot \nabla \ovV_j+\eps (G_{j+1/2}-G_{j-1/2})\ovV_j
$$
we have equivalently
\begin{align*}
l_j \dt  \ovV_j+\eps l_j \ovV_j\nabla\cdot \ovV_j +  l_j \nabla \zeta+&\frac{\eps }{h}\nabla\cdot {\bf R}_j+\frac{1}{\eps h}\int_{z_{j-1/2}}^{z_{j+1/2}} \nabla P_{\rm NH}
= {\mathbf S}_j
\end{align*}
with the source term ${\mathbf S}_j$ given by
\begin{equation}\label{defSj}
{\mathbf S}_j=-\eps \frac{1}{h}\big[ G_{j+1/2}(\ovV_j-V_j^+) -G_{j-1/2} (\ovV_j-V_j^-) \big].
\end{equation}

We have therefore derived a system of $(N+1)$ equations generalizing the equations \eqref{formzetaQ3_ND_V} derived in the case of a single layer,
\begin{equation}\label{layerav}
\begin{cases}
\dt \zeta+\nabla\cdot \big( \sum_{j=1}^N (l_j h \ovV_j)=0,\\
l_j \dt  \ovV_j+\eps l_j \ovV_j\nabla\cdot \ovV_j +  l_j \nabla \zeta+\frac{\eps }{h}\nabla\cdot {\bf R}_j+\frac{1}{\eps h}\int_{z_{j-1/2}}^{z_{j+1/2}} \nabla P_{\rm NH}
= {\mathbf S}_j.
\end{cases}
\end{equation}
\begin{remark}
{\bf i}. The assumption made above that the fluid is irrotational in each layer has not been used to derive \eqref{layerav}. It would be necessary however to establish a generalization of Proposition \ref{propclosed} showing that \eqref{layerav} is a closed system of equations, namely: given $\zeta$ and $(\ovV_j)_{1\leq j\leq N}$, there is a unique velocity field ${\bf U}$, incompressible in the whole fluid domain (and therefore with continuous normal velocity across the interfaces) such that for all $1\leq j\leq N$, ${\bf U}_j$ is irrotational in ${\mathcal L}_j$ and such that $\int_{z_{j-1/2}}^{z_{j+1/2}} V_j=\ovV_j$.\\
{\bf ii.} Multi-layer models relies on an incompressible approximation of the velocity field which is irrotational in each layer. It is however not irrotational in the whole fluid domain. Indeed, due to the possible discontinuities of the tangential velocity field, there is vorticity concentrated on the interfaces $\{z=z_{j+1/2}\}$. These "vortex sheets" must be understood as a discretization error and have no physical meaning (for physical vortex sheets for instance, there cannot be exchange terms across the interfaces \cite{Bardos}).\\
{\bf iii.} The assumption of irrotationality in each layer is actually not necessary; it is just required that the vorticity is small enough for its influence (e.g. the turbulent tensor) to be neglected. This assumption is (at least implicitly) made in the derivation of all multi-layer models that can be found in the literature. Of course a multi-layer generalization of the rotational models derived in Section \ref{sectrot} below is certainly possible. 
\end{remark}
As in the case of a single layer, we need to investigate the structure of the velocity field and of the non-hydrostatic pressure to derive simpler asymptotic models from this set of equations.

\subsubsection{The structure of the velocity and pressure fields}\label{procinnerlayer}

The assumption \eqref{contnorm} on the continuity of the normal velocities at the interfaces ensures the global incompressibility of the velocity field. Hence, as in the single layer case, we obtain that
$$
w=-\mu \nabla\cdot \int_{-1+\beta b}^z V,
$$
so that, in the layer ${\mathcal L}_j$, we have
$$
w_j=-\mu \nabla\cdot \big( (z-z_{j-1/2}) \ovV_j\big) - \mu \nabla\cdot \big(\sum_{k=1}^{j-1} h_k \ovV_k \big)- \mu \nabla\cdot \int_{z_{j-1/2}}^z V_j^*,
$$
where we recall that $\ovV_j$ is the vertical average of $V_j$ over ${\mathcal L}_j$ and $V_j^*=V_j-\ovV_j$. Since the velocity field is irrotational in each fluid layer, we deduce that
$$
V_j^*=-\Big( \int_{z_{j-1/2}}^z \nabla w_j \Big)^*.
$$
Proceeding as in \S \ref{sectinner} for the single layer case and defining the operators ${\bf T}_j$ and ${\bf T}_j^*$ acting  on $\R^d$ valued functions $W_j$ defined on ${\mathcal L}_j$ as
$$
{\bf T}_j W_j =\int_z^{z_{j+1/2}}\nabla\nabla\cdot \int_{z_{j-1/2}}^{z'} W_j
\quad \mbox{ and }\quad
{\bf T}^*_j W_j=\big({\bf T}_j W_j  \big)^*,
$$
we get that
\begin{equation}
\label{expvellayer}
\begin{cases}
\dsp V_j=\ovV_j+\mu {\bf T}_j^* \ovV_j+ \mu \big( \frac{1}{2}(z_{j+1/2}+z_{j-1/2})-z\big)\nabla\nabla\cdot \sum_{k=1}^{j-1}h_k\ovV_k+O(\mu^2),\\
\dsp w_j=-\mu \nabla\cdot \big( (z-z_{j-1/2}) \ovV_j\big) - \mu \nabla\cdot \big(\sum_{k=1}^{j-1} h_k \ovV_k \big)+O(\mu^2)
\end{cases}
\end{equation}
(since $\ovV_j$ does not depend on $z$, it is possible to derive an explicit expression for ${\bf T}_j^* \ovV_j$; it is obtained by replacing $\eps\zeta$ by $z_{j+1/2}$ and $\beta b$ by ${z_{j-1/2}}$ in \eqref{expinnerveltopo}).

To describe the multi-layer structure of the pressure field, it is convenient to define first the values $P_{\rm NH}^{(j+1/2)}$ of the non-hydrostatic pressure at the interfaces; taking into account the definition \eqref{defRPNH} of the non-hydrostatic pressure, these quantities can be determined through a downward induction,
$$
P_{\rm NH}^{N+1/2}=0, \qquad P_{\rm NH}^{j-1/2}=P_{\rm NH}^{j+1/2}
+\eps \int_{z_j-1/2}^{z_{j+1/2}} \big( \dt w_j+\eps V_j\cdot \nabla w_j +\frac{\eps}{\mu} w_j \dz w_j \big) ,
$$
for $j$ varying from $N$ to $1$. We then have, if $z$ is in the layer ${\mathcal L}_j$,
$$
\frac{1}{\eps} P_{\rm NH}(z)=\int_z^{z_{j+1/2}}  \big( \dt w_j+\eps V_j\cdot \nabla w_j +\frac{\eps}{\mu} w_j \dz w_j \big) +\frac{1}{\eps}P_{\rm NH}^{(j+1/2)}.
$$
Replacing in these formulas $V_j$ by $\ovV_j$ and $w_j$ by its approximation provided by \eqref{expvellayer} provides an expression for the non-hydrostatic pressure in terms of $h$ and $(\ovV)_j$, in the spirit of \eqref{expP} and \eqref{expPtopo}. The resulting expressions are quite complicated and omitted here.

\subsubsection{Multi-layer hydrostatic models}\label{sectmultiNSW}

Proceeding as in \eqref{sectinner}, one straightforwardly gets from \eqref{expvellayer} that
$$
V_j=\ovv_j+O(\mu), \qquad V_j^*=O(\mu), \qquad \ovV_j-V_j^\pm=O(\mu),\qquad w_j=O(\mu),
$$
from which one deduces that
\begin{equation}\label{sizetermslayer}
{\bf R}_j=O(\mu^2) ,\qquad \frac{1}{\eps h}\int_{z_{j-1/2}}^{z_{j+1/2}} \nabla P_{\rm NH}=O(\mu) \quad\mbox{ and }\quad {\mathbf S}_j=O(\eps\mu).
\end{equation}
Neglecting the $O(\mu)$ terms in \eqref{layerav} as for the derivation of the single layer NSW equations \eqref{eqNSW} we obtain the (hyperbolic) multi-layer NSW equations
\begin{equation}\label{multilayerNSW}
\begin{cases}
\dt \zeta +\nabla\cdot \big( \sum_{j=1}^N  l_j h \ovV_j )=0,\\
l_j \dt  \ovV_j+\eps l_j \ovV_j\cdot \nabla  \ovV_j+  l_j \nabla \zeta=0;
\end{cases}
\end{equation}
note in particular that the equations on $\ovV_j$ and $\ovV_k$ ($k\neq j$) are decoupled (or more accuretaly, they are only coupled through the presence of the term in $l_j\nabla \zeta$). This is because, at the precision of the model, $\ovV_j=V_j^\pm$ so that the source term ${\mathbf S}_j$ can be neglected together with the turbulent and non hydrostatic terms. This would not be the case if a different choice had been made for $V_j^\pm$. Other possibilities can for instance been found in \cite{audusse2011} where $V_j^+$ is taken equal to either $\ovV_{j+1}$ or $\ovV_j$ according to an upwind scheme based on the sign of $G_{i+1/2}$; in \cite{fernandez2014}, where a model with viscosity is considered, an analysis of the viscous terms leads to $V_j^+=\frac{1}{2}(\ovV_j+\ovV_{j+1})$.

\subsubsection{Multi-layer non hydrostatic models in the Boussinesq regime}\label{sectmultiBouss}

A multi-layer Boussinesq-type model can also been derived from \eqref{layerav} under the weak nonlinearity and small topography assumptions \eqref{weakNL} and \eqref{weaktopo}. Keeping the $O(\mu)$ terms but neglecting as in \S \ref{sectBoussinesq} the terms of size $O(\mu^2)$ (and therefore also $O(\eps\mu)$ terms under the weak nonlinearity assumption), one readily gets from \eqref{sizetermslayer} that the turbulent and exchange terms ${\bf R}_j$ and ${\bf S}_j$ can still be neglected, but that there is a contribution from the non-hydrostatic pressure
\begin{equation}\label{PNHBouss}
\frac{1}{\eps h} \int_{z_{j-1/2}}^{z_{j+1/2}}\nabla P_{\rm NH}=-\mu \sum_{k=1}^N{T}_{\rm jk} \Delta \dt \ovV_j+O(\mu^2) 
\end{equation}
where the symmetric matrix $({T}_{jk})_{1\leq j,k\leq N}$ is defined by
$$
{T}_{jk}=-\frac{1}{6}l_j^3 \delta_{jk}+l_j l_k \big( \frac{1}{2} l_{j\vee k}+\sum_{m=j\vee k+1}^N l_m \big),
$$
with  $\delta_{jk}$ standing for Kronecker's symbol and $j\vee k=\max\{j,k\}$. This leads to the following multi-layer Boussinesq equations
\begin{equation}\label{multilayerBouss}
\begin{cases}
\dsp \dt \zeta +\nabla\cdot \big( \sum_{j=1}^N  l_j h \ovV_j )=0,\\
\dsp l_j \dt  \ovV_j-\mu \sum_{k=1}^N T_{jk}\Delta \dt \ovV_k+\eps l_j \ovV_j\cdot \nabla  \ovV_j+  l_j \nabla \zeta=0,
\end{cases}
\qquad (1\leq j\leq N);
\end{equation}
contrary to what we saw for the multi-layer NSW equations \eqref{multilayerNSW}, the equations on the velocities  $\ovV_j$  ($1\leq j\leq N$) are now coupled through the dispersive terms. One readily checks that \eqref{multilayerBouss} coincides with \eqref{Boussinesq} in the case of a single layer. It is certainly more convenient to see the $N$ equations on the $\ovV_j$ as a single evolution equation in ${\bf V}$ where
$$
\overline{\bf V}^{\rm T}=(\ovV_1^{\rm T},\dots,\ovV_N^{\rm T}).
$$
In dimension $d=1$ (the generalization to $d=2$ is straightforward), this vectorial equation can be written
$$
({\rm diag}(l_j)-\mu T\dx^2)\dt \overline{\bf V}+\eps {\rm diag}(l_j\ovV_j)  \dx \overline{\bf V}+{\bf l}\dx \zeta=0
$$
where ${\rm diag}(u_j)$ stands for the diagonal matrix with entries $u_1,\dots, u_N$ and ${\bf l}^{\rm T}=(l_1,\dots,l_N)$. The resolution of the multi-layer Boussinesq system requires therefore the inversion of an $N\times N$ matricial second order differential operator (as opposed to the scalar operator $1-\mu/3\dx^2$ in the case of a single layer). This is numerically more involved, but the interest of this multi-layer model is that its linear dispersion can approximate the dispersion relation of the full water waves equations with a very good accuracy. We refer to \cite{fernandez} for an analysis of this dispersion relation (the model studied there differs from \eqref{multilayerBouss} but only in the nonlinear terms, which do not affect the linear dispersion relation); apart from this, the mathematical analysis (see in particular the open problems mentioned in \S \ref{sectBoussinesq}) and the numerical implementation of \eqref{multilayerBouss} remain to be done.
\begin{remark}
When $N=2$, the matrix $T$ is given by
$$
T=\left(\begin{array}{cc} \frac{1}{3}l_1^3+l_1^2l_2 & \frac{1}{2}l_1l_2^2\\ \frac{1}{2}l_1l_2^2 & \frac{1}{3}l_2^2\end{array}\right)
$$
with $l_2=1-l_1$ and $0 < l_1 < 1$. The corresponding dispersion relation can be computed and $l_1$ can be chosen to match the dispersion relation of the water waves equations in the best way possible. It is however not possible to choose $l_1$ in such a way that the Taylor expansions of both expressions coincides at order $O(\mu^2)$ while this is possible for some of the $abcd$ systems \eqref{abcd} and for the Isobe-Kakinuma model (see \S \ref{sectIsobe} below, and more specifically Remark \ref{remcompareIK}).
\end{remark}

\subsubsection{Multi-layer fully nonlinear non hydrostatic models} \label{sectmultiSGN}

A multi-layer SGN-type model generalizing \eqref{multilayerNSW} can be derived by keeping the $O(\mu)$ terms and dropping the $O(\mu^2)$ terms in \eqref{layerav}. Working at this precision, and without making any weak nonlinearity assumption, the turbulent term can still be neglected, but it is necessary to keep the $O(\eps\mu)$ exchange term ${\bf S}_j$ and therefore to study the mass exchange coefficients $G_{j+1/2}$ defined in \eqref{defGj} as well as the vertical deviations of the horizontal velocity $\ovV_j-V_j^\pm$. 
\begin{lemma}
The following identities hold
\begin{align*}
G_{j+1/2}-\frac{1}{\eps }\dt z_{k+1/2}&= \nabla\cdot \sum_{k=1}^j h_k \ovV_k,\\
G_{j+1/2}&= \sum_{k=1}^j \big[\nabla\cdot (h_k \ovV_k)-l_k \nabla\cdot \big( \sum_{m=1}^N \nabla\cdot (l_m h \ovV_m) \big],\\
\ovV_j-\ovV_j^\pm&=-\mu {\bf T}^*_j \ovV_j\pm \mu\frac{h_j}{2}\nabla\nabla\cdot \sum_{k=1}^{j-1} h_k \ovV_k+O(\mu^2).
\end{align*}
\end{lemma}
\begin{proof}
Integrating vertically the incompressibility condition \eqref{Euler2} (in dimensionless form) from $z_{k-1/2}$ to $z_{k+1/2}$ yields
$$
\dt h_k +\eps \nabla\cdot  ( h_k \ovV_k)=\eps \big( G_{k+1/2}-G_{k-1/2}\big).
$$
Summing these identities from $k=1$ to $j$ yields the first identity of the lemma. For the second one, we recall that $h_k=l_k h$ and remark that $\dt h_k=l_k \dt h=-\eps l_k \nabla\cdot \sum_{m=1}^N h_m \ovV_m$. Finally, the last identity is a direct consequence of \eqref{expvellayer}.
\end{proof}
This lemma allows one to replace the exchange term ${\bf S}_j$ in \eqref{layerav} by a differential polynomial of $h$ and $(\ovV_j)_{1\leq j \leq N}$. For the non-hydrostatic term we can use the results of \S \ref{procinnerlayer} to obtain a fully nonlinear generalization of \eqref{PNHBouss}. This leads to a multi-layer generalization of the SGN equations \eqref{SGN_V} which to our knowledge has not been studied or numerically implemented yet. Multi-layer non-hydrostatic models goes back at least to the works \cite{casulli,Stelling:2003fk} that led to the Swash simulation code, and they have also been used for the Nhwave simulation code \cite{Ma}. In these references, the framework is slightly different since the multi-layer aspect appears through a vertical discretization of the velocity field;  the vertical non-hydrostatic pressure gradient is for instance discretized using the Keller box scheme in \cite{Stelling:2003fk}. The link between these numerical approaches and multi-layer modeling was shown in \cite{bai2013dispersion,bai2015dispersion} where it was shown that  this discretization does not provide the correct dispersion relation when applied in the case of a single layer (one gets a coefficient $1/4$ instead of $1/3$ in \eqref{dispBouss}), but that this drawback gets compensated by the increase of the number of layers. More recently, the multi-layer non hydrostatic approach described above has been used in \cite{fernandez} as a generalization of previous works on multi-layer hydrostatic models (e.g. \cite{audusse2011,fernandez2014}). The authors propose a fully nonlinear model that has the advantage of reducing to the correct SGN model in the particular case of a single layer. The model proposed in the lines above also coincides with the SGN model in the case of a single layer. For multiple layers, its linear part is the same as for the model of  \cite{fernandez} but there are differences in the nonlinear terms essentially due to the fact that instead of using $V_j^+$ (for instance) in \eqref{defSj}, the authors use an interface velocity $\widetilde{V}_{j+1/2}$ defined as a linear combination of $\ovV_j$ and $\ovV_{j+1}$. It could of course be of interest to investigate these differences, both numerically and theoretically (by controlling the error with the full Euler equations for instance).

\subsection{The Isobe-Kakinuma model}\label{sectIsobe}

Another set of equations providing a high order approximation to the water waves equations in shallow water is the Isobe-Kakinuma model. Isobe \cite{isobe1995} and \cite{kakinuma2001} started from Luke's variational formulation of the water waves equations (see Remark \ref{remLuke}) but replaced the velocity potential $\Phi$ by an approximation $\Phi^{\rm app}$ in the expression for the Lagrangian. In dimensionless form, this approximation is taken of the form
$$
\Phi^{\rm app}=\sum_{i=0}^N \Psi_i(z;\beta b)\phi_i(X,t)
$$
where the functions $\Psi_i$ depend on the vertical coordinate $z$ and on the bottom parametrization $b$, while the functions $\phi_j$ are unknown quantities to be determined. In order to do so, an approximate Lagrangian density is introduced
$$
{\mathcal L}_{\rm app}(\phi_0,\dots,\phi_N,\zeta):={\mathcal L}_{\rm Luke}(\Phi^{\rm app},\zeta),
$$
where we recall that ${\mathcal L}_{\rm Luke}$ is defined in \eqref{eqLuke}. The Isobe-Kaninuma model is obtained by writing the Euler-Lagrange equation corresponding to this approximated Lagrangian. The resulting model depends therefore on the choice of the basis functions ${\Psi}_i$.  A natural and commonly made choice is to take
$$
\Psi_i(z;\beta b)=\mu^{p_i/2}(z+1-\beta b)^{p_i}
$$
with $p_i=i$ in the general case and $p_i=2i$ in the case of a flat bottom ($b=0$). The resulting Isobe-Kakinuma (IK) model is a set of $N+2$ equations; in the case of a flat bottom (see \cite{nemoto2018} for the generalization to non flat bottoms), it is given (in variables with dimensions) by
\begin{equation}\label{eqIK}
\begin{cases}
 h^{2i}\dt \zeta
+\sum_{j=0}^N \mu^j \Big[ \nabla\cdot \Big( \frac{1}{2(i+j) +1}h^{2(i+j)+1}\nabla\phi_j\Big)- \frac{4ij}{2(i+j)-1}h^{2(i+j)-1}\phi_j\Big]=0, \\
\hspace{9.5cm}(0\leq i\leq N),\\
 \sum_{i=0}^N h^{2i}\dt \phi_i +\zeta +\frac{\eps}{2}\sum_{i,j=0}^N \mu^{i+j}\Big(h^{2(i+j)}\nabla\phi_i\cdot \nabla\phi_j+\frac{4ij}{\mu^2} h^{2(i+j)-2}\phi_i\phi_j\Big)=0.
\end{cases}
\end{equation}
\begin{remark}
The approximate velocity field associated to the IK approximation is ${\bf U}^{\rm app}=\nabla_{X,z} \Phi^{\rm app}$. By construction, it is irrotational, as opposed to the approximate velocity field used for the multi-layer models of \S \ref{sectmultilayer}. Conversely, while this latter was by construction incompressible, ${\bf U}^{\rm app}$ is only appoximately incompressible. Actually, the $N+1$ evolutions equations on $\dt \zeta$ can be viewed as conditions to impose that ${\bf U}^{\rm app}$ is approximately divergence free in the following sense
$$
\int_{-1}^{\eps \zeta} (z+1)^{2i}\big(\mu \nabla\cdot V^{\rm app}+\dz w_{\rm app}\big)=0, \qquad 0\leq i \leq N.
$$
Indeed, this equation is equivalent to
\begin{align*}
h^{2i}\big(w_{\rm app}&-\eps\mu \nabla\zeta\cdot V_{\rm app}\big)\\
&+\sum_{j=0}^N \mu^j \Big[ \nabla\cdot \Big( \frac{1}{2(i+j) +1}h^{2(i+j)+1}\nabla\phi_j\Big)- \frac{4ij}{2(i+j)-1}h^{2(i+j)-1}\phi_j\Big]=0,
\end{align*}
which coincides with the $i$-th equation in \eqref{eqIK} upon replacing $w_{\rm app}-\eps\mu \nabla\zeta\cdot V_{\rm app}$ by $\dt \zeta$ as a consequence of the dimensionless version of the kinematic boundary condition \eqref{Euler4}.
 \end{remark}
The structure of this system is quite unusual as there are $N+1$ equations on the surface elevation $\zeta$ and one equation on a linear combination of the $\phi_i$. Clearly, the system is overdetermined on $\zeta$ and the  problem is characteristic in time. It is therefore not well posed in general and certain constraints are necessary on the $\phi_j$ to construct a solution. Let us briefly sketch here the strategy used in \cite{murakami2015solvability,nemoto2018} to prove well-posedness for \eqref{eqIK} (their proof also cover the case with a non flat topography). To make things even easier, let us consider the simplest case $N=1$; the system \eqref{eqIK} can then be written as
$$
\begin{cases}
\zeta_t+\nabla\cdot \big(h\nabla\phi_0+\frac{\mu}{3}h^3 \nabla \phi_1)=0,\\
h^2\zeta_t+\nabla\cdot \big(\frac{1}{3}h^3\nabla\phi_0+\frac{\mu}{5}h^5 \nabla \phi_1)-\frac{4}{3}h^3\phi_1=0,\\
\dt \phi_0+\mu h^2 \dt \phi_1+\zeta+\frac{\eps}{2}\vert \nabla \phi_0\vert^2+\eps \mu h^2 \nabla\phi_0\cdot \nabla\phi_1+2\eps \mu h^2\phi_1^2=0,
\end{cases}
$$
or, in compact form,
\begin{equation}\label{neweqIK}
\begin{cases}
\dt \zeta-L_{11}\phi_0-\mu L_{12}\phi_1=0,\\
h^2 \dt \zeta-L_{12}^* \phi_0-L_{22}\phi_1=0,\\
\dt \phi_0+\mu h^2 \dt \phi_1+\zeta+\frac{\eps}{2} \vert \nabla \phi_0\vert^2=\eps \mu F_1,
\end{cases}
\end{equation}
where the expressions for the second order differential operators $L_{ij}$ and the function $F_1$ follow easily by identification between both formulations. Eliminating $\zeta$ from the first two equations gives the constraint that $\phi_0$ and $\phi_1$ must satisfy, namely,
$$
h^2\big( L_{11}\phi_0+L_{12}\phi_1\big)=L_{12}^*\phi_0+L_{22}\phi_1
$$
or, more explicitly,
\begin{equation}\label{constr}
\frac{1}{2}\Delta\phi_0+(1+\frac{1}{10}\mu h^2 \Delta)\phi_1=0.
\end{equation}
Time differentiating this relation and using the first equation of \eqref{neweqIK} to eliminate $\dt \zeta$, one gets 
$$
\frac{1}{2}\Delta\dt \phi_0+(1+\frac{1}{10}\mu h^2 \Delta)\dt \phi_1=\eps \mu F_2.
$$
where $F_2$ is a function of $\zeta$, $\phi_0$ and $\phi_1$ and their space derivatives. Complementing this equation with the first and third equations in \eqref{eqIK} this yields a set of three equations on $\zeta$, $\phi_0$ and $\phi_1$,
\begin{equation}\label{IKcar}
\begin{cases}
\zeta_t+\nabla\cdot \big(h\nabla\phi_0+\frac{\mu}{3}h^3 \nabla \phi_1)=0,\\
\left(\begin{array}{cc} 1 & \mu h^2 \\ \frac{1}{2}\Delta & (1+\mu \frac{1}{10}h^2\Delta)\end{array}\right)\dt\left(\begin{array}{c} \phi_0\\ \phi_1\end{array}\right) +\left(\begin{array}{c} \zeta +\frac{\eps}{2}\vert \nabla\phi_0\vert^2 \\
0 \end{array}\right)=\eps \mu \left(\begin{array}{c} F_1 \\ F_2 \end{array}\right)
\end{cases}
\end{equation}
Contrary to \eqref{eqIK}, this system is non characteristic and  can be solved under certain hyperbolicity conditions (non vanishing depth, Rayleigh-Taylor condition), see \cite{murakami2015solvability}. The dispersion relation associated to \eqref{IKcar} is easily computed, 
$$
\omega_{\rm IK}^2(k)=k^2 c_{\rm IK}(k)^2 \quad \mbox{ with }\quad c_{\rm IK}(k)^2=\frac{1+\frac{1}{15}\mu k^2}{1+\frac{2}{5}\mu k^2}.
$$
(there is also a trivial component $\omega=0$ in the dispersion relation that corresponds to the propagation of the constraint \eqref{constr}). The dispersive properties of \eqref{IKcar} are excellent, since one can check that $c_{\rm IK}^2$ is the $[2/2]$-Pad\'e approximant of the square of the phase velocity of the linear water waves equations, namely, $c_{\rm WW}^2=\omega_{\rm WW}(k)^2/k^2$. This remarkable property can actually be generalized for $N\geq 1$;  indeed, it is shown in \cite{nemoto2018} that if $p_i=2i$ then the square of the phase velocity associated to the Isobe-Kakinuma model \eqref{eqIK}  is the $[2N/2N]$ Pad\'e approximant of $c_{\rm WW}^2$. Since the IK model has been derived in a formal way, it is not possible to say a priori that it furnishes a good approximation to the water waves equations (in the sense discussed in \S \ref{sectjustif} below). However, quite surprisingly, this is the case and the matching of this model with the full water waves equations is also excellent at the nonlinear level: as proved in \cite{iguchi2018,iguchi2}, the IK model furnishes an approximation of order $O(\mu^{2N+1})$ of the water waves equations in the case of a flat bottom (and of order $O(\mu^{2[N/2]+1})$ when the bottom is not flat). As the multi-layer models considered in the previous section, and contrary to the higher order SGN systems of \cite{matsuno} for instance, the IK model also has the interesting feature that it does not contain high order derivatives.

The strategy adopted in \cite{iguchi2018,iguchi2} to prove that the IK model is a high order shallow water approximation of the water waves equations is the following. With $\zeta$, $\phi_0$ and $\phi_1$ solving \eqref{neweqIK}, and setting $\psi=\phi_0+\mu h^2 \phi_1$,  it is shown that $(\zeta,\psi)$ solves the dimensionless version \eqref{ZCS_ND} of the Zakharov-Craig-Sulem formulation of the water waves equations up to $O(\mu^3)$ terms --we consider here the case of $N=1$ with a flat bottom but, as shown in \cite{iguchi2},  this can be generalized to any $N\geq 1$ and non flat bottoms. The error with the exact solution of the water waves equations can then be controled using the well-posedness and stability results of the ZCS formulation proved in \cite{Alvarez,Iguchi2009}. The key step of this approximation results consist therefore in checking that
$$
\nabla\cdot \big(h \nabla\phi_0+\mu \frac{1}{3}h^3 \nabla \phi_1\big)=-\frac{1}{\mu}G_\mu[\eps\zeta]\psi+O(\mu^3),
$$
where $G_\mu[\eps\zeta]=G_\mu[\eps\zeta,0]$ is the dimensionless Dirichlet-Neumann operator defined in Remark \ref{remZCS_ND}. Indeed, this identity directly shows that the first equation in \eqref{neweqIK} is equivalent at $O(\mu^3)$ precision to the first equation of the ZCS equations \eqref{ZCS_ND}, and the third equation of \eqref{neweqIK} can be similarly matched with the second equation of \eqref{ZCS_ND}. 

\begin{remark}
It is also of interest to point out that, as remarked in \cite{DucheneIguchi}, the IK model satisfies the same canonical Hamiltonian structure \eqref{Hamil} as the Zakharov-Craig-Sulem formulation of the water waves equations. More precisely, the equations on $\zeta$ and $\psi$ derived from the IK model are the canonical Hamiltonian equations for an Hamiltonian $H^{\rm IK}$ obtained by replacing $\Phi$ by $\Phi^{\rm app}$ in the Hamiltonian of the water waves equations.
\end{remark}

Since the excellent matching of the IK model with the ZCS equations may look quite unexpected, let us propose an alternative derivation of the IK model  when $N=1$ when the bottom is flat (the generalization to more general cases is not obvious). As above, we consider an approximation of the velocity potential of the form
$$
\Phi^{\rm app}(t,X,z)=\phi_0(t,X)+\mu (z+1)^2 \phi_1(t,X)
$$
and we impose that this approximation matches the exact velocity potential at the surface, that is,
\begin{equation}\label{newig1}
\phi_0+\mu h^2 \phi_1=\psi.
\end{equation}
Recalling that the evolution equation on $\zeta$ can be written $\dt \zeta +\nabla\cdot (h\ovV)=0$, where $\ovV$ is the vertically averaged horizontal velocity, a natural approximation is given by
$$
\dt \zeta +\nabla\cdot \big( h \ovV^{\rm app}\big)=0 \quad \mbox{ with } \quad \ovV^{\rm app}=\frac{1}{h}\int_0^{\eps\zeta} \nabla \Phi^{\rm app},
$$
and therefore
$$
\dt \zeta+\nabla\cdot \big(h \nabla\phi_0+\mu \frac{1}{3}h^3 \nabla \phi_1\big)= 0;
$$
this is the first equation of the IK model \eqref{neweqIK}.
We want this approximation to be as good as possible an approximation of the first equation of the ZCS formulation; more precisely, we want to choose $\phi_0$ and $\phi_1$ in terms of $\zeta$ and $\psi$ such that 
\begin{equation}\label{newig2}
\nabla\cdot \big(h \nabla\phi_0+\mu \frac{1}{3}h^3 \nabla \phi_1\big)=-\frac{1}{\mu}G_\mu[\eps\zeta]\psi +O(\mu^3).
\end{equation}
It is of course possible to replace $\frac{1}{\mu}G_\mu[\eps\zeta]\psi $ in the above condition by its third order expansion with respect to $\mu$. Such an expansion can be found in \cite{Alvarez,Iguchi2009,Lannes2013,iguchi2018} for instance. It can be written as follows
$$
\frac{1}{\mu}G_\mu[\eps\zeta]\psi=G_0+\mu G_1 +\mu^2 G_2 +O(\mu^3),
$$
with 
\begin{align*}
G_0&=-\nabla\cdot \big( h \nabla \psi \big),\\
G_1&=-\frac{1}{3}\Delta \big( h^3 \Delta \psi \big),\\
G_2 &=-\frac{1}{15}\Delta \big( h^3 \Delta (h^2 \Delta \psi)\big)-\frac{1}{15}\Delta \big( h^2 \Delta (h^3 \Delta \psi)\big)+\eps^2\frac{1}{5} \Delta \big( \abs{\nabla\zeta}^2 h^3 \Delta\psi \big).
\end{align*}
Finding $\phi_0$ and $\phi_1$ satisfying \eqref{newig1} and \eqref{newig2} can be reduced to the following system
$$
\begin{cases}
\phi_0+\mu h^2 \phi_1=\psi,\\
\nabla\cdot \big(h \nabla\phi_0+\mu \frac{1}{3}h^3 \nabla \phi_1\big)=-G_0-\mu G_1-\mu^2 G_2+O(\mu^3).
\end{cases}
$$
Replacing $\phi_0=\psi-\mu h^2 \phi_1$ in the second equation, one arrives after some computations to
$$
\phi_1=-\frac{1}{2}\Delta \psi - \frac{1}{10}\mu \Delta (h^2 \Delta \psi)-\frac{1}{10 h}\mu\Delta(h^3 \Delta \psi)+\mu \eps^2 \frac{3}{10}\abs{\nabla\zeta}^2 \Delta \psi
+O(\mu^2).
$$
Remarking that $f=(1+\mu A)g +O(\mu^2)$ is equivalent in the sense of Taylor expansions to $(1-\mu A) f=g +O(\mu^2)$, we obtain, after dropping the $O(\mu^2)$ residual, 
$$
\big[ (\big(1-\mu\eps h \Delta \zeta-\mu \eps^2 \abs{\nabla\zeta}^2\big)-\mu \frac{2}{5}\frac{1}{h^3}\nabla\cdot \big( h^5 \nabla \big) \big]\phi_1=-\frac{1}{2}\Delta \psi.
$$
Recalling that $\phi_0=\psi-\mu h^2 \phi_1$, one readily checks that $(\phi_0,\phi_1)$ satisfies the constraint \eqref{constr} of the IK model and is therefore the same pair as the one derived above with variational arguments. The second equation of \eqref{neweqIK}, or equivalently the constraint \eqref{constr}, is therefore equivalent to the condition \eqref{newig2}.

\begin{remark}\label{remcompareIK}
It is of interest to compare the IK model to the various (single or multi-layers) Boussinesq models developed in the previous sections. In order to do so, we set $\underline{V}=\nabla\psi$, $V_0=\nabla\phi_0$ and $V_1=\nabla\phi_1$, and we neglect the terms of order $O(\mu^2)$ and $O(\eps\mu)$. Taking also the gradient of the equation on $\psi$, one readily obtains
$$
\begin{cases}
(1-\mu \frac{2}{5}\Delta)\dt \zeta +\nabla\cdot (h \underline{V}) -\mu \frac{1}{15}\Delta\nabla\cdot \underline{V}=0,\\
\dt \underline{V}+\nabla\zeta +\eps \underline{V}\cdot \nabla \underline{V}=0,
\end{cases}
$$ 
which is one of the $abcd$ systems \eqref{abcd} previously derived. 
\end{remark}

\subsection{Scalar models}\label{sectscalar}

Intuitively, in the one dimensional case $d=1$, waves can be decomposed into components that  "go to the left" or "go to the right". It is therefore not a surprise that waves are then governed by a system of two scalar evolution equations. The idea behind scalar asymptotic models is that if we want to describe only waves that go mainly, say, "to the right", then a single scalar equation should be enough. We make this idea more precise in this section.

\medbreak

\noindent
{\bf N.B.} {\it Throughout this section, we shall focus on the case of a flat topography $b=0$}. We refer for instance to \cite{Johnson1973,Miles1979,VG1993,Israwi2010} for generalizations to a non flat topography.

\medbreak

 In dimension $d=1$, the SGN equations \eqref{SGN_V} reduce at leading order in $\eps$ and $\mu$ to the linear wave equation
$$
\begin{cases}
\dt \zeta +\dx \ovv=0,\\
\dt \overline{v}+\dx \zeta=0,
\end{cases}
$$
so that any perturbation of the rest state can be decomposed into a left-going and a right-going wave. Purely right-going waves are obtained when $\zeta=u$ and are therefore determined by
\begin{equation}\label{rightlin}
(\dt + \dx)\zeta=0 \quad \mbox{ and }\quad u=\zeta.
\end{equation}
The scalar models that are described below generalize this approach to more complex asymptotic models than the linear wave equation.
\begin{remark}
For the sake of simplicity, we consider here waves that are essentially unidirectional. In general, a perturbation of the surface elevation creates two counter-propagating waves. It is, under certain assumptions, possible to describe them by two \emph{uncoupled} scalar equations, as shown in \cite{Schneider:2000dz} for the KdV equation for instance (see also \cite{Lannes2013}).
\end{remark}
\subsubsection{A fully nonlinear, non-dispersive model}\label{sectFNscal}

Let us consider here the NSW equations which is fully nonlinear (no smallness assumption on $\eps$) but neglects all the terms of order $O(\mu)$ (where the dispersive terms are, as shown above); this is equivalent to taking $\mu=0$ in \eqref{SGN_V},
\begin{equation}\label{NSW1d}
\begin{cases}
\dt \zeta +\dx (h\overline{v})=0,\\
\dt \overline{v}+\eps\overline{v}\dx\overline{v}+\dx \zeta=0,
\end{cases}
\qquad (h=1+\eps\zeta).
\end{equation}
In the subcritical case, i.e. when $h-\eps^2\overline{v}^2>0$, this hyperbolic system can be diagonalized using the Riemann invariants. More precisely, introducing
$$
R_\pm(\zeta,\overline{v})=2\big(\sqrt{h}-1\big)\pm \eps \ovv
\quad\mbox{ and }\quad
\lambda_\pm(\zeta,\overline{v})=\pm \eps \ovv +\sqrt{h},
$$
the NSW equation can be diagonalized into two coupled transport equations
$$
(\dt \pm \lambda_\pm\dx)R_\pm =0.
$$
Purely right-going waves are therefore obtained if $R_-=0$ and therefore characterized by
\begin{equation}\label{burgers}
\dt\zeta+\dx \zeta+  3\eps \frac{\zeta}{1+\sqrt{1+\eps\zeta}}\dx\zeta=0,\quad\mbox{ and }\quad
\ovv=\frac{2}{\eps} \big(\sqrt{1+\eps\zeta}-1\big);
\end{equation}
as expected, this is a $O(\eps)$ perturbation of the relations defining right-going waves for the linear waves equations. The equation for $\zeta$ is a non-viscous Burgers equations whose solutions form shocks at the time scale $O(1/\eps)$. Note that solutions to the scalar model \eqref{burgers} are {\it exact} solutions to the NSW system \eqref{NSW1d}, sometimes called simple waves.
\begin{remark}\label{removv}
In \eqref{burgers}, $\zeta$ is determined through the resolution of a scalar evolution equation, and $\ovv$ is given by an algebraic expression in terms of $z$. It is of course possible to switch the roles of $\zeta$ and $\ovv$, leading to another kind of simple wave,
\begin{equation}\label{burgersv}
\dt \ovv +\dx \ovv +\eps \frac{3}{2}\ovv\dx \ovv=0 
\quad\mbox{ and }\quad
\zeta=\ovv+\eps\frac{1}{4}\ovv^2.
\end{equation}
\end{remark}

\subsubsection{A fully dispersive, linear model}
The symmetric case compared with the Burgers model \eqref{burgers} consists in  neglecting all the nonlinearities ($\eps=0$) and to keep all the terms in $\mu$ (the validity of the resulting model is therefore not restricted to shallow water regimes). For such an approximation, it is more convenient to work with the ZCS formulation \eqref{ZCS_ND}. The linear model thus obtained is
\begin{equation}\label{fullydisp}
\begin{cases}
\dsp \dt \zeta -  \omega_{\rm WW}(D)^2 \psi=0,\\
\dsp \dt \psi + \zeta =0.
\end{cases}
\end{equation}
where the symbol $\omega_{\rm WW}(k)$ of the Fourier multiplier $\omega_{\rm WW}(D)$ is given by
$$
\omega_{\rm ww}(k)=k \Big(\frac{\tanh(\sqrt{\mu}k)}{\sqrt{\mu}k}\Big)^{1/2}=:k c_{\rm WW}(k).
$$
The above system can therefore be diagonalized into two scalar uncoupled nonlocal equations
$$
\begin{cases}
\dt \big(\zeta+c_{\rm WW}(D)\dx\psi \big)+c_{\rm WW}(D) \dx  \big(\zeta+c_{\rm WW}(D)\dx\psi \big)=0,\\
\dt \big(\zeta-c_{\rm WW}(D)\dx\psi \big)-i c_{\rm WW}(D) \dx \big(\zeta-c_{\rm WW}(D)\dx\psi \big)=0.
\end{cases}
$$
Right-going waves correspond to waves with a positive group velocity and are therefore obtained when the equations are reduced to the first of these two scalar equations, i.e. when
\begin{equation}\label{linWitham}
\dt \zeta + c_{\rm WW}(D) \dx \zeta=0
\quad \mbox{ and }\quad \ovv= c_{\rm WW}(D)\zeta,
\end{equation}
where for the second relation, we used the identity $D\frac{\tanh(\sqrt{\mu}D)}{\sqrt{\mu}}\psi=-\dx\ovv$, which is exact when $\eps=\beta=0$. One can check that, as expected, \eqref{linWitham} is a formal $O(\mu)$ perturbation of \eqref{rightlin}.

As in Remark \ref{removv}, one can alternatively derive an equation on $\ovv$ and express $\zeta$ in terms of $\ovv$; one obtains
\begin{equation}\label{linWitham2}
\dt \ovv + c_{\rm WW}(D) \dx \ovv=0
\quad \mbox{ and }\quad \zeta= c_{\rm WW}(D)^{-1}\ovv.
\end{equation}

Here again, solutions to the scalar approximations \eqref{linWitham} or \eqref{linWitham2} furnish {\it exact} solutions to the underlying system \eqref{fullydisp}.
\subsubsection{The Whitham equation(s)}

We have so far obtained a fully nonlinear, nondispersive approximation ($\mu=0$, full dependence on $\eps$) and a fully dispersive, linear, approximation ($\eps=0$, full dependence on $\mu$). These approximations are given respectively by \eqref{burgers} and \eqref{linWitham}. Combining both models, a $O(\eps\mu)$ approximation is obtained, namely
\begin{equation}\label{altWhitham}
\dt \zeta + c_{\rm WW}(D) \dx \zeta+3\eps \frac{\zeta}{1+\sqrt{1+\eps\zeta}}\dx\zeta=0
\end{equation}
and
$$
 \ovv= c_{\rm WW}(D)  \zeta+ \frac{2}{\eps} \big(\sqrt{1+\eps\zeta}-1-\frac{1}{2}\eps\big).
$$

Taking $\ovv$ instead of $\zeta$ as reference to build the scalar approximation, as in Remark \ref{removv}, one obtains the following approximation
\begin{equation}\label{Whitham}
\dt \ovv + c_{\rm WW}(D) \dx \ovv+\frac{3}{2}\eps \ovv\dx\ovv=0
\quad\mbox{ and }\quad
 \zeta= c_{\rm WW}^{-1}(D)  \ovv+ \frac{\eps}{4}\ovv^2.
\end{equation}
This latter equation is known as the Whitham equation, proposed in \cite{Whitham,Whitham_book} as an alternative to the KdV equation with weaker (and better) dispersive properties, and able to reproduce peaking and wave breaking. This equation has been intensively studied in recent years. For instance, existence and stability of solitary waves has been proved in \cite{EGW} and their peaking towards cusped solutions in \cite{EhrnstromWahlen}, and wave breaking has been rigorously established in \cite{Naumkin,ConstantinEscher,Hur} as a manifestation of the general rule that weakly dispersive perturbations to the Burgers equation lead to the formation of singularities \cite{CCG}. We also refer to \cite{klein2018whitham} for several related conjectures motivated by numerical computations. A natural question is to ask wether such results still hold for  the alternative Whitham equation \eqref{altWhitham} on the surface elevation.

\subsubsection{The KdV and BBM equations}

The KdV and BBM equations are the scalar equations associated to the Boussinesq equations \eqref{Boussinesq}, which, we recall, are a $O(\mu^2)$ approximation of the water waves equations \eqref{formzetaQ3_ND_V} under the weak nonlinearity assumption \eqref{weakNL}, namely, $\eps=O(\mu)$. These equations can be derived from the Boussinesq equations \eqref{Boussinesq} along a procedure similar to the one used below to derive the Camassa-Holm equation from the SGN equation. As we show now, it can also be derived directly from the Whitham equation \eqref{altWhitham}.

Indeed, under the weak nonlinearity approximation, the Whitham equation \eqref{altWhitham} furnishes a $O(\eps\mu)=O(\mu^2)$ approximation of the water waves equation \eqref{formzetaQ3_ND_V}. This will remain true if we replace the non local dispersive term of the Whitham equation by a $O(\mu^2)$ approximation and the nonlinear term by a $O(\eps^2)=O(\mu^2)$ approximation. Since
$$
c_{\rm WW}(D)\dx \zeta= \dx \zeta +\frac{1}{6}\mu \dx^3 \zeta+O(\mu^2)
\quad \mbox{ and }\quad
 \frac{3\eps\zeta}{1+\sqrt{1+\eps\zeta}}\dx\zeta=\frac{3}{2}\eps\zeta\dx\zeta+O(\eps^2),
$$
we obtain the KdV equation
\begin{equation}\label{KdV}
\dt \zeta+\dx\zeta+\frac{1}{6}\mu \dx^3 \zeta+\frac{3}{2}\eps\zeta\dx\zeta=0;
\end{equation}
it is notable that one arrives at the same equation if we make similar approximations on the Whitham equation \eqref{Whitham} on the velocity $\ovv$ instead of the surface elevation $\zeta$.

We have seen in \S \ref{sectBoussinesq} that there is a whole family of Boussinesq systems, the $abcd$ systems \eqref{abcd} that all furnish a $O(\mu^2)$ approximation  to the water waves equations under the weak nonlinearity assumption \eqref{weakNL}. One of the arguments used to derive the $abcd$ system from the Boussinesq system \eqref{Boussinesq} is the so-called BBM trick that was introduced to derive the BBM equation from the KdV equation \cite{BBM}. It consists in remarking that owing to \eqref{KdV} and the weak nonlinearity assumption, one has
$
\dt \zeta=-\dx \zeta +O(\mu)
$, so that $\mu\dx^3 \zeta=-\mu\dx^2 \dt \zeta +O(\mu^2)$. Without damaging the $O(\mu^2)$ precision of the KdV approximation, one can use instead the BBM equation,
\begin{equation}\label{BBM}
(1-\frac{1}{6}\mu \dx^2)\dt \zeta +\dx \zeta +\frac{3}{2}\eps\zeta\dx\zeta=0
\end{equation}
or more generally, any member of the KdV/BBM family
\begin{equation}\label{KdVBBM}
\big( 1 +(p-\frac{1}{6}) \mu\dx^2 \big) \dt \zeta +\dx\zeta + \mu p \dx^3 \zeta +\frac{3}{2}\eps\dx\zeta=0\qquad (p\leq \frac{1}{6}).
\end{equation}

\subsubsection{The Camassa-Holm equation}

The equations from the KdV/BBM family \eqref{KdVBBM} are all globally well posed in reasonable Sobolev spaces and therefore unable to reproduce the wave breaking phenomenon. The reason of this is that dispersion balances the nonlinearity. There are two possibilities to avoid such a situation. The first one is to work with a model with weaker dispersion: this corresponds to the Whitham equations \eqref{altWhitham} and \eqref{Whitham} which, under the weak nonlinearity assumption, furnish an approximation of the same precision as the KdV/BBM family. And indeed, as we have seen, the Whitham equation (the classical one \eqref{Whitham} at least) can lead to wave breaking. The second possibility to obtain wave breaking is to work with a model having stronger nonlinearities. In order to do so while keeping the $O(\mu^2)$ precision of the KdV/BBM family, one can relax the weak nonlinearity assumption \eqref{weakNL} and replace it by
\begin{equation}\label{moderate}
\mbox{Moderate nonlinearity:} \qquad \eps=O(\sqrt{\mu}).
\end{equation}

Under this assumption, one must keep the $O(\eps\mu)$ terms in order to keep the same $O(\mu^2)$ precision as for the KdV-BBM family. Among the asymptotic systems derived above, the only one that takes this terms into account is the SGN model. In dimension $d=1$, this model can be written under the form
\begin{equation}\label{SGN1d}
\begin{cases}
\dt \zeta +\dx (h\ovv)=0,\\
\dt \ovv +\eps \ovv \dx \ovv+\dx \zeta =\frac{\mu}{3}\frac{1}{h}\dx\big[ h^3 (\dx\dt \ovv +\eps \ovv \dx^2\ovv-\eps (\dx\ovv)^2)\big].
\end{cases}
\end{equation}
Let us for instance seek an equation on $\ovv$ and an algebraic expression for $\zeta$ in terms of $\ovv$. Since the SGN equations are a $O(\mu)$ perturbation of the NSW equations, right going waves are expected to be a $O(\mu)$ perturbation of the Burgers equation \eqref{burgersv}, that is,
$$
\dt \ovv +\dx \ovv +\eps \frac{3}{2}\ovv\dx \ovv+\mu P=0 
\quad\mbox{ and }\quad
\zeta=\ovv+\eps\frac{1}{4}\ovv^2+\mu R. 
$$
where $P$ and  a function of $\ovv$ and its derivatives. Plugging the ansatz for $\zeta$ in the second equation of \eqref{SGN1d}, one gets
$$
\dt \ovv +\dx \ovv +\eps \frac{3}{2}\ovv\dx \ovv+\mu \dx R=\frac{1}{3}\mu (1-\eps\ovv)\dx^2 \dt \ovv+\frac{1}{3}\mu\eps \dx \big[3\ovv \dx\dt \ovv+\ovv\dx^2\ovv-(\dx\ovv)^2\big]
$$
or equivalently, using the ansatz for the scalar equation for $\ovv$,
$$
P=\dx R-\frac{1}{3} (1-\eps\ovv)\dx^2 \dt \ovv-\frac{1}{3}\eps \dx \big[3\ovv \dx\dt \ovv+\ovv\dx^2\ovv-(\dx\ovv)^2\big].
$$
This last equation gives $P$ in terms of $R$ and we therefore just have to find an expression for this latter quantity in terms of $\ovv$. In order to do so, we plug the ansatz for $\zeta$ in the first equation of \eqref{SGN1d}. This yields an evolution equation for $\ovv$ that should of course be the same as our ansatz. By identification, this yields an expression for $R$, from which we deduce $P$. We refer to \cite{ConstantinLannes} or \cite{Lannes2013} for the details of the computations; the final outcome is a family of Camassa-Holm equations that generalizes the above KdV/BBM family,
\begin{equation}\label{CHv}
\dt \ovv +\dx \ovv +\eps \frac{3}{2}\ovv\dx \ovv+\mu \big({\bf a}\dx^3 \ovv +{\bf b}\dx^2\dt \ovv \big)
=\eps\mu\big({\bf c} \ovv \dx^3 \ovv +{\bf d}\dx\ovv\dx^2 \ovv\big),
\end{equation}
where
$$
{\bf a}=p, \quad {\bf b}=p-\frac{1}{6}, \quad {\bf c}=-\frac{3}{2}p-\frac{1}{6}, \quad {\bf d}=-\frac{9}{2}p -\frac{23}{24},
$$
the parameter $p$ coming, as for the KdV-BBM family, from using the BBM trick. Note that a wider range of parameters can be achieved by performing a change of variable on the velocity of the same kind as \eqref{Nwogu} in the derivation of the $abcd$ systems \eqref{Boussinesq}. The equation one would obtain for $\zeta$ is \cite{Lannes2013}
\begin{equation}\label{CHzeta}
\dt\zeta+\dx \zeta+  3\eps \frac{\zeta}{1+\sqrt{1+\eps\zeta}}\dx\zeta+\mu \big({\bf a}\dx^3 \zeta +{\bf b}\dx^2\dt \zeta \big)
=\eps\mu\big({\bf c} \zeta \dx^3 \zeta +{\bf d}\dx\zeta\dx^2 \zeta\big);
\end{equation}
expanding the nonlinear terms into powers of $\eps$ up to $O(\eps^4)$ (recall that under the assumption of moderate nonlinearity, one has $O(\eps^4)=O(\mu^2)$ so that the corresponding terms can be neglected), one gets \cite{ConstantinLannes}
\begin{align}
\nonumber
\dt\zeta+\dx \zeta+  \frac{3}{2}\eps \zeta\dx\zeta -\frac{3}{8}\eps^2 \zeta^2 \dx \zeta +\frac{3}{16}\eps^3 \zeta^3 \dx\zeta &+\mu \big({\bf a}\dx^3 \zeta +{\bf b}\dx^2\dt \zeta \big)\\
\label{CHzeta2}
&=\eps\mu\big({\bf c} \zeta \dx^3 \zeta +{\bf d}\dx\zeta\dx^2 \zeta\big).
\end{align}
Compared to the KdV-BBM family, the inclusion of new nonlinear terms of size $O(\eps\mu)$ (as well as $O(\eps^2)$ and $O(\eps^3)$ in \eqref{CHzeta2})  restores the possibility of wave breaking, as shown in \cite{ConstantinLannes} for \eqref{CHv} and \eqref{CHzeta2} (and this could likely be extended to \eqref{CHzeta2}); we recall that wave breaking for $\ovv$ means that $\ovv$ remains bounded but that $\dx\ovv$ blows up in finite time (and a similar definition holds of course for $\zeta$). This wave breaking is shown to occur on a $O(1/\eps)$ time scale, which is the same as for the Burgers equations \eqref{burgers} and \eqref{burgersv}. 

Let us mention finally that \eqref{CHv} can be related, up to some rescaling, to the Camassa-Holm equation \cite{Fokas,Camassa}
$$
\dt U+\widehat{\kappa}\dx U+ 3 U \dx U- \dt\dx^2 U= 2 \dx U\dx^2 U+U\dx^3 U\qquad (\widehat{\kappa}\neq 0)
$$
provided that ${\bf b}<0$, ${\bf a}\neq {\bf b}$, ${\bf b}=-2{\bf c}$, ${\bf d}=2{\bf c}$ or to the Degasperis-Procesi equation \cite{DP}
$$
\dt U+\widehat{\kappa}\dx U+ 4 U \dx U- \dt\dx^2 U= 3 \dx U\dx^2 U+U\dx^3 U\qquad (\widehat{\kappa}\neq 0)
$$
provided that ${\bf b}<0$, ${\bf a}\neq {\bf b}$, ${\bf b}=-\frac{8}{3}{\bf c}$, ${\bf d}=3c$. There is a huge literature devoted to these two equations and which can be used to get some insight on the behavior of \eqref{CHv} (note however that the case $\widehat{\kappa}=0$, which has a very rich mathematical structure, {\it cannot} be related to a one directional shallow water wave propagation model). A natural question is therefore to ask which of these properties remain true for other ranges of the parameters in \eqref{CHv} and for the equations \eqref{CHzeta} and \eqref{CHzeta2} on the surface elevation.

\subsubsection{Two dimensional generalizations}

As we have seen, in dimension $d=1$ and in shallow water, perturbations of the surface elevation split at first order into two counter-propagating waves, and the scalar models derived above describe the evolution of each of these two components. In dimension $d=2$, there is no such splitting and therefore no direct generalization of the above. It is however possible to consider \emph{weakly transverse} waves for which the scale of the dependance on the transverse direction $y$ is larger. More precisely, it is possible to consider waves of the form
$$
\zeta(t,x,y)=\underline{\zeta}(t,x,\sqrt{\mu} y)
$$
for some profile function $\underline{\zeta}$. The fact that the transverse dependence is weaker allows the one dimensional splitting to operate before the transverse dependence is felt: the waves split up into two counter propagating components, each of them described by a scalar equation on $\underline{\zeta}$. The most famous weakly transverse equation is certainly the Kadomtsev-Petviashvili (KP) equation, which is the weakly transverse generalization of the KdV equation and is given by
\begin{equation}\label{KP}
\dt \underline{\zeta}+\dx \underline{\zeta}+\mu \frac{1}{2}\dx^{-1}\partial_Y^2\underline{\zeta} +\mu \frac{1}{6}\dx^3 \underline{\zeta} +\eps \frac{3}{2}\underline{\zeta}\dx \underline{\zeta}=0,
\end{equation}
where $Y$ stands for the variable $\sqrt{\mu}y$. We refer to \cite{LannesSaut,Lannes2013} and references therein for the derivation and justification of this model and to \cite{klein2012numerical,KleinSaut} for recent reviews of related mathematical issues. Of course, weakly transverse generalizations of the other scalar models are possible along the same lines; for instance, a weakly transverse generalization of the Whitham equation \eqref{Whitham} was proposed in \cite{Lannes2013,lannes2013remarks} and reads
\begin{equation}\label{Whitham2D_tr}
\dt \ovv + c_{\rm WW}(\abs{D^\mu}) \big( 1+\mu \frac{D_x^2}{D_1^2}\big)^{1/2}\dx \ovv+\frac{3}{2}\eps \ovv\dx\ovv=0,
\end{equation}
where $c_{\rm WW}$ is the same function as in \eqref{Whitham} and $\abs{D^\mu}$ is the Fourier multiplier $\abs{D^\mu}=\big( D_x^2 +\mu D_y^2 \big)^{1/2}$. One readily recovers \eqref{KP} from \eqref{Whitham2D_tr} by a formal Taylor expansion at $\mu=0$. It is also of interest to remark that the singularity due to the operator $\dx^{-1}$ in KP is removed in \eqref{Whitham2D_tr} --this singularity comes by the way from the underlying linear wave equation, and is therefore not specific to the dispersive or nonlinear terms of the KdV equation.

\subsection{Justification procedure}\label{sectjustif}

Except for the IK model for which we explained how to rigorously justify the approximation, the asymptotic models derived in the previous sections have been derived somewhat formally. Indeed, when we wrote, for instance, that a residual was of size $O(\mu^2)$, we did not precisely define the meaning of this notation. In the case of a flat bottom, the solutions of the exact equations \eqref{formzetaQ3_ND_V} depend on the two parameters $\eps$ and $\mu$ and, implicitly, we used the notation $R_{\eps,\mu}=O(\mu^2)$ with the following meaning
$$
 \exists C>0, \forall \mu\in  (0,1), \forall \eps\in (0,1), \qquad \frac{1}{\mu^2} \Vert R_{\eps,\mu} \Vert_{{\mathbb A}}\leq C \Vert \zeta_{\eps,\mu} , \ovV_{\eps,\mu} \Vert_{{\mathbb B}}
$$
(in the weakly nonlinear scaling one should additionally impose $\eps\leq c\mu$ for some constant $c>0$) where ${\mathbb A}$ and ${\mathbb B}$ are two functional spaces, for instance ${\mathbb A}=L^\infty(0,T/\eps;H^{s_1}(\R^d))$ and ${\mathbb B}=W^{k,\infty}(0,T/\eps;H^{s_2}(\R^d))$ with $k\geq 0$, $0\leq s_1\leq s_2$, and where $(\zeta_{\eps,\mu},\ovV_{\eps,\mu})$ stands for an exact solution of \eqref{formzetaQ3_ND_V} (the subscript $\eps,\mu$ has been added here to make the dependance on these two parameters explicit). Of course, this definition does not ensure in general that $R_{\eps,\mu}=O(\mu^2)$ in ${\mathbb A}$ with the standard meaning of this notation, namely,
$$
R_{\eps,\mu}=O(\mu^2)\iff   \exists \widetilde{C}>0,  \forall \eps,\mu \in (0,1), \qquad \frac{1}{\mu^2} \Vert R_{\eps,\mu} \Vert_{{\mathbb A}}\leq \widetilde{C}
$$
In order for this latter property to be true, one must show that {\it the family $(\zeta_{\eps,\mu},\ovV_{\eps,\mu})_{\eps,\mu}$ is uniformly bounded in ${\mathbb B}$} for $\eps,\mu\in (0,1)$. This is in general the most difficult part of the justification procedure that requires special care; for instance well-posedness results based on microlocal analysis such as \cite{Alazard2014} provide a very precise information on the minimal regularity required for local well-posedness of the water waves equations, but do not provide such uniform bounds. The necessary uniform bounds for a time scale of order $O(1/\eps)$ which is the physically relevant one were proved in \cite{Alvarez} and \cite{Iguchi2009} (see also \cite{Lannes2013}). With such results at hand, the above computations show that the solutions of the water waves equations are \emph{consistent} with the asymptotic models, in the sense that they solve them up to a small residual. 

In order to obtain a \emph{convergence} result, the last step is to prove the local well-posedness of the asymptotic model under consideration, and the stability of the solutions with respects to perturbations: typically, if $u$ solves an asymptotic model up to a residual of order $O(\mu^p)$ then it is $O(\mu^p t)$ close to the exact solution of this model with same initial data on the time scale for which the stability result holds (typically $O(1/\eps)$). Such results are usually much easier to obtain than the uniform bounds for the water waves equations, and provide the expected convergence result: the solutions to the water waves equations are close (how close depending on the precision of the model) to the solutions to the asymptotic model under consideration. We do not give too much details on these aspects which are treated in great generality and detail in \cite{Lannes2013}.

\section{Extension to rotational flows}\label{sectrot}

The goal of this section is to show how to generalize the results of Section \ref{sect2} when non zero vorticity is allowed, that is, when assumption \eqref{Euler3} is removed from the basic equations. We first show how to generalize  the Zakharov-Craig-Sulem formulation of the water waves equation \eqref{ZCS} as well as the elevation-discharge formulation \eqref{formzetaQ3} when vorticity is present. In order to introduce the dimensionless version of these equations, the notion of {\it strength} of the vorticity needs to be introduced. For most of this section, we consider a strength $\alpha=1/2$ which is the largest one for which we have rigorous bounds that allow the justification of the asymptotic models along a procedure similar to the one described in \S \ref{sectjustif} for the irrotational case. As in the irrotational case, an asymptotic description of the velocity and pressure field in the fluid domain is needed in order to understand the contribution of the turbulent and non-hydrostatic components in the averaged Euler equations \eqref{formzetaQ3_ND_V}; this analysis is performed in \S \ref{sectinnervort}. The incidence on the NSW and SGN models is then discussed in \S \ref{sectNSWSGNvort}; it is in particular shown that the SGN equations must be extended by a third equation on some turbulent tensor. This extended model can serve as a basis for the modeling of wave breaking, provided that some ad hoc mechanism is added to the equations; this is done in \S \ref{WB}. Finally, several models are formally derived in the presence of a big vorticity in \S \ref{sectlargevort}.

\subsection{The water waves equations with vorticity}\label{sectWWrot}

If one wants to take vorticity effects into account, it is necessary to remove the assumption \eqref{Euler3} from the water waves equations \eqref{Euler1}-\eqref{Euler6}. The vorticity ${\bf \omega}:=\curl \bU$ is therefore not identically equal to zero and satisfies instead the vorticity equation
\begin{equation}\label{vorteq}
\dt \bom +\bU\cdot \nabla_{X,z} \bom =\bom\cdot \nabla_{X,z}\bU
\end{equation}
(we treat here the case $d=2$, the adaptation to the case $d=1$ being straightforward). We show here how to generalize, in the presence of vorticity, the two formulations of the water waves equations considered in these notes, namely, the Zakharov-Craig-Sulem formulation \eqref{ZCS} and the elevation-discharge formulation \eqref{formzetaQ3}. The dimensionless version of these equations is then given and the notion of {\it strength} of the vorticity is introduced.

\subsubsection{The generalized Zakharov-Craig-Sulem formulation in the presence of vorticity}

The Zakharov-Craig-Sulem (ZCS) equations \eqref{ZCS} are a system of two evolution equations on $\zeta$ and $\psi$, this latter quantity being defined as the trace at the surface of the velocity potential $\Phi$ defined by the relation $\bU=\nabla_{X,z} \Phi$. This relation being a reformulation of the irrotationality assumption \eqref{Euler3}, there is no direct generalization of the (ZCS) equations in the presence of vorticity.

Instead, it was noticed in \cite{CastroLannes1} that, in the irrotational framework, one has
$$
\nabla\psi=\underline{V}+\underline{w}\nabla \zeta,
$$
where $\underline{V}$ and $\underline{w}$ respectively denote the horizontal and vertical component of the velocity field evaluated at the surface of the fluid domain. We can therefore seek directly an equation on
\begin{align*}
U_\parallel&:=\underline{V}+\underline{w}\nabla \zeta \\
&=\big(\underline{U}\times N\big)_{\rm h},
\end{align*}
the subscript ${\rm h}$ denoting the horizontal component. Taking the trace of the Euler equation \eqref{Euler1} at the surface and taking the horizontal component of the cross product of the resulting equation with $N$, one arrives after some computations at
$$
\dt U_\parallel+ g \nabla \zeta+\frac{1}{2} \nabla \abs{U_\parallel}^2 -\frac{1}{2} \nabla \big( (1+\abs{\nabla\zeta}^2)\underline{w}^2\big)=-\underline{\omega}\cdot N \uV^\perp,
$$
where we also used the fact that since the pressure $P$ is constant at the surface, $(\nabla_{X,z} P)_{\vert_{z=\zeta}} \times N=0$. Denoting by $\Pi$ and $\Pi_\perp$ the orthogonal projectors onto gradient and orthogonal gradient vector fields,
$$
\Pi=- \frac{\nabla \nabla^{\rm T}}{\Delta}, \qquad \Pi_\perp=- \frac{\nabla^\perp  (\nabla^\perp)^{\rm T}}{\Delta},
$$
we can decompose $U_\parallel$ as
$$
U_\parallel= \Pi U_\parallel +\Pi_\perp U_\parallel=\nabla\psi+\nabla\widetilde{\psi}
$$
for some scalar functions $\psi$ and $\widetilde{\psi}$. Remarking that $\Delta \widetilde{\psi}=\underline{\omega}\cdot N$, there is no need to derive an equation for $\widetilde{\psi}$. For $\psi$ however, such an equation is necessary, and it is obtained by applying $\Pi$ to the above evolution equation on $U_\parallel$. We can now state the extended Zakharov-Craig-Sulem formulation in the presence of vorticity
\begin{equation}\label{ZCSvort}
\begin{cases}
\dt \zeta +\underline{V}\cdot \nabla \zeta -\underline{w}=0,\\
\dt \psi + g  \zeta+\frac{1}{2}  \abs{U_\parallel}^2 -\frac{1}{2}  \big( (1+\abs{\nabla\zeta}^2)\underline{w}^2\big)= \frac{\nabla^{\rm T}}{\Delta} (\underline{\omega}\cdot N \uV^\perp),\\
\dt \bom +\bU\cdot \nabla_{X,z} \bom =\bom\cdot \nabla_{X,z}\bU,
\end{cases}
\end{equation}
which is a closed system of equations in $(\zeta,\psi,\bom)$ in the sense that it is possible to reconstruct the full velocity field $\bU$ (and a fortiori its trace $\underline{U}$ at the surface) in terms of these three quantities. The derivation and mathematical analysis (local well-posedness, Hamiltonian structure, uniform bounds, shallow water asymptotics, etc.) of this formulation can be found in \cite{CastroLannes1}. A generalization of this formulation in the presence of a Coriolis force can also be found in \cite{Melinand2017}.

\subsubsection{The generalized elevation-discharge formulation in the presence of vorticity}

The derivation of the averaged Euler equations \eqref{formzetaQ3} did not require the irrotationality assumption \eqref{Euler3} and are therefore still valid in the presence of vorticity,
\begin{equation}\label{formzetaQ3vort}
\begin{cases}
\dsp \dt \zeta+\nabla\cdot Q=0,\\
\dsp \dt Q+\nabla\cdot(\frac{1}{h}Q\otimes Q)+gh\nabla \zeta+\nabla\cdot\Rey+\frac{1}{\rho}\int_{-h_0+b}^\zeta \nabla P_{\rm NH}=0,
\end{cases}
\end{equation}
with $P_{\rm NH}$ and ${\bf R}$ still defined by \eqref{PNH} and \eqref{defR} respectively. The difference with the irrotational case is that there is no such thing as Proposition \ref{propclosed}, i.e., these equations do not form a closed set of evolution equations in $(\zeta,Q)$.  A possible generalization would be to prove that \eqref{formzetaQ3vort} and the vorticity equation \eqref{vorteq} form a closed set of equations in $(\zeta,Q,\bom)$. From the definition of $P_{\rm NH}$ and ${\bf R}$, this would require to generalize the reconstruction mapping of Proposition  \ref{propclosed}  by a mapping ${\mathfrak R}[\zeta]: (Q,\bom) \mapsto \bU$ where $\bU=(V^{\rm T},w)^{\rm T}$ satisfies
$$
\int_{-h_0+b }^{\zeta} V=Q, \qquad N_{\rm b}\cdot U_{\rm b}=0, \qquad \curl \bU= \bom, \qquad {\rm div}\, \bU=0.
$$

\subsubsection{Nondimensionalized equations and well-posedness of the equations}

Proceeding as in \S \ref{sectnondim}, and with the same notations, it is possible to derive a dimensionless version of \eqref{ZCSvort} and \eqref{formzetaQ3vort} provided that we define the {\it strength of the vorticity}. An important effect of the vorticity is that it induces a vertical shear;  recalling that the vertical variable is scaled by  $h_0$, that   the horizontal velocity $V$ is scaled by $a \sqrt{g/h_0}$, a typical scale to measure this shear is the natural scale of $\dz V$, namely  $\Omega_0=a/h_0  \sqrt{g/h_0}$. This motivates the following definition
\begin{equation}\label{strength}
\mbox{The vorticity is of strength }\alpha>0\mbox{ if } \Omega_0^{-1}\curl \bU =O(\mu^\alpha).
\end{equation}
Omitting the tildes for dimensionless quantities and defining
\begin{equation}\label{defomegamu}
\bom_\mu = \left( \begin{array}{c}
\mu^{-\alpha}\big( \dz V^\perp -\nabla^\perp w) \\
-\mu^{1/2-\alpha}\nabla \cdot V^\perp
 \end{array}\right)
\end{equation}
this means that $\bom_\mu$ is a $O(1)$ quantity with respect to $\mu$. The time evolution of $\bom_\mu$ is directly given by the non dimensionalization of \eqref{vorteq}, 
\begin{equation}\label{vortND}
\dt \bom_\mu +\frac{\eps}{\mu} \bU^\mu\cdot \nabla^\mu \bom_\mu=\frac{\eps}{\mu}\bom_\mu\cdot \nabla^\mu \bU^\mu
\end{equation}
where $\bU^\mu=\left(\begin{array}{c} \sqrt{\mu} V \\ w \end{array}\right)$ and  $\nabla^\mu=\left(\begin{array}{c} \sqrt{\mu} \nabla \\ \dz \end{array}\right)$ --note in particular that $\bom^\mu=\mu^{3/2-\alpha}\nabla^\mu \times \bU^\mu$.

We shall mainly consider throughout these notes the case $\alpha=1/2$, which is the smallest value of $\alpha$ (and therefore the strongest vorticity) for which it is known that the nondimensionalized generalized ZCS equations \eqref{ZCSvort} are well-posed over a time $O(1/\eps)$ and uniformly with respect to $\mu \leq 1$. This result, proved in \cite{CastroLannes1}, ensures that all the asymptotic expansions performed in \S\ref{sectinnervort} and \S \ref{sectNSWSGNvort} are justified.

Extending such a result to larger vorticities (i.e. to smaller values of $\alpha$) is still an open problem, but it is however possible to derive some asymptotic models in such regimes, as shall be done in \S \ref{sectlargevort}.

\subsection{The inner structure of the velocity and pressure fields in the presence of vorticity} \label{sectinnervort}

{\it For the reasons explained above, we consider here a vorticity of strength $\alpha=1/2$}, in the sense of \eqref{strength}. As in \S \ref{sectinner} for the irrotational case, it is possible to  describe the inner structure of the velocity field in shallow water in the presence of vorticity. With
the nondimensionalization \eqref{vortND}, the relations \eqref{liste} that were used in the irrotational case must be replaced by
\begin{equation}\label{listevort}
\begin{cases}
\mu\nabla\cdot V+\dz {w}=0,\\
\sqrt{\mu}\dz V-\sqrt{\mu}\nabla w=-\mu\bom_{\mu,{\rm h}}^\perp,\\
\mu\nabla^\perp \cdot V=\mu \bom_{\mu,{\rm v}},\\
w_b-\beta \mu \nabla b\cdot V_b=0.
\end{cases}
\end{equation}
where $\bom$ is as defined in \eqref{defomegamu} (with $\alpha=1/2$) and $\bom_{\rm h}$ and $\bom_{\rm v}$ denotes its horizontal and vertical components.
The first and last equations can be used to obtain
$$
w=-\mu\nabla\cdot \big[ (1+z-\beta b)\overline{V}\big]-\mu\nabla\cdot \int_{-1+\beta b}^z V^*,
$$
which is the same relation as in the irrotational case. The influence of the vorticity appears when we plug this relation into the second equation, leading to
\begin{align*}
V^*
=&\mu \Big( \int_z^{\eps\zeta} \nabla\nabla\cdot \big[(1+z'-\beta b)\overline{V}\big]{\rm d}z'\Big)^*+\mu \Big(\int_{z}^{\eps\zeta}\nabla\nabla\cdot \int_{-1+\beta b}^z V^* \Big)^*\\
&+\sqrt{\mu}\Big(\int_z^{\eps\zeta}\bom_{\mu,{\rm h}}^\perp\Big)^*
\end{align*}
(this expression differs from the corresponding irrotational one by the presence of the last term). Defining ${\bf T}[\eps\zeta,\beta b]$ and ${\bf T}^*[\eps\zeta,\beta b]$ as in \eqref{defTT}, we can write in condensed form
$$
(1-\mu {\bf T}^*)V^*=\sqrt{\mu} V_{\rm sh}^*+\mu {\bf T}^* \overline{V}
$$
where $V_{\rm sh}$ is the shear velocity created by the vorticity,
$$
V_{\rm sh}=\int_z^{\eps\zeta}\bom_{\mu,{\rm h}}^\perp.
$$
so that 
\begin{align*}
V^*&=\sqrt{\mu} V_{\rm sh}^*+\mu {\bf T}^* \overline{V}+\mu^{3/2} T^* V^*_{\rm sh}+O(\mu^2).
\end{align*}
This shows that  {\it the fluctuation of the horizontal velocity around its average is mainly due to the influence of the vorticity}, which contributes at order $O(1/\sqrt{\mu})$ while the dispersion associated to the (irrotational) nonlocal effects only contributes at order $O(1/\mu)$. The shallow water expansion of the velocity field in the presence of vorticity is therefore given when the bottom is flat by
\begin{equation}\label{expinnervelvort}
\begin{cases}
V=\overline{V}+\sqrt{\mu}V_{\rm sh}^*- \frac{1}{2}\mu \big(  (1+z)^2-\frac{1}{3} h^2 \big) \nabla\nabla\cdot \overline{V}+\mu^{3/2} T^* V^*_{\rm sh}+O(\mu^2),\\
w=-\mu (1+z)\nabla\cdot \overline{V}-\mu^{3/2}\nabla\cdot \int_{-1}^z V_{\rm sh}^*+O(\mu^2);
\end{cases}
\end{equation}
the generalization to non flat bottoms is given in \eqref{expinnervelvorttopo}. Note that contrary to what happens for the horizontal velocity, the contribution of the vorticity to the vertical velocity is smaller than the irrotational contribution (and this remains true for larger vorticities).

As in \S \ref{sectNH}, plugging these approximations into the formula for the non hydrostatic pressure, namely,
$$
\frac{1}{\eps}P_{\rm NH}= \int_z^{\eps \zeta} \big( \dt w +\eps V\cdot \nabla w +\frac{\eps}{\mu}w \dz w \big)
$$
to obtain an asymptotic expression of the non hydrostatic pressure field in the fluid domain. One easily checks that the new vorticity terms contribute to order $O(\mu^{3/2})$, so that the expansion \eqref{expP} derived in the irrotational framework remains valid, but with a residual term of order $O(\mu^{3/2})$ instead of $O(\mu^2)$,
\begin{align}
\label{expProt}
\frac{1}{\eps} P_{\rm NH}&=-\mu\big[ \frac{h^2}{2}-\frac{(1+z)^2}{2}\big]  \big(\dt +\eps \ovV\cdot \nabla -\eps\nabla\cdot \ovV\big)\nabla\cdot\ovV
+O(\mu^{3/2});
\end{align}
(similarly, when the bottom is not flat, \eqref{expPtopo} still holds, but with a residual of order $O(\mu^{3/2})$ instead of $O(\mu^2)$).
\begin{remark}
Of course, plugging \eqref{expinnervelvort} into the above formula for $P_{\rm NH}$, one can get a more precise expansion, up to order $O(\mu^2)$. The additional terms are quite complicated however, and for the sake of clarity, we chose here to limit our analysis to a $O(\mu^{3/2})$ precision; we refer to \cite{CastroLannes2} for the full $O(\mu^2)$ expansion.
\end{remark}
Note finally that even though the vorticity does not appear in \eqref{expProt}, it plays a role in the evolution of $\zeta$ and $\ovV$. It is therefore not surprising that the reconstruction of the surface elevation from pressure measurements is more complex in the presence of vorticity, since it is then more complicated to replace space derivatives by time derivatives or more generally Fourier multipliers in time (indeed, a pressure sensor does not provide any information on the variations in space of the pressure). This has been done only in some particular cases such as solitary waves \cite{Henry} and linear plane waves \cite{HenryThomas}.

\subsection{The NSW and SGN equations in the presence of vorticity}\label{sectNSWSGNvort}

{\it We remind that we consider here a vorticity of strength $\alpha=1/2$}, in the sense of \eqref{strength}. The  "turbulent" and non-hydrostatic terms in \eqref{formzetaQ3_ND} can be expended as follows, following the results of \S \ref{sectinnervort},
\begin{align*}
\eps\nabla\cdot \Rey&=\eps \mu \nabla\cdot {\bf E}+O(\eps\mu^{3/2}) \\
  \frac{1}{\eps}\int_{-1}^{\eps\zeta}\nabla P_{\rm NH}&=
 \mu h  {\mathcal T}\big[\dt \ovV+\nabla\cdot \big(h  \ovV\otimes \ovV\big)\big] +\mu \eps h{\mathcal Q}_1\big(\zeta, \ovV\big)
 +O(\eps \mu^{3/2}),
\end{align*}
where the symmetric tensor ${\bf E}$ measures the quadratic self interaction of the fluctuation $V^*_{\rm sh}$ of the shear velocity $V_{\rm sh}$ created by the vorticity,
$$
{\bf E}=\int_{-1}^{\eps \zeta} V^*_{\rm sh}\otimes V^*_{\rm sh}.
$$
Therefore, for large amplitude waves $\eps=O(1)$, the contribution of the vorticity term to the averaged Euler equations  due to the "turbulent" term $\eps \nabla\cdot {\bf E}$, which is of size $O(\eps\mu)$, is larger than the rotational part of the non hydrostatic pressure, which is of size $O(\mu^{3/2})$. In the weakly nonlinear regime \eqref{weakNL}, i.e. if $\eps=O(\mu)$, this is the opposite situation. Both contribution are of equal order in the medium amplitude regime $\eps=O(\sqrt{\mu})$. 
\begin{remark}
As explained above, we work here with a $O(\mu^{3/2})$ precision instead of the $O(\mu^2)$ precision used for the SGN equation in the irrotational case. The computations are pushed further in \cite{CastroLannes2} to keep the $O(\mu^2)$ precision. It is in particular shown that new turbulent terms appear at order $O(\eps\mu^{3/2})$, and the $O(\mu^{3/2})$  terms of the  non hydrostatic pressure are also computed explicitly.
\end{remark}

Let us now consider the consequences of this new "turbulent" term on the Nonlinear Shallow Water and Serre-Green-Naghdi models.

\subsubsection{The NSW equations in the presence of vorticity}

As seen above, the first contribution of the rotational terms to the averaged Euler equations is of size $O(\eps\mu)$, which is below the $O(\mu)$ precision of the NSW equations \eqref{eqNSWop}. Therefore, {\it in the presence of vorticity, the NSW equations \eqref{eqNSWop} still furnish a $O(\mu)$ approximation to the (rotational) water waves equations.}

If the dynamics of the surface elevation $\zeta$ and of the average velocity $\ovV$ are not affected by the vorticity, this does not mean that there are no rotational effects at all. For instance, in the irrotational setting, the horizontal velocity is independent of the vertical coordinate, see \eqref{expinnervel}, so that the horizontal velocity at the surface is well approximated by the average velocity,
$$
 \underline{V}(t,x)=\ovV(t,X)+O(\mu)
 \quad \mbox{ where }\quad
 \underline{V}(t,x):=V\big(t,X,\eps \zeta(t,X)\big).
 $$
 As shown by \eqref{expinnervelvort}, a corrective term must be added to this approximation if one wants to keep the same precision, namely,
 $$
  \underline{V}(t,x)=\ovV(t,X)- \sqrt{\mu} \frac{1}{h}\int_{-1}^{\eps \zeta} \int_z^{\eps \zeta }\bom_{\mu,{\rm sh}}^\perp+O(\mu)
 $$
 (the corrective term being equal to $V^*_{\rm sh}$ evaluated at the surface); if one is interested, say, in the motion of  drifters at the surface in a zone with background currents, this corrective term should be added to the velocity furnished by the NSW equations.
 
\subsubsection{The SGN equations in the presence of vorticity}
Plugging the above expansions into the averaged Euler equations \eqref{formzetaQ3_ND} and dropping the $O(\mu^{3/2})$ terms, one obtains the same SGN equations as in \eqref{SGN} but with the additional "turbulent" term $\eps \mu \nabla\cdot {\bf E}$ in the momentum equation (or $\eps\mu \frac{1}{h}\nabla\cdot {\bf E}$ if we work with the formulation in $(\zeta,\ovV)$ variables \eqref{SGN_V}). The difficulty here is that ${\bf E}$ is not a function of $\zeta$ and $Q$ but of the horizontal component of the vorticity $\bom_{\mu,{\rm h}}$ (through $V_{\rm sh}$). In order to compute it, it seems therefore necessary to solve the vorticity equation \eqref{vortND} which is an equation cast in the fluid domain which is $d+1$ dimensional (while the SGN equation are cast on $\R^d$); solving this equation would be essentially as challenging numerically as solving the full free surface Euler equations. Fortunately, it happens that it is possible to derive an equation satisfied by ${\bf E}$ on $\R^d$; after some computations (see \cite{CastroLannes2}), one gets that up to $O(\eps\sqrt{\mu})$ terms, the symmetric tensor ${\bf E}$ solves the equation
\begin{equation}\label{defeqE}
\dt {\bf E}+\eps \ovV\cdot \nabla{\bf E}+\eps \nabla\cdot \ovV {\bf E}+\eps \nabla \ovV^{\rm T}{\bf E}+\eps {\bf E}\nabla \ovV=0.
\end{equation}
The conclusion is that {\it the presence of the vorticity can be taken into account in the SGN equations without having to solve the vorticity equation} \eqref{vortND} but by extending the SGN equations by a third coupled evolution equation on ${\bf E}$. The SGN equations in $(\zeta,\ovV)$ variables \eqref{SGN_V} then become
\begin{equation}\label{SGNvort}
\begin{cases}
\dt \zeta +\nabla\cdot (h\ovV)=0,\\
(1+\mu {\mathcal T})\big[ \dt \ovV +\eps \ovV\cdot \nabla \ovV \big] +\nabla\zeta +\eps \mu  {\mathcal Q}_1 (h ,\ovV)+ \eps \mu \frac{1}{h}\nabla\cdot {\bf E}=0,\\
\dt {\bf E}+\eps \ovV\cdot \nabla{\bf E}+\eps \nabla\cdot \ovV {\bf E}+\eps \nabla \ovV^{\rm T}{\bf E}+\eps {\bf E}\nabla \ovV=0.
\end{cases}
\end{equation}
\begin{remark}\label{remSGNvort}
Contrary to \eqref{SGN_V} which are precise up to $O(\mu^2)$ terms, the above equations are precise up to $O(\eps\mu^{3/2})$ terms only. The $O(\mu^2)$ precision is reached in \cite{CastroLannes2}, but the equations \eqref{SGNvort} must be further extended by two other coupled evolution equations: one is  the third order turbulent tensor ${\bf F}$ and the other one is the second order momentum  $V^\sharp$ of the fluctuation of the shear velocity,
$$
{\bf F}=\int_{-1}^{\eps \zeta} V^*_{\rm sh}\otimes V^*_{\rm sh}\otimes V^*_{\rm sh}
\quad\mbox{ and }\quad
V^\sharp=\frac{12}{h^3}\int_{-1}^{\eps \zeta} (z+1)^2 V_{\rm sh}^* .
$$
We also refer to \cite{CastroLannes2} for generalization to non flat topographies.
\end{remark}
\begin{remark}
In the one dimensional case, the vorticity is of the form $\bom_\mu=(0,\omega,0)^T$ and the case of a constant vorticity $\omega=\omega_0={\rm const}$ is of particular interest. Indeed, one then has ${\bf E}=\frac{1}{12} h^3 \omega_0^2$ and one can check that the equation on ${\bf E}$ reduces to the mass conservation equation. The system \eqref{SGN} reduces therefore to a system of two equations, namely,
$$
\begin{cases}
\dt \zeta +\dx (h\ovv)=0,\\
\big(1-\frac{\mu}{3} \dx (h^3 \dx \cdot)\big )\big[ \dt \ovv +\eps \ovv \dx  \ovv \big] +\dx \zeta +\eps \mu \frac{2}{3h} \dx \big(h^3 (\dx \ovv)^2 \big) + \eps^2 \mu \frac{\omega_0^2}{4 } h\dx \zeta =0
\end{cases}
$$
which can be related to the system derived in \cite{Choi:2003fv} for the propagation of waves above a constant shear through a change of variable for the velocity (the averaged velocity considered in \cite{Choi:2003fv} is not the vertical average of the full horizontal velocity but of its difference with the constant shear).
\end{remark}

Some qualitative analysis (e.g. existence of solitary waves) and numerical simulations have been performed in \cite{LannesMarche} for the one dimensional version of \eqref{SGNvort} as well as for some of the more complex models mentioned in Remark \ref{remSGNvort} but the mathematical analysis of these models remains to be done.

 As shown in \cite{CastroLannes2}, the equations \eqref{SGNvort} admit a local conservation of energy, the energy density being here the sum of the energy density associated to the irrotational SGN equations and of a rotation (or turbulent) energy ${\mathfrak e}_{\rm rot}$; a similar correction must also be made for the energy flux, so that \eqref{locNRJNSW} becomes
\begin{equation}\label{locNRJNSWrot}
\dt \big({\mathfrak e}_{\rm NSW}+{\mathfrak e}_{\rm rot}\big)+ \nabla\cdot \big( {\mathfrak F}_{\rm NSW}+ {\mathfrak F}_{\rm rot}\big)=0,
\end{equation}
where 
$$
{\mathfrak e}_{\rm rot}=\frac{1}{2}{\rm Tr}{\bf E}
\quad \mbox{ and }\quad
{\mathfrak F}_{\rm rot}=\frac{1}{2}{\rm Tr}{\bf E}\ovV+{\bf E}\ovV.
$$
Therefore, there is local conservation of the total energy, which is the sum of the irrotational one ${\mathfrak e}_{\rm SGN}$ and of  the rotational one ${\mathfrak e}_{\rm rot}$. There can therefore be a transfer of energy between both quantities. It is therefore tempting to try to model wave breaking --during which the mechanical energy (i.e. the sum of the potential and kinetic energies)  of the waves is dissipated-- by a mechanism that would ensure such a transfer to the turbulent energy at wave breaking.

\subsection{Wave breaking and enstrophy creation}\label{WB}

The derivation of \eqref{SGNvort} is rigorously justified by the uniform bounds derived in \cite{CastroLannes1} on the solution of \eqref{ZCSvort}. This rigorous approach breaks down when singularities form, and in particular when wave breaking occurs. The models proposed below are therefore far from being mathematically justified and comparison with experimental observations is at this day the best way to validate them.

Various formal approaches have been proposed to extend the range of application of SGN types models to realistic physical configurations in coastal oceanography, where wave breaking obviously has to be taken into account. It has for instance been proposed to switch locally (in the vicinity of wave breaking) from the SGN equations to the NSW equation  \cite{tonelli,BCLM,filippini,duran2017}, and to treat wave breaking as shocks  (see \S \ref{sectweakNSW}), a difficulty being to find good "breaking criteria" telling us when to switch to and back the NSW equation \cite{Tissier}. Another common approach (see for instance \cite{schaffer,svendsen, Kennedy}) is to model wave breaking by the addition of an eddy viscosity near wave breaking. Here again, one needs a "breaking criterion" to tell us when and where to add this eddy viscosity, and one must also propose an expression for this eddy viscosity, which can for instance be based on hyperbolic shock wave theory \cite{Guermond} or other physical considerations \cite{Musumeci}. We refer to \cite{brocchini,kazolea2} for surveys on these questions.

There is an intense research activity around these topics and at this day, no conclusive solution has been found. A seductive approach based on a series of works \cite{RG1,RG2,RG3,Kazakova2,Duran} is based on the idea mentioned above of a transfer mechanism between mechanical and turbulent energy. We describe this approach (and more specifically \cite{Kazakova2,Duran}) below with the formalism developed throughout these notes. For the sake of clarity, we stick here to the one dimensional case and a flat bottom.

To start with, let us rewrite \eqref{SGNvort} in dimension $d=1$; the turbulent tensor ${\bf E}$ is then a scalar, denoted $E$ and it is convenient to introduce, as in \cite{RG1,RG2} the enstrophy $\varphi=\frac{1}{h^3}E$, so that \eqref{SGNvort} can be written
$$
\begin{cases}
\dt \zeta +\dx (h\ovv)=0,\\
(1+\mu {\mathcal T})\big[ \dt \ovv+\eps \ovv\dx\ovv  \big] +\dx\zeta +\eps \mu  {\mathcal Q}_1 (h ,\ovv)+ \eps \mu \frac{1}{h}\dx (h^3 \varphi )  =0,\\
\dt (h\varphi)+\eps \dx (\ovv \varphi )=0.
\end{cases}
$$
The first step proposed in \cite{Kazakova2} is to add an eddy viscosity term in the last term of the momentum equation
$$
(1+\mu {\mathcal T})\big[ \dt \ovv+\eps \ovv\dx\ovv  \big] +\dx\zeta +\eps \mu  {\mathcal Q}_1 (h ,\ovv)+ \eps \mu \frac{1}{h}\dx (h^3 \varphi - \nu_{T} h \dx \ovv)  =0,
$$
where the eddy viscosity coefficient $\nu_{T}$ is discussed below and is a source of energy dissipation, as illustrated by the fact the the energy conservation law \eqref{locNRJNSWrot} becomes
$$
\dt \big({\mathfrak e}_{\rm NSW}+{\mathfrak e}_{\rm rot}\big)+ \dx\big( {\mathfrak F}_{\rm NSW}+ {\mathfrak F}_{\rm rot}\big)= - \eps \mu\nu_T h (\dx \ovv)^2;
$$
there is therefore a dissipation of the total energy while we rather want, at first order, a conservation of this total energy, and a transfer from the mechanical energy ${\mathfrak e}_{SGN}$ to the turbulent energy ${\mathfrak e}_{\rm rot}$. This can only be achieved through the creation of a corresponding source term in the equation for the enstrophy, namely,
$$
\dt (h\varphi)+\eps \dx (\ovv \varphi )=2 \eps \mu\nu_T \frac{1}{h} (\dx \ovv)^2;
$$
quite obviously, enstrophy (or, equivalently, turbulent energy), is created in the vicinity of wave breaking, where the gradient of the velocity becomes important; the mechanical energy of the wave is consequently decreased. This mechanism restores the local conservation of the total energy \eqref{locNRJNSWrot}. However, in a second step, the small scale dissipation of the total energy must be taken into account; there should therefore be a dissipation mechanism ${\mathcal D}$ such that
$$
\dt \big({\mathfrak e}_{\rm NSW}+{\mathfrak e}_{\rm rot}\big)+ \nabla\cdot \big( {\mathfrak F}_{\rm NSW}+ {\mathfrak F}_{\rm rot}\big)=-{\mathcal D}.
$$
Assuming that this dissipation mechanisms acts at the level of the turbulent energy, one must consequently modify the enstrophy equation that becomes
$$
\dt (h\varphi)+\eps \dx (\ovv \varphi )= \eps \mu\nu_T \frac{2}{h} (\dx \ovv)^2- \frac{2}{h}{\mathcal D}.
$$
The final equations then become
\begin{equation}\label{SGNvort1dtransf}
\begin{cases}
\dt \zeta +\dx (h\ovv)=0,\\
(1+\mu {\mathcal T})\big[ \dt \ovv+\eps \ovv\dx\ovv  \big] +\dx\zeta +\eps \mu  {\mathcal Q}_1 (h ,\ovv)+ \eps \mu \frac{1}{h}\dx (h^3 \varphi -\nu_T h \dx \overline{v})  =0,\\
\dt (h\varphi)+\eps \dx (\ovv \varphi )= \eps \mu\nu_T \frac{2}{h} (\dx \ovv)^2- \frac{2}{h}{\mathcal D}.
\end{cases}
\end{equation}

\begin{remark}
The derivation of \eqref{SGNvort1dtransf} relies on quite sound physical arguments but a good amount of physical modeling is still required to propose expressions for the eddy viscosity $\nu_T$ and the dissipation term ${\mathcal D}$. There is still no consensus regarding what these terms should be. For instance, $\nu_T=C_\nu h \sqrt{gh}$ is proposed in \cite{Musumeci} while \cite{Kazakova2} suggests expressions based on the enstrophy,
$$
\nu_T=C_p h^2 \sqrt{\varphi}
\quad\mbox{ and }\quad
{\mathcal D}=\frac{1}{2}C_r h^2 \varphi^{3/2},
$$
with $C_p$ and $C_r$ dimensionless numerical coefficients. A drawback of this last choice is that the enstrophy (or turbulent energy) stays equal to zero if it is initially zero, but good matching with experimental data are observed  in many cases \cite{Kazakova2,Duran}.
\end{remark}

\subsection{What about larger vorticities?}\label{sectlargevort}

We considered in the previous section SGN type models derived under the assumption of a vorticity strength $\alpha=1/2$, where we recall that the vortex strength is defined in  \eqref{strength}. This is the strongest vorticity for which bounds on the solutions to the rotational water waves equations \eqref{ZCSvort} have been established uniformly with respect to $\mu\in (0,1)$ and for times of order $O(1/\eps)$ \cite{CastroLannes1}. Owing to these uniform bounds, the asymptotic expansions of \S \ref{sectinnervort} and \S \ref{sectNSWSGNvort} are rigorously justified. In this section, we consider flows with a larger vorticity strength $0<\alpha<1/2$, not covered therefore by the theoretical bounds of \cite{CastroLannes1}. The derivation of the models derived below is therefore only a formal one.

The first step is to generalize the expansion \eqref{expinnervelvort} of the inner velocity field to the case of a vorticity strength $0<\alpha<1/2$; by simple computations, one finds,
\begin{equation}\label{expinnervelvortlarge}
\begin{cases}
V=\overline{V}+\mu^\alpha V_{\rm sh}^*- \frac{1}{2}\mu \big(  (1+z)^2-\frac{1}{3} h^2 \big) \nabla\nabla\cdot \overline{V}+O(\mu^{1+\alpha}),\\
w=-\mu (1+z)\nabla\cdot \overline{V}-\mu^{1+\alpha}\nabla\cdot \int_{-1}^z V_{\rm sh}^*+O(\mu^2);
\end{cases}
\end{equation}
it follows that the turbulent and non hydrostatic components of the averaged Euler equation satisfy
\begin{align*}
\eps \nabla\cdot \Rey&=\eps \mu^{2\alpha} \nabla\cdot {\bf E}+O(\eps\mu^{1+\alpha}) \\
  \frac{1}{\eps}\int_{-1}^{\eps\zeta}\nabla P_{\rm NH}&=
 \mu h  {\mathcal T}\big[\dt \ovV+\nabla\cdot \big(h  \ovV\otimes \ovV\big)\big] +\mu \eps h{\mathcal Q}_1\big(\zeta, \ovV\big)
 +O( \eps \mu^{1+\alpha}).
\end{align*}

We show below how the NSW and Boussinesq models, which were not affected by the presence of a vorticity of strength $\alpha=1/2$, have to be modified in the presence of a stronger vorticity.

\subsubsection{The NSW equations with a large vorticity}

Of particular interest is the analysis of the rotational effects on the NSW equations when $0<\alpha<1/2$. Indeed, plugging the above expansion into the averaged Euler equations \eqref{formzetaQ3V} and neglecting the $O(\mu)$ terms, one finds
$$
\begin{cases}
\dt h + \nabla\cdot (h \ovV)=0,\\
\dt \ovV +\eps \ovV\cdot \nabla \ovV + \nabla\zeta +\eps \mu^{2\alpha} \frac{1}{h}\nabla\cdot {\bf E}=0,\\
\dt {\bf E}+\eps \ovV\cdot \nabla{\bf E}+\eps \nabla\cdot \ovV {\bf E}+\eps \nabla \ovV^{\rm T}{\bf E}+\eps {\bf E}\nabla \ovV=0,
\end{cases}
$$
which are the equations derived in \cite{Gavrilyuk2012} to describe the conservative motion of compressible fluids. In its one dimensional version, it is also the first model on which a mechanism of creation of entropy has been added to model wave breaking \cite{RG1,RG2}.

\subsubsection{The Boussinesq equations with a large vorticity} \label{sectBoussinesqlarge}

Under the assumption \eqref{weakNL} of weak nonlinearity, we can plug the above expansions into \eqref{formzetaQ3V} and neglect the $O(\mu^2)$ terms to obtain
\begin{equation}\label{Boussinesqvort}
\begin{cases}
\dt \zeta+\nabla\cdot  (h \ovV)=0,\\
(1-\mu \frac{1}{3}\nabla \nabla^{\rm T})\dt \ovV +\eps \ovV\cdot \nabla \ovV +\nabla \zeta+\eps \mu^{2\alpha}\nabla\cdot {\bf E}=0,\\
\dt {\bf E}+\eps \ovV\cdot \nabla{\bf E}+\eps \nabla\cdot \ovV {\bf E}+\eps \nabla \ovV^{\rm T}{\bf E}+\eps {\bf E}\nabla \ovV=0.
\end{cases}
\end{equation}
\begin{remark}
Contrary to what has been done in \S \ref{sectBoussinesq} in the irrotational case, it is not possible here to replace $(1-\mu \frac{1}{3}\nabla \nabla^{\rm T})\dt \ovV $ in the second equation by the simpler expression $(1-\mu \frac{1}{3}\Delta )\dt \ovV $. Indeed, the quantity $\nabla^\perp \cdot \ovV$ is no longer small enough to perform such a substitution.
\end{remark}

\appendix
\section{Generalized formula when the topography is not flat}\label{Appbottom}

For the sake of clarity, in many cases, we provided in the main text formulas for a flat topography. We give here the generalization of these formulas when the bottom is not flat. 

First, in the presence of a non flat topography, the expansion \eqref{expinnervel} for an irrotational flow must be replaced by 
\begin{equation}\label{expinnerveltopo}
\begin{cases}
V=&\overline{V}-\frac{1}{2}\mu \big((1+z-\beta b)^2-\frac{1}{3}h^2\big)\nabla\nabla\cdot \overline{V}\\
&+\beta \big( z-\eps\zeta+\frac{1}{2}h \big)\big[\nabla b\cdot \nabla\ovV+\nabla(\nabla b\cdot\ovV )\big]+O(\mu^2),\\
w=&-\mu\nabla\cdot \big[ (1+z-\beta b)\overline{V}\big]+O(\mu^2);
\end{cases}
\end{equation}
and, similarly, for the description of the pressure field in the fluid domain we now have
\begin{align}
\nonumber
\frac{1}{\eps} P_{\rm NH}&=-\mu\big[ \frac{h^2}{2}-\frac{(1+z-\beta b)^2}{2}\big]  \big(\dt +\eps \ovV\cdot \nabla -\eps\nabla\cdot \ovV\big)\nabla\cdot\ovV\\
\label{expPtopo}
&+\mu(\eps \zeta-z) h (\dt +\eps \ovV\cdot \nabla)(\beta \nabla b\cdot \ovV)+O(\mu^2).
\end{align}

The same procedure as in \S \ref{sectSGN} then leads to the following SGN equations in the presence of topography,
\begin{equation}\label{SGNtopo}
\begin{cases}
\dt \zeta +\nabla\cdot Q=0,\\
(1+\mu {\bf T})\big[ \dt Q +\nabla\cdot \big( \frac{1}{h} Q\otimes Q \big) \big] +h\nabla\zeta +h {\mathcal Q}_1 (h ,\frac{Q}{h})=0,
\end{cases}
\end{equation}
where ${\bf T}=h {\mathcal T} \frac{1}{h}$
and 
\begin{align*}
{\mathcal T}V=& -\frac{1}{3h}\nabla\big( h^3 \nabla\cdot V \big)\\
&+\beta \frac{1}{2h}\big[ \nabla \big( h^2 \nabla b \cdot V \big)-h^2 \nabla b\nabla\cdot V \big]+\beta^2 \nabla b \nabla b \cdot V,
\end{align*}
while
$$
{\mathcal Q}_1(V)=-2 {\mathcal R}_1 \big(\dx V\cdot \dy V^\perp+(\nabla\cdot V)^2\big)+\beta {\mathcal R}_2\big( V\cdot (V\cdot \nabla)\nabla b \big)
$$
and
$$
{\mathcal R}_1 w=-\frac{1}{3h}\nabla (h^3 w)-\beta \frac{h}{2} w \nabla b, \qquad
{\mathcal R}_2 w = \frac{1}{2h}\nabla (h^2 w)+\beta w \nabla b.
$$

Finally, in the presence of a vorticity of strength $\alpha=1/2$, the expansion of the velocity field is
\begin{equation}\label{expinnervelvorttopo}
\begin{cases}
V=&\overline{V}+\sqrt{\mu} V_{\rm sh}^*-\frac{1}{2}\mu \big((1+z-\beta b)^2-\frac{1}{3}h^2\big)\nabla\nabla\cdot \overline{V}\\
&+\beta \big( z-\eps\zeta+\frac{1}{2}h \big)\big[\nabla b\cdot \nabla\ovV+\nabla(\nabla b\cdot\ovV )\big]+O(\mu^2),\\
w=&-\mu\nabla\cdot \big[ (1+z-\beta b)\overline{V}\big]-\mu^{3/2}\nabla\cdot \int_{-1+\beta b}^z V_{\rm sh}^*+O(\mu^2).
\end{cases}
\end{equation}
\bibliographystyle{acm}
\bibliography{NON}

\begin{thebibliography}{100}

\bibitem{abgrall2010}
{\sc Abgrall, R., and Karni, S.}
\newblock A comment on the computation of non-conservative products.
\newblock {\em Journal of Computational Physics 229}, 8 (2010), 2759--2763.

\bibitem{Adamy}
{\sc Adamy, K.}
\newblock Existence of solutions for a {B}oussinesq system on the half line and
  on a finite interval.
\newblock {\em Discrete \& Continuous Dynamical Systems-A 29}, 1 (2011),
  25--49.

\bibitem{Alazard2014}
{\sc Alazard, T., Burq, N., and Zuily, C.}
\newblock On the cauchy problem for gravity water waves.
\newblock {\em Inventiones mathematicae 198}, 1 (2014), 71--163.

\bibitem{AG}
{\sc Alinhac, S., and G{\'e}rard, P.}
\newblock {\em Op\'erateurs pseudo-diff\'erentiels et th\'eor\`eme de
  {N}ash-{M}oser}.
\newblock Savoirs Actuels. [Current Scholarship]. InterEditions, Paris, 1991.

\bibitem{Alvarez}
{\sc Alvarez-Samaniego, B., and Lannes, D.}
\newblock Large time existence for 3d water-waves and asymptotics.
\newblock {\em Invent. math. 171\/} (2008), 485--541.

\bibitem{Alvarez2}
{\sc Alvarez-Samaniego, B., and Lannes, D.}
\newblock A {N}ash-{M}oser theorem for singular evolution equations.
  {A}pplication to the {S}erre and {G}reen-{N}aghdi equations.
\newblock {\em Indiana Univ. Math. J. 57}, 1 (2008), 97--131.

\bibitem{Amick}
{\sc Amick, C.~J.}
\newblock Regularity and uniqueness of solutions to the {B}oussinesq system of
  equations.
\newblock {\em J. Differential Equations 54}, 2 (1984), 231--247.

\bibitem{Antonopoulos}
{\sc Antonopoulos, D.~C., Dougalis, V., and Mitsotakis, D.~E.}
\newblock Initial-boundary-value problems for the {B}ona-{S}mith family of
  {B}oussinesq systems.
\newblock {\em Advances in Differential Equations 14}, 1/2 (2009), 27--53.

\bibitem{audusse2011}
{\sc Audusse, E., Bristeau, M.-O., Perthame, B., and Sainte-Marie, J.}
\newblock A multilayer {S}aint-{V}enant system with mass exchanges for shallow
  water flows. derivation and numerical validation.
\newblock {\em ESAIM: Mathematical Modelling and Numerical Analysis 45}, 1
  (2011), 169--200.

\bibitem{bai2013dispersion}
{\sc Bai, Y., and Cheung, K.~F.}
\newblock Dispersion and nonlinearity of multi-layer non-hydrostatic
  free-surface flow.
\newblock {\em Journal of Fluid Mechanics 726\/} (2013), 226--260.

\bibitem{bai2015dispersion}
{\sc Bai, Y., and Cheung, K.~F.}
\newblock Dispersion and kinematics of multi-layer non-hydrostatic models.
\newblock {\em Ocean Modelling 92\/} (2015), 11--27.

\bibitem{Melinand}
{\sc Benjamin, M.}
\newblock A mathematical study of meteo and landslide tsunamis: The {P}roudman
  resonance.
\newblock {\em Nonlinearity 28}, 11 (2015), 4037.

\bibitem{BBM}
{\sc Benjamin, T.~B., Bona, J.~L., and Mahony, J.~J.}
\newblock Model equations for long waves in nonlinear dispersive systems.
\newblock {\em Philos. Trans. Roy. Soc. London Ser. A 272}, 1220 (1972),
  47--78.

\bibitem{BG}
{\sc Benzoni-Gavage, S., and Serre, D.}
\newblock {\em Multi-dimensional hyperbolic partial differential equations:
  First-order Systems and Applications}.
\newblock Oxford University Press on Demand, 2007.

\bibitem{BMN}
{\sc Besse, C., M{\'e}sognon-Gireau, B., and Noble, P.}
\newblock Artificial boundary conditions for the linearized
  {B}enjamin--{B}ona--{M}ahony equation.
\newblock {\em Numerische Mathematik 139}, 2 (2018), 281--314.

\bibitem{BNS}
{\sc Besse, C., Noble, P., and Sanchez, D.}
\newblock Discrete transparent boundary conditions for the mixed {KDV--BBM}
  equation.
\newblock {\em Journal of Computational Physics 345\/} (2017), 484--509.

\bibitem{bocchi2019return}
{\sc Bocchi, E.}
\newblock On the return to equilibrium problem for axisymmetric floating
  structures in shallow water.
\newblock {\em arXiv preprint arXiv:1901.04023\/} (2019).

\bibitem{BonaChen}
{\sc Bona, J.~L., and Chen, M.}
\newblock A {B}oussinesq system for two-way propagation of nonlinear dispersive
  waves.
\newblock {\em Physica D 116\/} (1998), 191--224.

\bibitem{BonaChenSaut1}
{\sc Bona, J.~L., Chen, M., and Saut, J.-C.}
\newblock Boussinesq equations and other systems for small-amplitude long waves
  in nonlinear dispersive media. {I}. {D}erivation and linear theory.
\newblock {\em J. Nonlinear Sci. 12}, 4 (2002), 283--318.

\bibitem{BonaChenSaut2}
{\sc Bona, J.~L., Chen, M., and Saut, J.-C.}
\newblock Boussinesq equations and other systems for small amplitude long waves
  in nonlinear dispersive media: {II}. the nonlinear theory.
\newblock {\em Nonlinearity 17\/} (2004), 925--952.

\bibitem{BonaColinLannes}
{\sc Bona, J.~L., Colin, T., and Lannes, D.}
\newblock Long wave approximations for water waves.
\newblock {\em Arch. Ration. Mech. Anal. 178}, 3 (2005), 373--410.

\bibitem{Bona:2008ff}
{\sc Bona, J.~L., Lannes, D., and Saut, J.-C.}
\newblock Asymptotic models for internal waves.
\newblock {\em J. Math. Pures Appl. 89\/} (2008), 538--566.

\bibitem{bonasmith}
{\sc Bona, J.~L., and Smith, R.}
\newblock A model for the two-way propagation of water waves in a channel.
\newblock In {\em Mathematical Proceedings of the Cambridge Philosophical
  Society\/} (1976), vol.~79, Cambridge University Press, pp.~167--182.

\bibitem{Bonneton2010}
{\sc Bonneton, P., Bruneau, N., and Castelle, B.}
\newblock Large-scale vorticity generation due to dissipating waves in the surf
  zone.
\newblock {\em Discrete Contin. Dyn. Syst. Ser. B 13\/} (2010), 729--738.

\bibitem{BCLM}
{\sc Bonneton, P., Chazel, F., Lannes, D., Marche, F., and Tissier, M.}
\newblock A splitting approach for the fully nonlinear and weakly dispersive
  {G}reen--{N}aghdi model.
\newblock {\em J. Comput. Phys 230\/} (2011), 1479--1498.

\bibitem{BonnetonLannes2017}
{\sc Bonneton, P., and Lannes, D.}
\newblock Recovering water wave elevation from pressure measurements.
\newblock {\em Journal of Fluid Mechanics 833\/} (2017), 399--429.

\bibitem{Bonneton2018}
{\sc Bonneton, P., Lannes, D., Martins, K., and Michallet, H.}
\newblock A nonlinear weakly dispersive method for recovering the elevation of
  irrotational surface waves from pressure measurements.
\newblock {\em Coastal Engineering 138\/} (2018), 1--8.

\bibitem{bosi2019spectral}
{\sc Bosi, U., Engsig-Karup, A.~P., Eskilsson, C., and Ricchiuto, M.}
\newblock A spectral/hp element depth-integrated model for nonlinear wave--body
  interaction.
\newblock {\em Computer Methods in Applied Mechanics and Engineering 348\/}
  (2019), 222--249.

\bibitem{bouchut2015}
{\sc Bouchut, F., Fernandez-Nieto, E.~D., Mangeney, A., and Narbona-Reina, G.}
\newblock A two-phase shallow debris flow model with energy balance.
\newblock {\em ESAIM: Mathematical Modelling and Numerical Analysis 49}, 1
  (2015), 101--140.

\bibitem{Bresch2009}
{\sc Bresch, D.}
\newblock Shallow-water equations and related topics.
\newblock {\em Handbook of differential equations: evolutionary equations 5\/}
  (2009), 1--104.

\bibitem{bresch2019waves}
{\sc Bresch, D., Lannes, D., and Metivier, G.}
\newblock Waves interacting with a partially immersed obstacle in the
  boussinesq regime.
\newblock {\em arXiv preprint arXiv:1902.04837\/} (2019).

\bibitem{breschmetivier}
{\sc Bresch, D., and M{\'e}tivier, G.}
\newblock Anelastic limits for {E}uler-type systems.
\newblock {\em Appl. Math. Res. Express. AMRX}, 2 (2010), 119--141.

\bibitem{brocchini}
{\sc Brocchini, M.}
\newblock A reasoned overview on boussinesq-type models: the interplay between
  physics, mathematics and numerics.
\newblock {\em Proceedings of the Royal Society A: Mathematical, Physical and
  Engineering Sciences 469}, 2160 (2013), 20130496.

\bibitem{buckmaster2019}
{\sc Buckmaster, T., Shkoller, S., and Vicol, V.}
\newblock Formation of shocks for 2{D} isentropic compressible {E}uler.
\newblock {\em arXiv preprint arXiv:1907.03784\/} (2019).

\bibitem{Burtea2016b}
{\sc Burtea, C.}
\newblock Long time existence results for bore-type initial data for
  {BBM-Boussinesq} systems.
\newblock {\em Journal of Differential Equations 261}, 9 (2016), 4825 -- 4860.

\bibitem{Burtea2016}
{\sc Burtea, C.}
\newblock New long time existence results for a class of {B}oussinesq-type
  systems.
\newblock {\em Journal de Math{\'e}matiques Pures et Appliqu{\'e}es 106}, 2
  (2016), 203 -- 236.

\bibitem{Camassa}
{\sc Camassa, R., and Holm, D.~D.}
\newblock An integrable shallow water equation with peaked solitons.
\newblock {\em Phys. Rev. Lett. 71\/} (1993), 1661--1664.

\bibitem{CCG}
{\sc Castro, A., Cordoba, D., and Gancedo, F.}
\newblock Singularity formations for a surface wave model.
\newblock {\em Nonlinearity 23\/} (2010), 2835--2847.

\bibitem{CastroLannes2}
{\sc Castro, A., and Lannes, D.}
\newblock Fully nonlinear long-wave models in the presence of vorticity.
\newblock {\em J. Fluid Mech. 759\/} (2014), 642--675.

\bibitem{CastroLannes1}
{\sc Castro, A., and Lannes, D.}
\newblock Well-posedness and shallow-water stability for a new {H}amiltonian
  formulation of the water waves equations with vorticity.
\newblock {\em Indiana Univ. Math. J. 64\/} (2015), 1169--1270.

\bibitem{casulli}
{\sc Casulli, V., and Stelling, G.~S.}
\newblock Numerical simulation of 3d quasi-hydrostatic, free-surface flows.
\newblock {\em Journal of Hydraulic Engineering 124}, 7 (1998), 678--686.

\bibitem{Chazel}
{\sc Chazel, F.}
\newblock Influence of bottom topography on long water waves.
\newblock {\em M2AN Math. Model. Numer. Anal. 41}, 4 (2007), 771--799.

\bibitem{floarticle}
{\sc Chazel, F., Benoit, M., Ern, A., and Piperno, S.}
\newblock A double-layer boussinesq-type model for highly nonlinear and
  dispersive waves.
\newblock {\em Proceedings of The Royal Society A Mathematical Physical and
  Engineering Sciences 465\/} (03 2009).

\bibitem{Chazel2011}
{\sc Chazel, F., Lannes, D., and Marche, F.}
\newblock Numerical simulation of strongly nonlinear and dispersive waves using
  a {G}reen-{N}aghdi model.
\newblock {\em J. Sci. Comput. 48}, 1-3 (2011), 105--116.

\bibitem{Chen2012}
{\sc Chen, G.-Q., and Perepelitsa, M.}
\newblock Shallow water equations: viscous solutions and inviscid limit.
\newblock {\em Zeitschrift f{\"u}r angewandte Mathematik und Physik 63}, 6
  (2012), 1067--1084.

\bibitem{Choi1995}
{\sc Choi, W.}
\newblock Nonlinear evolution equations for two-dimensional surface waves in a
  fluid of finite depth.
\newblock {\em J. Fluid Mech. 295\/} (1995), 381--394.

\bibitem{Choi:2003fv}
{\sc Choi, W.}
\newblock Strongly nonlinear long gravity waves in uniform shear flows.
\newblock {\em Phys. Rev. E 68\/} (2003), 26305.

\bibitem{christodoulou2014}
{\sc Christodoulou, D., and Miao, S.}
\newblock {\em Compressible Flow and Euler's Equations}, vol.~9.
\newblock International Press Somerville, MA, 2014.

\bibitem{ClamondConstantin}
{\sc Clamond, D., and Constantin, A.}
\newblock Recovery of steady periodic wave profiles from pressure measurements
  at the bed.
\newblock {\em Journal of Fluid Mechanics 714\/} (2013), 463--475.

\bibitem{ConstantinEscher}
{\sc Constantin, A., and Escher, J.}
\newblock Wave breaking for nonlinear nonlocal shallow water equations.
\newblock {\em Acta Math. 181}, 2 (1998), 229--243.

\bibitem{ConstantinLannes}
{\sc Constantin, A., and Lannes, D.}
\newblock The hydrodynamical relevance of the {C}amassa-{H}olm and
  {D}egasperis-{P}rocesi equations.
\newblock {\em Arch. Ration. Mech. Anal. 192}, 1 (2009), 165--186.

\bibitem{Coulombel2003}
{\sc Coulombel, J.-F.}
\newblock Stability of multidimensional undercompressive shock waves.
\newblock {\em Interfaces and Free Boundaries 5}, 4 (2003), 367--390.

\bibitem{Coutand2007}
{\sc Coutand, D., and Shkoller, S.}
\newblock Well-posedness of the free-surface incompressible euler equations
  with or without surface tension.
\newblock {\em Journal of the American Mathematical Society 20}, 3 (2007),
  829--930.

\bibitem{CoutandShkoller1}
{\sc Coutand, D., and Shkoller, S.}
\newblock Well-posedness in smooth function spaces for moving-boundary 1-{D}
  compressible {E}uler equations in physical vacuum.
\newblock {\em Comm. Pure Appl. Math. 64}, 3 (2011), 328--366.

\bibitem{CoutandShkoller2}
{\sc Coutand, D., and Shkoller, S.}
\newblock Well-posedness in smooth function spaces for the moving-boundary
  three-dimensional compressible {E}uler equations in physical vacuum.
\newblock {\em Arch. Ration. Mech. Anal. 206}, 2 (2012), 515--616.

\bibitem{Craig2006}
{\sc Craig, W., Guyenne, P., Hammack, J., Henderson, D., and Sulem, C.}
\newblock Solitary water wave interactions.
\newblock {\em Phys. Fluids 18}, 5 (2006), 057106, 25.

\bibitem{Craig:2005zl}
{\sc Craig, W., Guyenne, P., and Kalisch, H.}
\newblock Hamiltonian long-wave expansions for free surfaces and interfaces.
\newblock {\em Communication on Pure and Applied Mathematics 58}, 12 (2005),
  1587--1641.

\bibitem{Craig2}
{\sc Craig, W., and Sulem, C.}
\newblock Numerical simulation of gravity waves.
\newblock {\em J. Comput. Phys 108\/} (1993), 73--83.

\bibitem{Craig1}
{\sc Craig, W., Sulem, C., and Sulem, P.-L.}
\newblock Nonlinear modulation of gravity waves: a rigorous approach.
\newblock {\em Nonlinearity 5}, 2 (1992), 497--522.

\bibitem{de2019priori}
{\sc de~Poyferr{\'e}, T.}
\newblock A priori estimates for water waves with emerging bottom.
\newblock {\em Archive for Rational Mechanics and Analysis 232}, 2 (2019),
  763--812.

\bibitem{DP}
{\sc Degasperis, A., and Procesi, M.}
\newblock Asymptotic integrability.
\newblock {\em Symmetry and perturbation theory 1}, 1 (1999), 23--37.

\bibitem{Delort}
{\sc Delort, J.-M.}
\newblock Long time existence results for solutions of water waves equations.
\newblock In {\em Proc. Int. Cong. of Math.\/} (2018), vol.~2, pp.~2209--2228.

\bibitem{Diperna}
{\sc DiPerna, R.~J.}
\newblock Convergence of the viscosity method for isentropic gas dynamics.
\newblock {\em Communications in mathematical physics 91}, 1 (1983), 1--30.

\bibitem{dougalis2009}
{\sc Dougalis, V., Mitsotakis, D., and Saut, J.}
\newblock On initial-boundary value problems for a boussinesq system of bbm-bbm
  type in a plane domain.
\newblock {\em Discrete Contin. Dyn. Syst 23}, 4 (2009), 1191--1204.

\bibitem{dougalis2009bis}
{\sc Dougalis, V., Mitsotakis, D., and Saut, J.}
\newblock On initial-boundary value problems for a boussinesq system of bbm-bbm
  type in a plane domain.
\newblock {\em Discrete Contin. Dyn. Syst 23}, 4 (2009), 1191--1204.

\bibitem{dougalis2010}
{\sc Dougalis, V., Mitsotakis, D., and Saut, J.-C.}
\newblock Boussinesq systems of bona-smith type on plane domains: theory and
  numerical analysis.
\newblock {\em Journal of Scientific Computing 44}, 2 (2010), 109--135.

\bibitem{duchene2011}
{\sc Duch{\^e}ne, V.}
\newblock Asymptotic models for the generation of internal waves by a moving
  ship, and the dead-water phenomenon.
\newblock {\em Nonlinearity 24}, 8 (2011), 2281.

\bibitem{Duchene}
{\sc Duch{\^e}ne, V.}
\newblock Rigorous justification of the {F}avrie--{G}avrilyuk approximation to
  the {S}erre--{G}reen--{N}aghdi model.
\newblock {\em Nonlinearity 32}, 10 (2019), 3772.

\bibitem{DucheneIguchi}
{\sc Duch\^ene, V., and Iguchi, T.}
\newblock A {H}amiltonian structure of the {I}sobe-{K}akinuma model for water
  waves.
\newblock {\em submitted\/} (2019).

\bibitem{duchene2016}
{\sc Duch{\^e}ne, V., and Israwi, S.}
\newblock Well-posedness of the {G}reen-{N}aghdi and {B}oussinesq-{P}eregrine
  systems.
\newblock {\em arXiv preprint arXiv:1611.04305\/} (2016).

\bibitem{duran2017}
{\sc Duran, A., and Marche, F.}
\newblock A discontinuous {G}alerkin method for a new class of
  {G}reen--{N}aghdi equations on simplicial unstructured meshes.
\newblock {\em Applied Mathematical Modelling 45\/} (2017), 840--864.

\bibitem{EGW}
{\sc Ehrnstrom, M., Groves, M., and Wahlen, E.}
\newblock On the existence and stability of solitary-wave solutions to a class
  of evolution equations of whitham type.
\newblock {\em Nonlinearity 25\/} (2012), 1--34.

\bibitem{EhrnstromWahlen}
{\sc Ehrnstrom, M., and Wahlen, E.}
\newblock On {W}hitham's conjecture of a highest cusped wave for a nonlocal
  dispersive equation.
\newblock {\em Ann. Henri Poincar\'e\/} (to appear).

\bibitem{favrie2017}
{\sc Favrie, N., and Gavrilyuk, S.}
\newblock A rapid numerical method for solving {S}erre--{G}reen--{N}aghdi
  equations describing long free surface gravity waves.
\newblock {\em Nonlinearity 30}, 7 (2017), 2718.

\bibitem{fernandez}
{\sc Fernandez-Nieto, E., Parisot, M., Penel, Y., and Sainte-Marie, J.}
\newblock Layer-averaged approximation of euler equations for free surface
  flows with a non-hydrostatic pressure.
\newblock {\em Commun. Math. Sci.(hal-01324012v3)\/} (2018).

\bibitem{fernandez2014}
{\sc Fern{\'a}ndez-Nieto, E.~D., Kon{\'e}, E., and Rebollo, T.~C.}
\newblock A multilayer method for the hydrostatic {N}avier-{S}tokes equations:
  a particular weak solution.
\newblock {\em Journal of Scientific Computing 60}, 2 (2014), 408--437.

\bibitem{Filippini2015}
{\sc Filippini, A.~G., Bellec, S., Colin, M., and Ricchiuto, M.}
\newblock On the nonlinear behaviour of {B}oussinesq type models:
  Amplitude-velocity vs amplitude-flux forms.
\newblock {\em Coastal Engineering 99\/} (2015), 109--123.

\bibitem{filippini}
{\sc Filippini, A.~G., Kazolea, M., and Ricchiuto, M.}
\newblock A flexible genuinely nonlinear approach for nonlinear wave
  propagation, breaking and run-up.
\newblock {\em Journal of Computational Physics 310\/} (2016), 381--417.

\bibitem{Fokas}
{\sc Fokas, A.~S., and Fuchssteiner, B.}
\newblock Symplectic structures, their {B}acklund transformation and hereditary
  symmetries.
\newblock {\em Physica D 4\/} (1981), 821--831.

\bibitem{Freistuhler1998}
{\sc Freist{\"u}hler, H.}
\newblock Some results on the stability of non-classical shock waves.
\newblock {\em J. Partial Differential Equations 11\/} (1998), 25--38.

\bibitem{fujiwara2015}
{\sc Fujiwara, H., and Iguchi, T.}
\newblock A shallow water approximation for water waves over a moving bottom.
\newblock In {\em Nonlinear Dynamics in Partial Differential Equations\/}
  (Tokyo, Japan, 2015), Mathematical Society of Japan, pp.~77--88.

\bibitem{gallagher}
{\sc Gallagher, I., and Saint-Raymond, L.}
\newblock On the influence of the earth's rotation on geophysical flows.
\newblock {\em Handbook of Mathematical Fluid Dynamics 4\/} (2007), 201--329.

\bibitem{Gavrilyuk2012}
{\sc Gavrilyuk, S., and Gouin, H.}
\newblock Geometric evolution of the reynolds stress tensor.
\newblock {\em International Journal of Engineering Science 59\/} (2012),
  65--73.

\bibitem{GavrilyukNkonga}
{\sc Gavrilyuk, S., Nkonga, B., Shyue, K.-M., and Truskinovsky, L.}
\newblock Generalized riemann problem for dispersive equations.
\newblock hal-01958328, 2018.

\bibitem{godlewski2018congested}
{\sc Godlewski, E., Parisot, M., Sainte-Marie, J., and Wahl, F.}
\newblock Congested shallow water model: roof modeling in free surface flow.
\newblock {\em ESAIM: Mathematical Modelling and Numerical Analysis 52}, 5
  (2018), 1679--1707.

\bibitem{GreenNaghdi}
{\sc Green, A.~E., and Naghdi, P.~M.}
\newblock A derivation of equations for wave propagation in water of variable
  depth.
\newblock {\em J. Fluid Mech. 78\/} (1976), 237--246.

\bibitem{Guermond}
{\sc Guermond, J.-L., Pasquetti, R., and Popov, B.}
\newblock Entropy viscosity method for nonlinear conservation laws.
\newblock {\em Journal of Computational Physics 230}, 11 (2011), 4248--4267.

\bibitem{Henry}
{\sc Henry, D.}
\newblock On the pressure transfer function for solitary water waves with
  vorticity.
\newblock {\em Mathematische Annalen 357}, 1 (2013), 23--30.

\bibitem{HenryThomas}
{\sc Henry, D., and Thomas, G.}
\newblock Prediction of the free-surface elevation for rotational water waves
  using the recovery of pressure at the bed.
\newblock {\em Philosophical Transactions of the Royal Society A: Mathematical,
  Physical and Engineering Sciences 376}, 2111 (2017), 20170102.

\bibitem{hu2009global}
{\sc Hu, J.}
\newblock Global well-posedness of the {BCL} system with viscosity.
\newblock {\em Chinese Annals of Mathematics, Series B 30}, 2 (2009), 153--172.

\bibitem{Hur}
{\sc Hur, V.}
\newblock Wave breaking in the {W}hitham equation.
\newblock {\em Advances in Mathematics 317\/} (2017), 410--437.

\bibitem{Iguchi2009}
{\sc Iguchi, T.}
\newblock A shallow water approximation for water waves.
\newblock {\em J. Math. Kyoto Univ. 49}, 1 (2009), 13--55.

\bibitem{iguchi2018}
{\sc Iguchi, T.}
\newblock Isobe--kakinuma model for water waves as a higher order shallow water
  approximation.
\newblock {\em Journal of Differential Equations 265}, 3 (2018), 935--962.

\bibitem{iguchi2}
{\sc Iguchi, T.}
\newblock A mathematical justification of the {I}sobe--{K}akinuma model for
  water waves with and without bottom topography.
\newblock {\em Journal of Mathematical Fluid Mechanics 20}, 4 (2018),
  1985--2018.

\bibitem{IguchiLannes}
{\sc Iguchi, T., and Lannes, D.}
\newblock Hyperbolic free boundary problems and applications to wave-structure
  interactions.
\newblock {\em Indiana Univ. Math. J.\/} (to appear).

\bibitem{ionescu}
{\sc Ionescu, A.~D., and Pusateri, F.}
\newblock Recent advances on the global regularity for irrotational water
  waves.
\newblock {\em Philosophical Transactions of the Royal Society A: Mathematical,
  Physical and Engineering Sciences 376}, 2111 (2017), 20170089.

\bibitem{isobe1995}
{\sc Isobe, M.}
\newblock Time-dependent mild-slope equations for random waves.
\newblock In {\em Coastal Engineering 1994}. 1995, pp.~285--299.

\bibitem{Israwi2010b}
{\sc Israwi, S.}
\newblock Derivation and analysis of a new $2d$ {Green-Naghdi} system.
\newblock {\em Nonlinearity 23\/} (2010), 2889--2904.

\bibitem{Israwi2010}
{\sc Israwi, S.}
\newblock Variable depth {K}d{V} equations and generalizations to more
  nonlinear regimes.
\newblock {\em M2AN Math. Model. Numer. Anal. 44}, 2 (2010), 347--370.

\bibitem{Israwi2011}
{\sc Israwi, S.}
\newblock Large time existence for 1{D} {Green-Naghdi} equations.
\newblock {\em Nonlinear Analysis: Theory, Methods \& Applications 74}, 1
  (2011), 81 -- 93.

\bibitem{JangMasmoudi1}
{\sc Jang, J., and Masmoudi, N.}
\newblock Well-posedness for compressible {E}uler equations with physical
  vacuum singularity.
\newblock {\em Comm. Pure Appl. Math. 62}, 10 (2009), 1327--1385.

\bibitem{JangMasmoudi2}
{\sc Jang, J., and Masmoudi, N.}
\newblock Well-posedness of compressible {E}uler equations in a physical
  vacuum.
\newblock {\em Comm. Pure Appl. Math. 68}, 1 (2015), 61--111.

\bibitem{Johnson1973}
{\sc Johnson, R.}
\newblock On the development of a solitary wave moving over an uneven bottom.
\newblock In {\em Mathematical Proceedings of the Cambridge Philosophical
  Society\/} (1973), vol.~73, Cambridge University Press, pp.~183--203.

\bibitem{kakinuma2001}
{\sc Kakinuma, T.}
\newblock A set of fully nonlinear equations for surface and internal gravity
  waves.
\newblock {\em WIT Transactions on The Built Environment 58\/} (2001).

\bibitem{Kazakova}
{\sc Kazakova, M.}
\newblock {\em Dispersive models of ocean waves propagation: Numerical issues
  and mod- elling}.
\newblock PhD thesis, Universit{\'e} de Toulouse, 2018.

\bibitem{KazakovaNoble}
{\sc Kazakova, M., and Noble, P.}
\newblock Discrete transparent boundary conditions for the linearized
  {G}reen-{N}aghdi system of equations.
\newblock arXiv: 1710.04016, 2017.

\bibitem{Kazakova2}
{\sc Kazakova, M., and Richard, G.~L.}
\newblock A new model of shoaling and breaking waves: one-dimensional solitary
  wave on a mild sloping beach.
\newblock {\em Journal of Fluid Mechanics 862\/} (2019), 552--591.

\bibitem{kazolea2}
{\sc Kazolea, M., and Ricchiuto, M.}
\newblock On wave breaking for {B}oussinesq-type models.
\newblock {\em Ocean Modelling 123\/} (2018), 16--39.

\bibitem{Kennedy}
{\sc Kennedy, A.~B., Chen, Q., Kirby, J.~T., and Dalrymple, R.~A.}
\newblock Boussinesq modeling of wave transformation, breaking, and runup. {I}:
  1d.
\newblock {\em J. Wtrwy., Port, Coast., and Oc. Engrg. 126}, 1 (2000), 39--47.

\bibitem{Kim}
{\sc Kim, J.~W., Bai, K.~J., Ertekin, R.~C., and Webster, W.~C.}
\newblock A derivation of the {G}reen-{N}aghdi equations for irrotational
  flows.
\newblock {\em J. Engrg. Math. 40}, 1 (2001), 17--42.

\bibitem{klein2018whitham}
{\sc Klein, C., Linares, F., Pilod, D., and Saut, J.-C.}
\newblock On {W}hitham and related equations.
\newblock {\em Studies in Applied Mathematics 140}, 2 (2018), 133--177.

\bibitem{klein2012numerical}
{\sc Klein, C., and Saut, J.-C.}
\newblock Numerical study of blow up and stability of solutions of generalized
  kadomtsev--petviashvili equations.
\newblock {\em Journal of nonlinear science 22}, 5 (2012), 763--811.

\bibitem{KleinSaut}
{\sc Klein, C., and Saut, J.-C.}
\newblock {\em Nonlinear dispersive {PDE}'s. Inverse Scattering and {PDE}
  Methods}.
\newblock to appear, 2020.

\bibitem{kwak2019b}
{\sc Kwak, C., and Mu{\~n}oz, C.}
\newblock Asymptotic dynamics for the small data weakly dispersive
  one-dimensional {H}amiltonian abcd system.
\newblock {\em arXiv preprint arXiv:1902.00454\/} (2019).

\bibitem{kwak2019}
{\sc Kwak, C., Mu{\~n}oz, C., Poblete, F., and Pozo, J.~C.}
\newblock The scattering problem for {H}amiltonian abcd {B}oussinesq systems in
  the energy space.
\newblock {\em Journal de Math{\'e}matiques Pures et Appliqu{\'e}es 127\/}
  (2019), 121--159.

\bibitem{Lannes2005}
{\sc Lannes, D.}
\newblock Well-posedness of the water-waves equations.
\newblock {\em J. Amer. Math. Soc. 18}, 3 (2005), 605--654 (electronic).

\bibitem{Lannes2013}
{\sc Lannes, D.}
\newblock {\em The Water Waves Problem: Mathematical Analysis and Asymptotics},
  vol.~188 of {\em Mathematical Surveys and Monographs}.
\newblock AMS, 2013.

\bibitem{Lannes2017}
{\sc Lannes, D.}
\newblock On the dynamics of floating structures.
\newblock {\em Annals of PDE 3}, 1 (2017), 11.

\bibitem{BonnetonLannes}
{\sc Lannes, D., and Bonneton, P.}
\newblock Derivation of asymptotic two-dimensional time-dependent equations for
  surface water wave propagation.
\newblock {\em Physics of Fluids 21}, 1 (Jan 2009), 016601.

\bibitem{LannesMarche}
{\sc Lannes, D., and Marche, F.}
\newblock Nonlinear wave--current interactions in shallow water.
\newblock {\em Studies in Applied Mathematics 136}, 4 (2016), 382--423.

\bibitem{LannesMetivier}
{\sc Lannes, D., and M\'{e}tivier, G.}
\newblock The shoreline problem for the one-dimensional shallow water and
  {G}reen-{N}aghdi equations.
\newblock {\em J. \'{E}c. polytech. Math. 5\/} (2018), 455--518.

\bibitem{LannesSaut}
{\sc Lannes, D., and Saut, J.-C.}
\newblock Weakly transverse boussinesq systems and the kadomtsev-petviashvili
  approximation.
\newblock {\em Nonlinearity 19\/} (2006), 2853--2875.

\bibitem{lannes2013remarks}
{\sc Lannes, D., and Saut, J.-C.}
\newblock Remarks on the full dispersion kadomtsev-petviashvli equation.

\bibitem{LannesWeynans}
{\sc Lannes, D., and Weynans, L.}
\newblock Generating boundary conditions for a boussinesq system.
\newblock Tech. Rep. arxiv.org/abs/1902.03973, 2019.

\bibitem{Lax}
{\sc Lax, P.~D.}
\newblock Mathematics and physics.
\newblock {\em Bull. Amer. Math. Soc. (N.S.) 45}, 1 (2008), 135--152
  (electronic).

\bibitem{Li1985}
{\sc Li, T.-T., and Yu, W.-C.}
\newblock Boundary value problems for quasilinear hyperbolic systems.
\newblock {\em Duke University Mathematics ser. 5\/} (1985).

\bibitem{Li}
{\sc Li, Y.~A.}
\newblock A shallow-water approximation to the full water wave problem.
\newblock {\em Comm. Pure Appl. Math. 59}, 9 (2006), 1225--1285.

\bibitem{linares2012}
{\sc Linares, F., Pilod, D., and Saut, J.-C.}
\newblock Well-posedness of strongly dispersive two-dimensional surface wave
  {B}oussinesq systems.
\newblock {\em SIAM Journal on Mathematical Analysis 44}, 6 (2012), 4195--4221.

\bibitem{Lindblad}
{\sc Lindblad, H.}
\newblock Well-posedness for the motion of an incompressible liquid with free
  surface boundary.
\newblock {\em Ann. of Math. (2) 162}, 1 (2005), 109--194.

\bibitem{Lions}
{\sc Lions, P.-L., Perthame, B., and Souganidis, P.~E.}
\newblock Existence and stability of entropy solutions for the hyperbolic
  systems of isentropic gas dynamics in {E}ulerian and {L}agrangian
  coordinates.
\newblock {\em Comm. Pure Appl. Math. 49}, 6 (1996), 599--638.

\bibitem{Liu}
{\sc Liu, T.~P., and Yang, T.}
\newblock Compressible flow with vacuum and physical singularity.
\newblock {\em Methods Appl. Anal. 7\/} (2000), 495--509.

\bibitem{luk2018shock}
{\sc Luk, J., and Speck, J.}
\newblock Shock formation in solutions to the 2d compressible {Euler} equations
  in the presence of non-zero vorticity.
\newblock {\em Inventiones mathematicae 214}, 1 (2018), 1--169.

\bibitem{Luke}
{\sc Luke, J.~C.}
\newblock A variational principle for a fluid with a free surface.
\newblock {\em J. Fluid Mech. 27\/} (1967), 395--397.

\bibitem{lynett2004linear}
{\sc Lynett, P.~J., and Liu, P. L.-F.}
\newblock Linear analysis of the multi-layer model.
\newblock {\em Coastal Engineering 51}, 5-6 (2004), 439--454.

\bibitem{Ma}
{\sc Ma, G., Shi, F., and Kirby, J.~T.}
\newblock Shock-capturing non-hydrostatic model for fully dispersive surface
  wave processes.
\newblock {\em Ocean Modelling 43-44\/} (2012), 22--35.

\bibitem{Majda1}
{\sc Majda, A.}
\newblock {\em The existence of multi-dimensional shock fronts}, vol.~281.
\newblock American Mathematical Soc., 1983.

\bibitem{Majda2}
{\sc Majda, A.}
\newblock {\em The stability of multi-dimensional shock fronts}, vol.~275.
\newblock American Mathematical Soc., 1983.

\bibitem{Majda3}
{\sc Majda, A.}
\newblock {\em Compressible fluid flow and systems of conservation laws in
  several space variables}, vol.~53.
\newblock Springer Science \& Business Media, 2012.

\bibitem{Marche}
{\sc Marche, F.}
\newblock {\em Theoretical and Numerical Study of Shallow Water Models.
  Applications to Nearshore Hydrodynamics}.
\newblock PhD thesis, Universit{\'e} de Bordeaux, 2005.

\bibitem{Matsuno1992}
{\sc Matsuno, Y.}
\newblock Nonlinear evolutions of surface gravity waves on fluid of finite
  depth.
\newblock {\em Phys. Rev. Lett. 69}, 4 (1992), 609--611.

\bibitem{Matsuno1993}
{\sc Matsuno, Y.}
\newblock Nonlinear evolution of surface gravity waves over an uneven bottom.
\newblock {\em J. Fluid Mech. 249\/} (1993), 121--133.

\bibitem{matsuno}
{\sc Matsuno, Y.}
\newblock Hamiltonian structure for two-dimensional extended {Green--Naghdi}
  equations.
\newblock {\em Proceedings of the Royal Society A: Mathematical, Physical and
  Engineering Sciences 472}, 2190 (2016), 20160127.

\bibitem{Melinand2017}
{\sc Melinand, B.}
\newblock Coriolis effect on water waves.
\newblock {\em ESAIM: Mathematical Modelling and Numerical Analysis 51}, 5
  (2017), 1957--1985.

\bibitem{benoit2017}
{\sc Mesognon-Gireau, B.}
\newblock The {C}auchy problem on large time for a {Boussinesq-Peregrine}
  equation with large topography variations.
\newblock {\em Advances in Differential Equations 22}, 7/8 (2017), 457--504.

\bibitem{Metivier2001}
{\sc M{\'e}tivier, G.}
\newblock Stability of multidimensional shocks.
\newblock {\em Advances in the theory of shock waves\/} (2001), 25--103.

\bibitem{Metivier2012}
{\sc M{\'e}tivier, G.}
\newblock {\em Small Viscosity and Boundary Layer Methods: Theory, Stability
  Analysis, and Applications}.
\newblock Springer Science \& Business Media, 2012.

\bibitem{miles1977hamilton}
{\sc Miles, J.~W.}
\newblock On hamilton's principle for surface waves.
\newblock {\em Journal of Fluid Mechanics 83}, 1 (1977), 153--158.

\bibitem{Miles1979}
{\sc Miles, J.~W.}
\newblock On the {K}orteweg---de {V}ries equation for a gradually varying
  channel.
\newblock {\em Journal of Fluid Mechanics 91}, 1 (1979), 181--190.

\bibitem{Ming2012}
{\sc Ming, M., Saut, J.~C., and Zhang, P.}
\newblock Long-time existence of solutions to {B}oussinesq systems.
\newblock {\em SIAM Journal on Mathematical Analysis 44}, 6 (2012), 4078--4100.

\bibitem{ming2017water}
{\sc Ming, M., and Wang, C.}
\newblock Water waves problem with surface tension in a corner domain i: A
  priori estimates with constrained contact angle.
\newblock {\em arXiv preprint arXiv:1709.00180\/} (2017).

\bibitem{ming2018water}
{\sc Ming, M., and Wang, C.}
\newblock Water waves problem with surface tension in a corner domain ii: the
  local well-posednes.
\newblock {\em arXiv preprint arXiv:1812.09911\/} (2018).

\bibitem{Mouragues2019}
{\sc Mouragues, A., Bonneton, P., Lannes, D., Castelle, B., and Marieu, V.}
\newblock Field data-based evaluation of methods for recovering surface wave
  elevation from pressure measurements.
\newblock {\em Coastal Engineering 150\/} (2019), 147 -- 159.

\bibitem{murakami2015solvability}
{\sc Murakami, Y., and Iguchi, T.}
\newblock Solvability of the initial value problem to a model system for water
  waves.
\newblock {\em Kodai Mathematical Journal 38}, 2 (2015), 470--491.

\bibitem{Musumeci}
{\sc Musumeci, R.~E., Svendsen, I.~A., and Veeramony, J.}
\newblock The flow in the surf zone: a fully nonlinear {B}oussinesq-type of
  approach.
\newblock {\em Coastal Engineering 52\/} (2005), 565--598.

\bibitem{Nalimov}
{\sc Nalimov, V.~I.}
\newblock The {C}auchy-{P}oisson problem (in russian).
\newblock {\em Dyn. Splosh. Sredy 104-210}, 18 (1974).

\bibitem{Naumkin}
{\sc Naumkin, P.~I., and Shishmarev, I.~A.}
\newblock {\em Nonlinear nonlocal equations in the theory of waves}.
\newblock Translation of mathematical monographs. Providence, RI: American
  Mathematical Society, 1994.

\bibitem{nemoto2018}
{\sc Nemoto, R., and Iguchi, T.}
\newblock Solvability of the initial value problem to the {I}sobe--{K}akinuma
  model for water waves.
\newblock {\em Journal of Mathematical Fluid Mechanics 20}, 2 (2018), 631--653.

\bibitem{Nwogu1993}
{\sc Nwogu, O.}
\newblock Alternative form of {B}oussinesq equations for nearshore wave
  propagation.
\newblock {\em J. Wtrwy., Port, Coast., and Oc. Engrg. 119\/} (1993), 616--638.

\bibitem{Oliveras2012}
{\sc Oliveras, K.~L., Vasan, V., Deconinck, B., and Henderson, D.}
\newblock Recovering the water-wave profile from pressure measurements.
\newblock {\em SIAM J. Appl. Math. 72\/} (2012), 897--918.

\bibitem{Peregrine67}
{\sc Peregrine, D.~H.}
\newblock Long waves on a beach.
\newblock {\em J. Fluid Mech. 27}, 4 (1967), 815--827.

\bibitem{RG3}
{\sc Richard, G., and Gavrilyuk, S.}
\newblock Modelling turbulence generation in solitary waves on shear shallow
  water flows.
\newblock {\em Journal of Fluid Mechanics 773\/} (2015), 49--74.

\bibitem{Duran}
{\sc Richard, G.~L., Duran, A., and Fabr{\`e}ges, B.}
\newblock A new model of shoaling and breaking waves. part 2. {R}un-up and
  two-dimensional waves.
\newblock {\em Journal of Fluid Mechanics 867\/} (2019), 146--194.

\bibitem{RG1}
{\sc Richard, G.~L., and Gavrilyuk, S.~L.}
\newblock A new model of roll waves: comparison with {B}rock's experiments.
\newblock {\em J. Fluid Mech. 698\/} (2012), 374--405.

\bibitem{RG2}
{\sc Richard, G.~L., and Gavrilyuk, S.~L.}
\newblock The classical hydraulic jump in a model of shear shallow-water flows.
\newblock {\em J. Fluid Mech. 725\/} (2013), 492--521.

\bibitem{saut2013asymptotic}
{\sc Saut, J.-C.}
\newblock {\em Asymptotic models for surface and internal waves}.
\newblock IMPA, 2013.

\bibitem{Saut2017}
{\sc Saut, J.-C., Wang, C., and Xu, L.}
\newblock The {C}auchy problem on large time for surface-waves-type
  {B}oussinesq systems {II}.
\newblock {\em SIAM Journal on Mathematical Analysis 49}, 4 (2017), 2321--2386.

\bibitem{Saut2012}
{\sc Saut, J.-C., and Xu, L.}
\newblock The {C}auchy problem on large time for surface waves {B}oussinesq
  systems.
\newblock {\em Journal de math{\'e}matiques pures et appliqu{\'e}es 97}, 6
  (2012), 635--662.

\bibitem{schaffer}
{\sc Sch{\"a}ffer, H.~A., Madsen, P.~A., and Deigaard, R.}
\newblock A {B}oussinesq model for waves breaking in shallow water.
\newblock {\em Coastal engineering 20}, 3-4 (1993), 185--202.

\bibitem{Schneider:2000dz}
{\sc Schneider, G., and Wayne, C.~E.}
\newblock The long-wave limit for the water wave problem i. the case of zero
  surface tension.
\newblock {\em Communication on Pure and Applied Mathematics 53}, 12 (2000),
  1475--1535.

\bibitem{Schochet1986}
{\sc Schochet, S.}
\newblock The compressible {E}uler equations in a bounded domain: existence of
  solutions and the incompressible limit.
\newblock {\em Comm. Math. Phys. 104}, 1 (1986), 49--75.

\bibitem{Schonbek}
{\sc Schonbek, M.~E.}
\newblock Existence of solutions for the {B}oussinesq system of equations.
\newblock {\em J. Differential Equations 42}, 3 (1981), 325--352.

\bibitem{Serre}
{\sc Serre, F.}
\newblock Contribution \`a l'\'etude des \'ecoulements permanents et variables
  dans les canaux.
\newblock {\em La Houille Blanche\/} (1953), 830--872.

\bibitem{Shatah}
{\sc Shatah, J., and Zeng, C.}
\newblock A priori estimates for fluid interface problems.
\newblock {\em Comm. Pure Appl. Math. 61}, 6 (2008), 848--876.

\bibitem{Stelling:2003fk}
{\sc Stelling, G., and Zijlema, M.}
\newblock An accurate and efficient finite-difference algorithm for
  non-hydrostatic free-surface flow with application to wave propagation.
\newblock {\em Int. J. Numer. Meth. Fluids 43\/} (2003), 1--23.

\bibitem{SuGardner}
{\sc Su, C.~H., and Gardner, C.~S.}
\newblock Korteweg-de {V}ries equation and generalizations. {III}. {D}erivation
  of the {K}orteweg-de {V}ries equation and {B}urgers equation.
\newblock {\em J. Math. Phys. 10}, 3 (1969), 536--539.

\bibitem{Bardos}
{\sc Sulem, C., Sulem, P.-L., Bardos, C., and Frisch, U.}
\newblock Finite time analyticity for the two- and three-dimensional
  {K}elvin-{H}elmholtz instability.
\newblock {\em Comm. Math. Phys. 80}, 4 (1981), 485--516.

\bibitem{svendsen}
{\sc Svendsen, I., Yu, K., and Veeramony, J.}
\newblock A {B}oussinesq breaking wave model with vorticity.
\newblock In {\em Coastal Engineering 1996}. 1997, pp.~1192--1204.

\bibitem{Taylor3}
{\sc Taylor, M.~E.}
\newblock {\em Partial differential equations. {III}}, vol.~117 of {\em Applied
  Mathematical Sciences}.
\newblock Springer-Verlag, New York, 1997.
\newblock Nonlinear equations, Corrected reprint of the 1996 original.

\bibitem{Tissier}
{\sc Tissier, M., Bonneton, P., Marche, F., Chazel, F., and Lannes, D.}
\newblock A new approach to handle wave breaking in fully non-linear boussinesq
  models.
\newblock {\em Coastal Engineering 67\/} (2012), 54--66.

\bibitem{tonelli}
{\sc Tonelli, M., and Petti, M.}
\newblock Simulation of wave breaking over complex bathymetries by a boussinesq
  model.
\newblock {\em Journal of Hydraulic Research 49}, 4 (2011), 473--486.

\bibitem{VG1993}
{\sc van Groesen, E., and Pudjaprasetya, S.~R.}
\newblock Uni-directional waves over slowly varying bottom. {I}. {D}erivation
  of a {K}d{V}-type of equation.
\newblock {\em Wave Motion 18}, 4 (1993), 345--370.

\bibitem{Vasan2017}
{\sc Vasan, V., and Oliveras, K.~L.}
\newblock Water-wave profiles from pressure measurements: Extensions.
\newblock {\em Applied Mathematics Letters 68\/} (2017), 175 -- 180.

\bibitem{Wei}
{\sc Wei, G., Kirby, J.~T., Grilli, S.~T., and Subramanya, R.}
\newblock A fully nonlinear {B}oussinesq model for surface waves. {I}. {H}ighly
  nonlinear unsteady waves.
\newblock {\em J. Fluid Mech. 294\/} (1995), 71--92.

\bibitem{WKS}
{\sc Wei, G., Kirby, J.~T., and Sinha, A.}
\newblock Generation of waves in {B}oussinesq models using a source function
  method.
\newblock {\em Coastal Engineering 36\/} (1999), 271--299.

\bibitem{whitham1962}
{\sc Whitham, G.}
\newblock Mass, momentum and energy flux in water waves.
\newblock {\em Journal of Fluid Mechanics 12}, 1 (1962), 135--147.

\bibitem{Whitham}
{\sc Whitham, G.~B.}
\newblock Variational methods and applications to water waves.
\newblock {\em Proc. Roy. Soc. London Ser. A 299\/} (1967).

\bibitem{Whitham_book}
{\sc Whitham, G.~B.}
\newblock {\em Linear and nonlinear waves}.
\newblock Pure and Applied Mathematics (New York). John Wiley \& Sons Inc., New
  York, 1999.
\newblock Reprint of the 1974 original, A Wiley-Interscience Publication.

\bibitem{Wu1}
{\sc Wu, S.}
\newblock Well-posedness in {S}obolev spaces of the full water wave problem in
  2-d.
\newblock {\em Invent. Math. 130}, 1 (1997), 39--72.

\bibitem{Wu2}
{\sc Wu, S.}
\newblock Well-posedness in {S}obolev spaces of the full water wave problem in
  3-d.
\newblock {\em J. Amer. Math. Soc. 12}, 2 (1999), 445--495.

\bibitem{Xue}
{\sc Xue, R.}
\newblock The initial--boundary value problem for the ``good'' {B}oussinesq
  equation on the bounded domain.
\newblock {\em Journal of Mathematical Analysis and Applications 343}, 2
  (2008), 975 -- 995.

\bibitem{Zakharov1968}
{\sc Zakharov, V.~E.}
\newblock Stability of periodic waves of finite amplitude on the surface of a
  deep fluid.
\newblock {\em Journal of Applied Mechanics and Technical Physics 9\/} (1968),
  190--194.

\bibitem{zeitlin2007}
{\sc Zeitlin, V.}
\newblock {\em Nonlinear dynamics of rotating shallow water: Methods and
  advances}, vol.~2.
\newblock Elsevier, 2007.

\end{thebibliography}

\end{document}